\documentclass[12pt,reqno]{amsart}

\usepackage[T1]{fontenc}
\usepackage{microtype}
\usepackage[margin=1in]{geometry}

\usepackage{amsmath,amssymb,amsthm,mathtools}
\usepackage{graphicx}
\usepackage{tikz}
\usepackage{tikz-cd}
\usepackage{enumitem}
\usepackage{float}
\usepackage{comment}
\usepackage{url}
\usepackage{cite}

\usepackage[bookmarksopen,bookmarksdepth=3]{hyperref}
\usepackage[capitalize,noabbrev]{cleveref}

\makeatletter
\let\cref@old@refstepcounter@optarg\refstepcounter@optarg
\def\refstepcounter@optarg[#1]#2{%
  \cref@old@refstepcounter@optarg[#1]{#2}%
  \global\let\cref@lastlabeltype\@tempa
  \xdef\cref@lastlabelcounter{#2}%
}
\@ifundefined{firstaid@cref@updatelabeldata}{}{%
  \renewcommand\firstaid@cref@updatelabeldata[1]{%
    \cref@constructprefix{#1}{\cref@result}%
    \edef\@tempb{#1}%
    \@ifundefined{cref@lastlabelcounter}{\let\@tempc\@empty}{\let\@tempc\cref@lastlabelcounter}%
    \ifx\@tempc\@tempb
      \let\@tempa\cref@lastlabeltype
    \else
      \@ifundefined{cref@#1@alias}%
        {\def\@tempa{#1}}%
        {\def\@tempa{\csname cref@#1@alias\endcsname}}%
    \fi
    \protected@xdef\cref@gcurrentlabel@temp{%
      [\@tempa][\arabic{#1}][\cref@result]%
      \csname p@#1\endcsname\csname the#1\endcsname}%
    \aftergroup\firstaid@cref@smugglelabel}%
}
\makeatother

\usepackage{cjhebrew}
\newcommand{\nun}{\text{\upshape\cjRL{n|}}}
\allowdisplaybreaks[1]
\newtheorem{prop}{Proposition}[section]
\newtheorem{thm}[prop]{Theorem}
\newtheorem{lm}[prop]{Lemma}
\newtheorem{cor}[prop]{Corollary}

\theoremstyle{definition}
\newtheorem{dfn}[prop]{Definition}

\theoremstyle{remark}
\newtheorem{rem}[prop]{Remark}
\newtheorem{ex}[prop]{Example}
\newtheorem{notn}[prop]{Notation}
\newtheorem{assum}[prop]{Assumption}

\crefname{prop}{Proposition}{Propositions}
\crefname{thm}{Theorem}{Theorems}
\crefname{lm}{Lemma}{Lemmas}
\crefname{cor}{Corollary}{Corollaries}
\crefname{conj}{Conjecture}{Conjectures}
\crefname{dfn}{Definition}{Definitions}
\crefname{rem}{Remark}{Remarks}
\crefname{ex}{Example}{Examples}
\crefname{notn}{Notation}{Notations}
\crefname{assum}{Assumption}{Assumptions}

\newlist{axioms}{enumerate}{1}
\crefname{axiomsi}{Axiom}{Axioms}
\Crefname{axiomsi}{Axiom}{Axioms}

\DeclareMathOperator{\Cone}{Cone}

\DeclareMathOperator{\Ker}{Ker}
\DeclareMathOperator{\Coker}{Coker}
\DeclareMathOperator{\End}{End}
\DeclareMathOperator{\Id}{Id}

\DeclareMathOperator{\tr}{tr}
\DeclareMathOperator{\GW}{GW}
\DeclareMathOperator{\ogw}{OGW}

\newcommand{\R}{\mathbb{R}}
\newcommand{\C}{\mathbb{C}}
\newcommand{\Z}{\mathbb{Z}}

\newcommand{\F}{\mathbb{F}}
\newcommand{\kk}{\Bbbk}

\newcommand{\CP}{\mathbb{C}P}

\newcommand{\M}{\mathcal{M}}
\newcommand{\J}{\mathcal{J}}
\newcommand{\mI}{\mathcal{I}}

\newcommand{\q}{\mathfrak{q}}
\newcommand{\m}{\mathfrak{m}}

\newcommand{\qop}[3]{\mathfrak{q}_{#1,#2}^{#3}}

\title{Open WDVV equations for non-commutative families of bounding cochains}
\author[E. Kosloff]{Elad Kosloff}
\address{Institute of Mathematics, The Hebrew University of Jerusalem}
\email{eladkosloff@gmail.com}
\date{\today}

\begin{document}

\begin{abstract}
We derive the noncommutative open WDVV equations for the open Gromov--Witten superpotential introduced in joint work of the author and Solomon, extending the Solomon--Tukachinsky equations to noncommutative differential graded algebras. The resulting recursive formulas facilitate the computation of open Gromov--Witten invariants with multiple boundary point constraints for even-dimensional Lagrangians.
We apply the open WDVV equations to compute all open Gromov--Witten invariants for the real loci of even-dimensional complex quadrics of dimension greater than two. These computations show that open Gromov--Witten invariants with multiple boundary point constraints can be nontrivial in arbitrary dimensions.

We also construct a character for a natural subalgebra of big quantum cohomology associated to a Lagrangian endowed with a point-like bounding cochain. The character is in the spirit of the Kodaira--Spencer map of Fukaya--Oh--Ohta--Ono. Applying the character construction to the even-dimensional quadric illustrates how point-like bounding cochains refine the information about closed quantum cohomology obtained from weak bounding cochains of Floer theory.
Open Gromov--Witten invariants with three boundary point constraints determine the two extensions of the character from the subalgebra to big quantum cohomology.
\end{abstract}

\maketitle

\vspace{-2.5em}

\tableofcontents
\section{Introduction}
\label{sec:introduction}

\subsection{Overview}\label{subsec:intro-overview}

In their work on the enumeration of rational plane curves,
Kontsevich and Manin~\cite{kontsevich1994gromov} illustrate how curve counts
can be determined by algebraic structures. For a closed symplectic manifold
$(X, \omega)$, the genus-zero Gromov--Witten potential $\Phi$ encodes counts of rational
curves in $X$, and its third derivatives define the quantum product.
Associativity of that product can be recast as a system of formal partial
differential equations for $\Phi$, the WDVV equations. For $X=\mathbb{CP}^2$,
these equations yield a recursion determining the numbers of degree-$d$
rational curves through $3d-1$ generic points from the fact that there is a
unique line through two distinct points.

Open Gromov--Witten theory seeks analogous counts of pseudoholomorphic disks
with boundary on a Lagrangian submanifold
$L\subset X$~\cite{liu2002moduli,solomon2006intersection,cho2008counting,fukaya2011counting,georgieva2016open,solomon2016point}.
Solomon and Tukachinsky~\cite{solomon2016point} define such counts as the
coefficients of a superpotential $\Omega$, in which a formal variable $s$
records the boundary point constraints. Together with the closed potential
$\Phi$, it satisfies formal partial differential equations, the open WDVV
equations~\cite{solomon2019relative}. Analogously to the closed case, these
equations yield a recursion determining the counts for
$(\mathbb{CP}^n,\mathbb{RP}^n)$ with $n$ odd from the fact that there are
exactly two oriented real lines through two distinct points of
$\mathbb{RP}^n$.

The main difficulty in defining open invariants is disk bubbling. The
relevant moduli spaces have codimension-one boundary strata.
Thus, disk counts change as the auxiliary data of the construction vary in
one-parameter families. Similar bubbling phenomena appear in Floer theory,
where they obstruct the Floer differential from squaring to zero.
Fukaya~\cite{fukaya1993morse} showed that bubbling gives rise to $A_\infty$
relations between operators $\mathfrak q_k$ that count disks with $k$
boundary inputs. In their study of Floer theory,
Fukaya--Oh--Ohta--Ono~\cite{fukaya2009lagrangian1,fukaya2009lagrangian2}
introduced bounding cochains, solutions $b$ of
$\sum_k\mathfrak q_k(b,\ldots,b)=c\cdot1_L$ in the resulting $A_\infty$
algebra. Solomon and Tukachinsky~\cite{solomon2016point} construct $\Omega$
from \emph{point-like bounding cochains}, families $b$ of bounding cochains
parameterized by the formal variable $s$ of degree $1-\dim L$, and prove
that these exist and are unique up to gauge equivalence where their
construction applies. The family $b$ plays two roles: it serves as a
boundary point constraint, and it cancels bubbling contributions.

If $\dim L$ is even, then $s$ has odd degree, and graded
commutativity forces $s^2=0$. Hence, $\Omega$ cannot record counts with more
than one boundary point constraint, such as disks through three points of
$L$. Point-like bounding cochains over noncommutative coefficients are constructed in joint work with
Solomon~\cite{kosloff2025pointlikenoncommutativef}. In this paper, we establish
the open WDVV equations for $A_\infty$ algebras over noncommutative
differential graded algebras. For the even-dimensional quadric, the
one-boundary-point count identifies the Fukaya-category summand containing
the real locus, while counts with several boundary point constraints
reveal finer structure in the corresponding summand of quantum cohomology.

We also define a map on a subalgebra of big quantum cohomology
using open Gromov--Witten invariants, and use the open WDVV equations
to show that it is a character.
Write $\Omega_k(t)$ for the generating function of open Gromov--Witten
invariants with $k$ boundary point constraints, where the bulk coordinates
$t$ record the interior constraints.
For a bulk class $u$, the derivative $\partial_u\Omega_1$ counts disks
with one more interior constraint, through a cycle Poincar\'e dual
to $u$.
Under the hypotheses of Theorem~\ref{thm:intro-boundary-character}, the character is
\[
 \mathcal L(u)=\partial_u\Omega_1(t).
\]
Heuristically, $\mathcal L$ is an evaluation at the critical point of
the Kodaira--Spencer map of Fukaya--Oh--Ohta--Ono~\cite{fukaya2016mirror,cholee2024kodaira}
which carries quantum
cohomology into the Jacobian ring of a Landau--Ginzburg potential. The
point-like bounding cochain plays the role of the critical point.
The Euler element $E$ of Dubrovin~\cite{dubrovin1996geometry}
is $c_1(TX)$ corrected by degree-weighted bulk coordinates, and
$\mathcal L(E)$ plays the role of the critical value.
At zero bulk, $E=c_1(TX)$, and $\mathcal L(c_1(TX))$ is the
Novikov-weighted count of Maslov-two disks through a point of $L$.

We apply the open WDVV theory with noncommutative coefficients developed in
this paper to the real quadric. Let $X\subset\CP^{n+1}$ be the quadric
\[z_0^2+\cdots+z_n^2-z_{n+1}^2=0\]
with $n>2$ even, and let $X(\R)\simeq S^n$ denote its real locus.
As in the odd-dimensional case~\cite{hugtenburg2024examples}, the open WDVV
equations give a recursion that determines all open Gromov--Witten invariants
of $(X,X(\R))$ from the genus-zero closed invariants of $X$, computed
in~\cite{beauville1995quantum}, and one open disk count in degree one.
Moreover, given two independent open counts in degree one, the recursion
recovers all genus-zero closed Gromov--Witten invariants of $X$.

The character these invariants produce agrees with what is already known. There
are no Maslov-two disks, so $\mathcal L(c_1(TX))=0$ at zero bulk. Zero is the
eigenvalue of quantum multiplication by $c_1(TX)$ on the summand of $QH^*(X)$
where it acts nilpotently, and $L$ split-generates the corresponding summand
of the Fukaya category~\cite{smith2012pencils}.

The boundary character has two extensions to the full big quantum cohomology
algebra. They agree on the Euler
element, and at zero bulk this is their common value on $c_1(TX)$, so an
eigenvalue does not tell them apart. They differ on the Poincaré dual of the
Lagrangian, $p=\mathrm{PD}[X(\R)]$, where they take opposite values, determined by
\[
 \mathcal L(p\star p)=2\Omega_3(t).
\]
The invariants on the right count disks with three boundary point constraints.
These are the invariants a graded-commutative coefficient algebra cannot record.

Defining quantum-cohomology characters may require extending the coefficient
ring. The toric mirror picture inspired by the Strominger--Yau--Zaslow (SYZ)
conjecture~\cite{strominger1996mirror} interprets these characters as evaluations
at critical points of a disk potential~\cite{auroux2007mirror}, giving a geometric
meaning to the added coefficients.
For $\mathbb{CP}^2$, the small quantum relation $h^{\star 3}=q$ yields three
characters after adjoining $q^{1/3}$ and the cube roots of unity to
$\mathbb{Q}[[q]]$.
The Clifford-torus picture gives a geometric interpretation. Roots of unity
occur as critical holonomies, while disk areas account for fractional Novikov
powers~\cite{auroux2007mirror,cho2004holomorphic,cho2006floer}. The holonomies provide a geometric
interpretation of the complex phase that is less apparent from operations
defined by counts of spheres with no obvious preferred direction.

The geometric role of coefficient extensions is also illustrated by
Haney's work on open Gromov--Witten invariants of
Lagrangians in the quintic, where algebraic coefficients arise from
local-system holonomies~\cite{haney2024open}.
To make this role explicit in our setting, we formulate the construction
axiomatically, separating the structure maps and
their identities from the coefficient extensions used in applications, and
adjoin coefficients in stages. For the quadric, the square root is
adjoined only in Section~\ref{subsec:even-quadric-boundary-character}, after the relevant
invariants have been computed.

Our character construction does not require a prior computation of a Floer
group. The open WDVV equations give the character of Theorem~\ref{thm:open-character}
when $\Omega_2(t)$ is identically zero after restricting to the bulk
variables used in the character construction.
Floer-theoretic conclusions follow from the same point-like family.
The condition $\Omega_2(t)=0$ implies
$HF\big(L,\mathfrak{q}_1^{\gamma_W,b|_{s=0}};R^0_W\big)\neq 0$,
while $\Omega_2(t)$ always
annihilates that group. Thus, in examples where
$\Omega_2(0)\ne 0$, such as the Chiang Lagrangian or real projective
space in $\mathbb{CP}^3$, Floer cohomology is trivial over the Novikov field at
zero bulk. In even dimensions, the vanishing is automatic once a point-like
bounding cochain exists, so existence already implies Floer nonvanishing. In odd dimensions, the same
identities suggest interesting applications, discussed for the Chiang Lagrangian
in Section~\ref{subsec:arithmetic-theory-open-gromov}.

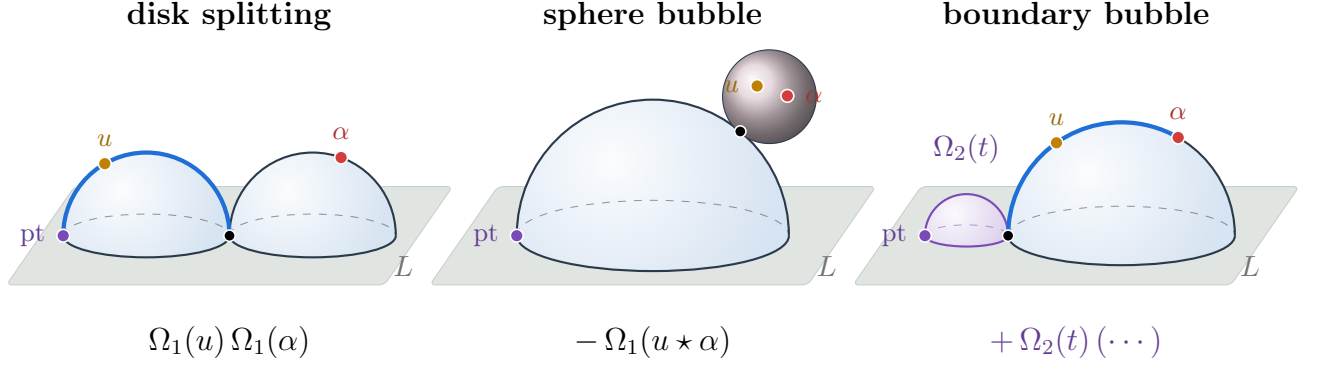
\begin{figure}[tbp]
\centering
\definecolor{domefill}{RGB}{215,228,242}
\definecolor{domeedge}{RGB}{43,58,75}
\definecolor{geoblue}{RGB}{31,111,214}
\definecolor{mkU}{RGB}{191,129,5}
\definecolor{mkA}{RGB}{210,60,60}
\definecolor{mkP}{RGB}{122,75,183}
\definecolor{slabfill}{RGB}{224,229,225}
\definecolor{slabedge}{RGB}{200,207,202}
\definecolor{spherecol}{RGB}{184,146,158}
\newcommand{\dome}[4]{%
 \begin{scope}
  \path[clip] (#1-#2,0) arc[start angle=180,end angle=0,radius=#2]
    arc[start angle=0,end angle=-180,x radius=#2,y radius=0.26*#2] -- cycle;
  \fill[#3] (#1-#2,-#2) rectangle (#1+#2,#2);
  \shade[inner color=white,outer color=#3,opacity=0.9]
    (#1-0.3*#2,0.55*#2) circle (1.25*#2);
 \end{scope}
 \draw[#4,dashed,opacity=0.5] (#1-#2,0) arc[start angle=180,end angle=0,x radius=#2,y radius=0.26*#2];
 \draw[#4,thick] (#1-#2,0) arc[start angle=180,end angle=0,radius=#2];
 \draw[#4,thick] (#1+#2,0) arc[start angle=0,end angle=-180,x radius=#2,y radius=0.26*#2];}
\newcommand{\slab}[1]{%
 \filldraw[fill=slabfill,draw=slabedge,rounded corners=2pt]
  (-#1+0.45,0.65) -- (#1+0.45,0.65) -- (#1-0.45,-0.65) -- (-#1-0.45,-0.65) -- cycle;
 \node[font=\small\itshape,text=black!55] at (#1-0.2,-0.4) {$L$};}
\newcommand{\mk}[3]{\node[circle,fill=#1,draw=white,line width=0.6pt,inner sep=1.7pt,label={[font=\footnotesize,text=#1!80!black]#2}] at (#3) {};}
\newcommand{\nd}[1]{\node[circle,fill=black,draw=white,line width=0.6pt,inner sep=1.5pt] at (#1) {};}
\begin{tikzpicture}[geo/.style={geoblue,line width=1.6pt,line cap=round}]
 \begin{scope}[xshift=-5.6cm]
  \slab{2.5}
  \dome{-1.1}{1.1}{domefill}{domeedge}
  \dome{1.1}{1.1}{domefill}{domeedge}
  \draw[geo] (-2.2,0) arc[start angle=180,end angle=0,radius=1.1];
  \mk{mkP}{left:pt}{-2.2,0}
  \mk{mkU}{above:$u$}{-1.65,0.953}
  \mk{mkA}{above:$\alpha$}{1.476,1.034}
  \nd{0,0}
  \node at (0,-1.4) {$\Omega_1(u)\,\Omega_1(\alpha)$};
  \node at (0,2.9) {\textbf{disk splitting}};
 \end{scope}
 \begin{scope}
  \slab{2.5}
  \dome{0}{1.8}{domefill}{domeedge}
  \shade[ball color=spherecol!45!white] (1.543,1.839) circle (0.62);
  \draw[domeedge] (1.543,1.839) circle (0.62);
  \mk{mkP}{left:pt}{-1.8,0}
  \nd{1.157,1.379}
  \mk{mkU}{left:$u$}{1.38,1.98}
  \mk{mkA}{right:$\alpha$}{1.78,1.85}
  \node at (0,-1.4) {$-\,\Omega_1(u\star\alpha)$};
  \node at (0,2.9) {\textbf{sphere bubble}};
 \end{scope}
 \begin{scope}[xshift=5.6cm]
  \slab{2.5}
  \dome{0.6}{1.5}{domefill}{domeedge}
  \dome{-1.45}{0.55}{mkP!25}{mkP}
  \draw[geo] (-0.9,0) arc[start angle=180,end angle=60,radius=1.5];
  \mk{mkP}{left:pt}{-2.0,0}
  \nd{-0.9,0}
  \mk{mkU}{above:$u$}{-0.26,1.229}
  \mk{mkA}{above:$\alpha$}{1.35,1.299}
  \node[text=mkP!80!black,font=\small] at (-1.45,1.15) {$\Omega_2(t)$};
  \node[text=mkP!80!black] at (0,-1.4) {$+\,\Omega_2(t)\,(\cdots)$};
  \node at (0,2.9) {\textbf{boundary bubble}};
 \end{scope}
\end{tikzpicture}
\caption{The proof of the character theorem can be understood quite faithfully
in geometric terms, even though the algebraic machinery hides this. One analyzes
the boundary strata of the geodesic-constrained moduli space of Solomon and
Tukachinsky~\cite{solomon2016point}, here the moduli space of disks with one
boundary point constraint and two interior markings, with the geodesic constraint
running from the boundary point through $u$ to $\alpha$. The figure shows its three kinds of boundary
strata, and below each is its contribution to the relation that gives the
character theorem.}
\label{fig:intro-character-strata}
\end{figure}

The quantum-cohomological information obtained here has parallels in
the literature. For a Lagrangian $L$ with nontrivial Floer cohomology over
the Novikov field, the open--closed map has nonzero image in quantum
cohomology~\cite[Proposition 18.1]{abouzaid2026quantum}.
If an idempotent $e$ in quantum cohomology has nonzero image under the
closed--open map $QH^*(X)\to HF^*(L)$, then $L$ is heavy with respect to
$e$~\cite[Theorem 1.6]{fukaya2019spectral}.
For Floer-nontrivial Lagrangians,
potential and curvature values also satisfy spectral constraints
\cite{auroux2007mirror,sheridan2016fano,yuan2025family,castronovo2024curved}.
The phenomenon behind the two quadric characters also appears categorically. The
decomposition of the Fukaya category by eigenvalues of quantum multiplication by
$c_1$ can be coarser than its decomposition by quantum
idempotents~\cite{evanslekili2019generating}. For the even quadric, the Floer
cohomology of the spherical real locus is semisimple over
$\mathbb{C}$~\cite{smith2012pencils}.
Related constructions include Fukaya-categorical decompositions of
quantum cohomology~\cite{sanda2021computation,smith2025summands}
and a relative cyclic open--closed map to
$S^1$-equivariant relative quantum homology
\cite{hugtenburg2024open,hirschihugtenburg2026open}.

Our technical starting point is the relative quantum cohomology of
Solomon and Tukachinsky~\cite{solomon2019relative}, whose cone complex,
relative potential and open WDVV equations we extend to noncommutative
coefficients.

The paper is organized as follows. Section~\ref{sec:open-closed-framework} recasts the
geometric setup in axiomatic form. Section~\ref{sec:structure-maps} states the axioms
for the structure maps and shows that the differential-form operations of
Solomon--Tukachinsky satisfy them. Section~\ref{sec:extension-of-scalars} constructs the
coefficient rings and the notion of a bounding pair.
Section~\ref{sec:tensor-relative-potentials} lets derivations of the coefficient rings act on
the cone complex, builds the tensor and relative potentials on it, and defines the enhanced
superpotential.
Section~\ref{sec:noncommutative-relative-quantum} derives the noncommutative open WDVV equations,
constructs the boundary character, computes its value on the Euler element, and draws the
consequences for Floer cohomology. Section~\ref{sec:even-quadric-background} treats the even
quadric. Appendix~\ref{app:scalar-extension} treats filtrations and scalar extension,
Appendix~\ref{subsec:bg-deformations} covers deformations by interior and boundary
insertions, and Appendix~\ref{app:tensor-potential-lemmas} proves degree and gauge
invariance of the tensor potential.
Appendix~\ref{app:standard-arguments} explains why Solomon--Tukachinsky's
proofs of closedness of the relative potential and flatness of the relative
quantum connection remain valid in our setting.

\subsection{Geometric setting and notation}
\label{subsec:intro-geometric-background}
\subsubsection{Differential forms and curve classes}
\label{subsubsec:intro-curve-classes}
Let $(X,\omega)$ be a closed symplectic manifold of dimension $2n$,
where $n>1$, and let $L\subset X$ be a compact, connected, oriented
Lagrangian with a fixed relative spin structure $\mathfrak{s}$.  Fix an $\omega$-tame
almost complex structure $J$.  Write $i:L\hookrightarrow X$ for the
inclusion and $A^*(X)$ and $A^*(L)$ for differential forms
with real coefficients.  The pairings are
\[
 \langle\eta,\xi\rangle_X=(-1)^{|\xi|}\int_X\eta\wedge\xi,
 \qquad
 \langle a,b\rangle_L=(-1)^{|b|}\int_La\wedge b.
\]
Here $|\eta|$ denotes the degree of a homogeneous element.  The
pairings descend to cohomology.

Symplectic area and Maslov index define homomorphisms
\[
 \omega:H_2(X,L;\Z)\longrightarrow\R,\qquad
 \mu_L:H_2(X,L;\Z)\longrightarrow\Z.
\]
Following \cite{solomon2019relative}, fix a possibly trivial subgroup
$S_L\subset\ker\omega\cap\ker\mu_L$ and set
\[
 \Pi:=H_2(X,L;\Z)/S_L,\qquad
 \varpi:H_2(X;\Z)\longrightarrow\Pi,
\]
where $\varpi$ denotes throughout the natural map to relative homology
followed by the quotient.  Both area and Maslov index descend to $\Pi$,
retaining the notation $\omega$ and $\mu_L$.  Since $L$ is oriented,
$\mu_L$ takes even values.  Write
\[
 \beta_0:=0\in\Pi,\qquad A_0:=0\in H_2(X;\Z).
\]

\subsubsection{Coefficients}
\label{subsubsec:intro-coefficients}
Let $T^\beta$ be a formal symbol for $\beta\in\Pi$, with
$T^{\beta_1}T^{\beta_2}=T^{\beta_1+\beta_2}$ and
$|T^\beta|=\mu_L(\beta)$.  Define the Novikov rings
\[
\begin{aligned}
 \Lambda&:=\left\{\sum_{i=0}^\infty a_iT^{\beta_i}\ \middle|\
 a_i\in\R,\ \beta_i\in\Pi,\ \omega(\beta_i)\ge0,\
 \omega(\beta_i)=0\Longrightarrow\beta_i=\beta_0,\
 \lim_{i\to\infty}\omega(\beta_i)=\infty\right\},\\
 \Lambda_c&:=\left\{\sum_{i=0}^\infty a_iT^{\beta_i}\in\Lambda\ \middle|\
 \beta_i\in\operatorname{im}\varpi\right\}.
\end{aligned}
\]
Define the positive-area ideals
\[
 \Lambda_+:=\left\{\sum_{i=0}^\infty a_iT^{\beta_i}\in\Lambda\ \middle|\
 \omega(\beta_i)>0\ \forall i\right\},\qquad
 \Lambda_{c,+}:=\Lambda_c\cap\Lambda_+.
\]

For a finite-dimensional graded vector space $V$ with homogeneous
basis $v_1,\ldots,v_N$, write
\[
 \R[[V]]=\R[[t_1,\ldots,t_N]],\qquad |t_j|=2-|v_j|,
 \qquad Q_V=\Lambda_c\otimes\R[[V]].
\]
The variables $t_j$ are graded commutative.  Let $W$ be another
finite-dimensional graded real vector space.  Put
\[
 R_W^0=\Lambda\otimes\R[[W]],\qquad
 R_W=R_W^0\otimes\R\langle s\rangle,
 \qquad |s|=1-n.
\]
The factor $\R\langle s\rangle$ is the free associative algebra on
$s$; no relation $s^2=0$ is imposed when $s$ is odd.

Let $w_1,\ldots,w_N$ be a homogeneous basis of $W$, with coordinates
$t_i$ as above.  Let $u=(u_W,x_u)$, where
$u_W=\sum_i f_iw_i\in W\otimes Q_W$ and $x_u\in R_W$.
Suppose $f_i$ and $x_u$ are homogeneous of degrees $|t_i|+r$ and
$|s|+r$, respectively, for some $r\in\Z$.  Define the degree-$r$
derivation $\partial_u$ of $R_W$ by
\[
 \partial_u(t_i)=f_i,\qquad \partial_u(s)=x_u,
 \qquad \partial_u(T^\beta)=0.
\]
For $v\in W\otimes Q_W$, write $\partial_v=\partial_{(v,0)}$.  The same
formulas define $\partial_v$ on $Q_V$ for $v\in V\otimes Q_V$.
Write $\partial_s:=\partial_{(0,1)}$.

Let $\mathfrak m_V=(t_1,\ldots,t_N)\subset\R[[V]]$ be the maximal ideal,
and define $\mathfrak m_W$ similarly.  Set
\begin{equation}\label{eq:positive-deformation-ideals}
 \mI_{Q_V}=\Lambda_{c,+}Q_V+\mathfrak m_VQ_V,
 \qquad
 \mI_{R_W}=\Lambda_+R_W+\mathfrak m_WR_W.
\end{equation}
Tensor products are completed as specified in
Appendix~\ref{app:scalar-extension}.

Write $\mathcal Z(R_W)$ for the graded center and define the graded
$\mathcal Z(R_W)$-module
\[
 HH_0(R_W)=R_W/[R_W,R_W]_{\mathrm{gr}},\qquad
 [a,b]_{\mathrm{gr}}=ab-(-1)^{|a||b|}ba.
\]
Write $[a]$ for the class of $a$, with action $z\cdot[a]=[za]$.
When $s$ is odd, $[s^2]=0$ but $s^2\cdot[s]=[s^3]$.

\subsubsection{Moduli spaces and operations}
\label{subsubsec:intro-operations}
The geometric input for open WDVV consists of disk, sphere, and
geodesic operations.  We use the differential-form construction
of~\cite{solomon2016differential,solomon2019relative}, in which the
disk operations determine a family of $A_\infty$ structures on
$A^*(L)\otimes\Lambda$ in the framework
of~\cite{fukaya2010cyclic,fukaya2009lagrangian1}.
For $\beta\in\Pi$, write
$\M_{k+1,l}(\beta)=\M_{k+1,l}(\beta;J)$ for the moduli space of
stable genus-zero $J$-holomorphic maps to $(X,L)$ in class $\beta$,
with one boundary component, $k+1$ cyclically ordered boundary
marked points, and $l$ interior marked points.  The evaluation maps are
\[
 \operatorname{ev}^{b,\beta}_j:\M_{k+1,l}(\beta)\longrightarrow L
 \quad(0\le j\le k),\qquad
 \operatorname{ev}^{i,\beta}_j:\M_{k+1,l}(\beta)\longrightarrow X
 \quad(1\le j\le l).
\]
Similarly, let $\M_{l+1}(A)$ be the moduli space of genus-zero stable
maps to $X$ in class $A\in H_2(X;\Z)$ with $l+1$ marked points, and
write $\operatorname{ev}^A_j:\M_{l+1}(A)\to X$ for its evaluation maps.

To define the operations by fiber integration, we require regularity
of all these stable maps, smooth orbifold structures on the moduli
spaces (with corners in the disk case), and proper submersions
$\operatorname{ev}^{b,\beta}_0$ and $\operatorname{ev}^A_0$.
The sphere spaces carry their usual orientations; the disk spaces
are oriented by the fixed relative spin structure $\mathfrak{s}$
\cite[Chapter~8]{fukaya2009lagrangian2}.
We impose the corresponding conditions on the geodesic and
one-parameter moduli spaces used in~\cite{solomon2019relative}.
Fix a connected component $\J$ of the space of $\omega$-tame
almost complex structures obeying these conditions; all choices
of $J$ are made within this component.

Integration along the fibers defines the disk operations
\[
 \q_{k,l}: A^*(L;\Lambda)^{\otimes k}\otimes A^*(X)^{\otimes l}
 \longrightarrow A^*(L;\Lambda),
\]
given by
\[
\begin{aligned}
 &\q_{k,l}(a_1,\ldots,a_k;\eta_1,\ldots,\eta_l)\\
 &\quad=\delta_{k,1}\delta_{l,0}\,d a_1
 +(-1)^{\epsilon(a)}\sum_{\beta\in\Pi}T^\beta
 (\operatorname{ev}^{b,\beta}_0)_*
 \left(
 \bigwedge_{j=1}^l(\operatorname{ev}^{i,\beta}_j)^*\eta_j
 \wedge
 \bigwedge_{j=1}^k(\operatorname{ev}^{b,\beta}_j)^*a_j
 \right),
\end{aligned}
\]
where $\epsilon(a)=1+\sum_{j=1}^kj(|a_j|+1)$.
The sphere operations are
\[
 \q_{\emptyset,l}(\eta_1,\ldots,\eta_l)
 =\sum_A(-1)^{w_{\mathfrak{s}}(A)}T^{\varpi(A)}(\operatorname{ev}^A_0)_*
 \left(\bigwedge_{j=1}^l(\operatorname{ev}^A_j)^*\eta_j\right),
\]
where $w_{\mathfrak{s}}\in H^2(X;\Z/2)$ is the class determined by the relative
spin structure $\mathfrak{s}$, with $i^*w_{\mathfrak{s}}=w_2(TL)$.  Finally,
$\q_{-1,l}=\sum_\beta T^\beta\q_{-1,l}^\beta$ is the scalar disk
operation, obtained by integrating the interior constraints over
the moduli space with no boundary marked points. 
As in~\cite{kosloff2025pointlikenoncommutativef}, we extend the Fukaya
operations to non-commutative coefficients.  The extensions of all
the operations used here to the completed coefficient modules are
specified in Appendix~\ref{app:scalar-extension}.

The even-dimensional quadric with real locus $L\simeq S^n$ studied
in~\cref{sec:even-quadric-background} satisfies these regularity
conditions with its standard symplectic and complex structures,
by the homogeneous-space construction
of~\cite[Example~1.5 and Remark~1.6]{solomon2016differential}.
The geodesic conditions follow from the same group actions,
as explained in~\cite[Section~1.3.12]{solomon2019relative}.
Beyond this regular setting, the analogous construction uses
compatible virtual fundamental chains in place of ordinary fiber
integration.  Kuranishi constructions
are developed in
\cite{abouzaid2021complex,abouzaid2024gromov,fukaya2010cyclic,
fukaya2011counting,fukaya2009lagrangian1,fukaya2009lagrangian2,
fukaya2010toric1,fukaya2011toric2,fukaya2016mirror,
fukaya2017antisymplectic,fukaya2019spectral,fukaya2020kuranishi,
fukaya2024construction,fukayaono1999arnold,hirschihugtenburg2025field};
with such compatible chains, the constructions extend beyond the
regular setting, subject to the remaining hypotheses, and $\J$
may be the full space of $\omega$-tame almost complex structures.
Hirschi and Hugtenburg~\cite{hirschihugtenburg2026open} apply their
theory to construct differential-form operators
satisfying the structure-map axioms in~\cref{dfn:structure-maps}.
The polyfold methods
of~\cite{hofer2007fredholm1,hofer2010integration,hofer2009fredholm3,
hofer2009fredholm2,liwehrheim2014ainfinity} offer another prospective
approach to this extension.

\subsubsection{Bounding pairs}
\label{subsubsec:intro-bounding-pairs}
Define the trace map and its kernel by
\[
 \mathfrak i_\R:A^*(X)\longrightarrow\R[-n],\qquad
 \mathfrak i_\R(\eta)=\langle1,i^*\eta\rangle_L,
 \qquad
 \ker\mathfrak i_\R=\left\{\eta:\int_Li^*\eta=0\right\}.
\]
Here $M[j]^d=M^{d+j}$ denotes the grading shift.  The kernel is a
subcomplex.  Its cohomology supplies the relative bulk classes, with
the inclusion-induced map
\[
 \rho:H^*(\ker\mathfrak i_\R)\longrightarrow H^*(X;\R).
\]
Let $W\subset H^*(\ker\mathfrak i_\R)$ be a finite-dimensional graded
subspace.  Choose closed representatives of the basis $w_i$ of $W$,
denoted by the same symbols, and set
\[
 \gamma_W=\Gamma_W=\sum_i t_iw_i
 \in\bigl(\mI_{Q_W}(\ker\mathfrak i_\R\otimes Q_W)\bigr)^2.
\]
For a closed bulk form $\gamma$ of degree two, the deformed
$A_\infty$ operations are
\[
 \mathfrak m_k^\gamma(a_1,\ldots,a_k)
 =\sum_{l\ge0}\frac1{l!}
 \q_{k,l}(a_1,\ldots,a_k;\gamma^{\otimes l}).
\]
A \emph{bounding pair over $W$} is a pair $(\gamma_W,b)$ with
\[
 b\in\bigl(((s)+\mI_{R_W})(A^*(L)\otimes R_W)\bigr)^1
\]
such that the Maurer--Cartan equation holds:
\begin{equation}\label{eq:intro-maurer-cartan}
 \sum_{k\ge0}\mathfrak m_k^{\gamma_W}(b^{\otimes k})
 =\sum_{k,l\ge0}\frac1{l!}
 \q_{k,l}(b^{\otimes k};\gamma_W^{\otimes l})
 =c\cdot1_L,
 \qquad c\in\mI_{R_W}\cap\mathcal Z(R_W),\quad |c|=2.
\end{equation}
Following \cite[Definition~1.1]{solomon2019relative}, the pair is
\emph{separated} if $\langle1,b\rangle_L$ is independent of the bulk
coordinates.  It is \emph{point-like} if
\begin{equation}\label{eq:point-like-bounding-pair}
 \langle1,b\rangle_L=-s.
\end{equation}
Thus every point-like pair is separated.  The existence of a bounding
pair is an additional hypothesis; the point-like families for the
even-dimensional Lagrangian spheres considered below are supplied
by~\cite{kosloff2025pointlikenoncommutativef}.

\subsubsection{The closed potential and quantum product}
\label{subsubsec:intro-closed-potential}
Let $U\subset H^*(X;\R)$ be an even graded subspace.  Choose closed
forms $\Delta_j$ whose cohomology classes form a homogeneous basis of
$U$, and set
\[
 \gamma_U=\sum_j t_j\Delta_j,
\]
where $t_j$ are the corresponding formal coordinates.
Define the closed Gromov--Witten potential by
\begin{equation}\label{eq:closed-potential}
 \Phi_U=\sum_{l\ge0}\frac1{(l+1)!}
 \langle\q_{\emptyset,l}(\gamma_U^{\otimes l}),\gamma_U\rangle_X
 \in Q_U.
\end{equation}
Its coefficients are the genus-zero closed Gromov--Witten invariants,
with the relative-spin signs included in $\q_{\emptyset,l}$.
The big quantum product is
\[
 a\star_U b
 =\sum_{l\ge0}\frac1{l!}
 [\q_{\emptyset,l+2}(a,b,\gamma_U^{\otimes l})].
\]
For constant classes $a,b,d\in U$, the quantum product and closed
potential satisfy \cite[Sections~1.2.6--1.2.7]{solomon2019relative}
\[
 \langle a\star_U b,d\rangle_X
 =\langle a,b\star_U d\rangle_X
 =\partial_a\partial_b\partial_d\Phi_U.
\]
\subsection{Results}
\label{subsec:intro-results}

\subsubsection{The relative and enhanced potentials}
\label{subsubsec:intro-potentials}
The scalar disk potential must be combined with sphere contributions
to account for disks whose boundary collapses.  Following
\cite{solomon2019relative}, this combination is naturally a class in
a cone complex.  Extend $\mathfrak i_\R$ to
$\mathfrak i:A^*(X)\otimes Q_W\to HH_0(R_W)[-n]$ and set
\begin{equation}\label{eq:intro-cone}
 \Cone{\mathfrak i}
 =(A^*(X)\otimes Q_W)\oplus HH_0(R_W)[-n-1],\qquad
 d_{\Cone}(\eta,\xi)=(d\eta,\mathfrak i(\eta)).
\end{equation}
For a bounding pair $(\gamma_W,b)$, put
\[
 \q_{\emptyset,0}^{\gamma_W}
 =\sum_{l\ge0}\frac1{l!}\q_{\emptyset,l}(\gamma_W^{\otimes l}),
\]
\[
 \q_{-1,0}^{\gamma_W,b}
 =\sum_{k,l\ge0}\frac{\langle
      \q_{k,l}(b^{\otimes k};\gamma_W^{\otimes l}),b\rangle_L}
      {l!(k+1)}
  +\sum_{l\ge0}\frac1{l!}\q_{-1,l}(\gamma_W^{\otimes l}).
\]
The pair
$\psi=(\q_{\emptyset,0}^{\gamma_W},-\q_{-1,0}^{\gamma_W,b})$
is closed in the cone.  Its class
\[
 \Psi=[\psi]\in H^*(\Cone{\mathfrak i})
\]
is the \emph{relative potential}.  

Write $\Ker\mathfrak i$ and $\Coker\mathfrak i$ for the kernel and
cokernel of the induced map on cohomology.  The cone gives the exact
sequence
\begin{equation}\label{eq:cone-i-exact-sequence}
 0\longrightarrow\Coker\mathfrak i
 \xrightarrow{\ \overline x\ }H^*(\Cone{\mathfrak i})
 \xrightarrow{\ \pi\ }\Ker\mathfrak i\longrightarrow0,
 \qquad
 \overline x([r])=[(0,-r)],\quad \pi[(\eta,\xi)]=[\eta].
\end{equation}
To specify a splitting, consider also the exact sequence before
extending scalars,
\[
 0\longrightarrow\Coker\mathfrak i_\R
 \xrightarrow{\ \overline y\ }H^*(\ker\mathfrak i_\R)
 \xrightarrow{\ \rho\ }\Ker\mathfrak i_\R\longrightarrow0,
\]
where $\overline y$ is the connecting map and the common degree shifts
are suppressed.  Fix a graded $\R$-linear left inverse $P_\R$ of $\overline y$.
Let
\[
 P:H^*(\Cone{\mathfrak i})\longrightarrow\Coker\mathfrak i,
 \qquad
 P\circ\overline x=\Id,
\]
be the graded $Q_W$-linear splitting constructed from $P_\R$ in
\cref{subsec:enhanced-superpotential}.
Define the \emph{enhanced superpotential} by
\[
 \overline\Omega=P\Psi\in\Coker\mathfrak i.
\]
The cokernel is
\[
 \Coker\mathfrak i\simeq
 \begin{cases}
 HH_0(R_W),&[L]=0\text{ in }H_n(X;\R),\\
 HH_0(R_W)/Q_W,&[L]\ne0\text{ in }H_n(X;\R),
 \end{cases}
\]
with the degree shift suppressed.  In the second case,
$\Coker\mathfrak i_\R=0$ and $P$ is canonical.  For a fixed splitting,
both potentials depend only on the gauge class of the bounding pair.

For a point-like pair, define the enhanced open invariants by
\begin{equation}\label{eq:intro-open-invariants}
 \ogw_{\beta,k}(A_1,\ldots,A_l)
 =k!\,[T^\beta s^k]\left.
 \partial_{A_l}\cdots\partial_{A_1}\overline\Omega\right|_{t=0},
 \qquad A_j\in W.
\end{equation}
Thus $k$ is the number of boundary point constraints and the $A_j$
are the interior constraints. The non-commutative invariants retain the
degree, unit, zero-energy, and divisor axioms
of~\cite{solomon2019relative}; see~\cref{prop:enhanced-ogw-axioms}.

\subsubsection{Non-commutative open WDVV equations}
\label{subsubsec:intro-open-wdvv}
To state the equations, assume:
\begin{enumerate}[label=\textup{(A.\arabic*)}]
 \item $U\subset H^{\mathrm{even}}(X;\R)$ contains the unit,
 $U\otimes Q_U$ is closed under $\star_U$, and the restricted
 Poincar\'e pairing is nondegenerate.
 \item $W=\rho^{-1}(U)$ is even, and $(\gamma_W,b)$ is a
 separated bounding pair over $W$.
\end{enumerate}
The map $\rho$ induces a coefficient homomorphism
$\rho^*:Q_U\to Q_W$.  Denote the pulled-back quantum product by
$\star=\rho^*(\star_U)$.  Set
\[
 U_L=U\cap\operatorname{Im}\rho,\qquad W'=\ker(P_\R|_W).
\]
Then $\rho:W'\to U_L$ is an isomorphism.  Let $\kappa:W\to W$ be
the projection onto $W'$ along $\ker(\rho|_W)$.  Its pullback
$\kappa^*$ sets the complementary bulk coordinates to zero.

Choose a basis $\{\Gamma_l\}_{l\in I_L}$ of $W'$ and extend
$\Delta_l=\rho(\Gamma_l)$ to a basis
$\{\Delta_l\}_{l\in I_U}$ of $U$, with $I_L\subset I_U$.  Put
\[
 g_{lm}=\int_X\Delta_l\smile\Delta_m,\qquad
 (g^{lm})=(g_{lm})^{-1}.
\]
For $u=(u_W,x_u)\in(W\otimes Q_W)\oplus R_W$, write
$\overline u=\rho(u_W)$, and
write $\overline w=\rho(w)$ for $w\in W$.  For constant $a,b,d\in U$, abbreviate
$\Phi_{abd}=\rho^*(\partial_a\partial_b\partial_d\Phi_U)$ and
extend this tensor $Q_W$-trilinearly.  For such $u$ and a constant
$w\in W$, write
$\overline\Omega_{uw}=\partial_u\partial_w\overline\Omega$.

\begin{thm}[Non-commutative open WDVV]
\label{thm:intro-open-wdvv}
Let $u,v\in(W\otimes Q_W)\oplus R_W$ be such that
$\partial_u$ and $\partial_v$ are homogeneous derivations, and let
$w\in W$.  Then the following identity holds in $\Coker\mathfrak i$:
\begin{equation}\label{eq:intro-open-wdvv}
\begin{aligned}
 &\sum_{\substack{l\in I_L\\m\in I_U}}
    \overline\Omega_{u\Gamma_l}g^{lm}
    \Phi_{\Delta_m\overline w\overline v}
    -(\partial_uc)\cdot\overline\Omega_{vw}\\
 &\quad=(-1)^{|\partial_u||\partial_v|}
  \left(\sum_{\substack{l\in I_U\\m\in I_L}}
    \Phi_{\overline u\overline w\Delta_l}g^{lm}
    \overline\Omega_{v\Gamma_m}
    -(\partial_vc)\cdot\overline\Omega_{uw}\right).
\end{aligned}
\end{equation}
Here $c\in\mathcal Z(R_W)$ is the coefficient in~\eqref{eq:intro-maurer-cartan}.
Multiplication by $\partial_uc$ and $\partial_vc$ is well defined on
$\Coker\mathfrak i$ by~\cref{lm:mc-coefficient-actions-coker}.
\end{thm}

The proof is given in~\cref{subsec:open-wdvv}.
As in~\cite{solomon2019relative}, the equations follow from a tensor
potential on the cone.  In the present conventions it is
\begin{equation}\label{eq:tensor-potential-definition}
 \nun(\eta,\xi)=
 \bigl(\q_{\emptyset,1}^{\gamma_W}(\eta),
       -\q_{-1,1}^{\gamma_W,b}(\eta)-c\cdot\xi\bigr).
\end{equation}

\begin{thm}[Boundary Euler identity]
\label{thm:intro-boundary-euler}
Let $(\gamma_W,b)$ be a point-like bounding pair with Maurer--Cartan
coefficient $c$ as in~\eqref{eq:intro-maurer-cartan}.  Let
$\mathcal E_s=\partial_{(0,s)}$.  Then
\[
 \mathcal E_s\overline\Omega=[-cs]
 \quad\text{in }\Coker\mathfrak i.
\]
\end{thm}
This result follows from~\cref{thm:enhanced-superpotential-axioms}.
This identity expresses the Maurer--Cartan coefficient $c$ in terms of
coefficients of the potential and is used in both the
character construction and the recursions below.

\subsubsection{A character of quantum cohomology}
\label{subsubsec:intro-character}
Following Solomon and Tukachinsky~\cite{solomon2019relative}, we
identify a subalgebra of quantum cohomology
(\cref{cor:liftable-quantum-subalgebra}):
\[
 \mathcal A=U_L\otimes Q_W\subset(U\otimes Q_W,\star).
\]
For $u\in\mathcal A$, use its unique lift in $W'\otimes Q_W$ to
define $\partial_u$.  Write $c_0=c|_{s=0}$.  The boundary Euler
identity gives
\begin{equation}\label{eq:one-boundary-superpotential}
 [\overline\Omega]_{s^1}=-c_0[s],
\end{equation}
so applying the bulk derivations to $-c_0$ gives generating functions
for the invariants with one boundary constraint.

\begin{thm}[Boundary character]
\label{thm:intro-boundary-character}
Suppose the bounding pair is point-like and satisfies
\[
 \kappa^*\bigl((\partial_sc)|_{s=0}\bigr)=0.
\]
Then the map
\[
 \mathcal L:(\mathcal A,\star)\longrightarrow\kappa^*R_W^0,
 \qquad \mathcal L(u)=-\kappa^*(\partial_uc_0)
\]
is a degree-zero character.  More explicitly,
\[
 \mathcal L(1)=1,\qquad
 \mathcal L(u\star v)=\mathcal L(u)\mathcal L(v),\qquad
 \mathcal L(qu)=\kappa^*(q)\mathcal L(u)\quad(q\in Q_W).
\]
\end{thm}

This is~\cref{thm:open-character}.  If $L$ is monotone, that is,
$\omega$ is a positive multiple of $\mu_L$ on $\Pi$, and
$c_1(TX)\in U$, the character also determines its value on the
Euler element:
\begin{equation}\label{eq:euler-element}
 E=c_1(TX)+\sum_{m\in I_L}
       \left(1-\frac{|\Delta_m|}{2}\right)t_m\Delta_m,
 \qquad \mathcal L(E)=-\kappa^*c_0.
\end{equation}
Here $t_m$ is dual to $\Gamma_m$.  The Euler identity is proved
in~\cref{thm:euler-character}.

\subsubsection{Even-dimensional Lagrangian spheres}
Now suppose that $L$ is monotone, $n$ is even, and
$H^*(L;\R)\simeq H^*(S^n;\R)$.  This includes every even-dimensional
Lagrangian sphere in a monotone symplectic manifold, such as a
vanishing cycle of a nodal degeneration.  Take
$U=H^{\mathrm{even}}(X;\R)$ and $p=\operatorname{PD}[L]$.  A point-like
bounding pair $(\gamma_W,b)$ over $W$
exists~\cite{kosloff2025pointlikenoncommutativef}.  Since $R_W^0$ is
even and $|s|$ is odd, $c$ has only even powers of $s$, so
\cref{thm:intro-boundary-character} applies with no further
hypothesis.

Here $U_L=p^\perp$ (\cref{lm:liftable-orthogonal-submodule}) and
$\langle p,p\rangle_X=2(-1)^{n/2}\ne0$
by~\cite[Lemma~5.6]{hugtenburg2024examples}, so
$U\otimes Q_W=\mathcal A\oplus Q_Wp$ and $\kappa^*=\Id$.  Extend the
basis of $U_L$ by $\Delta_L=p$, and let
\[
 D=\mathcal L\bigl(p\star p-g^{LL}\Phi_{ppp}p\bigr),\qquad
 Q'_W=Q_W[\xi]/\bigl(\xi^2-(g^{LL}\Phi_{ppp})^2-4D\bigr),
 \quad |\xi|=n.
\]

\begin{thm}[Characters from even-dimensional spheres]
\label{thm:intro-monotone-sphere-character}
The character $\mathcal L$ extends to two characters
\[
 \mathcal L_\pm:(U\otimes Q_W,\star)\longrightarrow Q'_W,
 \qquad
 \mathcal L_\pm(p)=\tfrac12\bigl(g^{LL}\Phi_{ppp}\pm\xi\bigr).
\]
They agree on the Euler element, where both equal the generating
function of invariants with one boundary constraint:
\[
 \mathcal L_\pm(E)=-c_0
 =\sum_{\substack{\beta\in\Pi\\ l\ge0}}\frac{T^\beta}{l!}\,
   \ogw_{\beta,1}(\Gamma_W^{\otimes l}),
\]
with $\ogw_{\beta,1}$ extended $Q_W$-multilinearly.
\end{thm}
The extensions are given in~\cref{prop:boundary-character-extensions}, and the
value on $E$ follows from~\eqref{eq:euler-element}, since
$E\in\mathcal A$.

\subsubsection{Reconstruction for the even quadric}
\label{subsec:intro-even-quadric}
Let
\[
 X=\{z_0^2+\cdots+z_n^2-z_{n+1}^2=0\}\subset\CP^{n+1},
 \qquad L=X(\R)\simeq S^n,\qquad n>2\text{ even}.
\]
Fix an orientation of $L$ and the relative spin structure induced by
its spin structure.  Let $h\in H^2(X;\R)$ be the hyperplane class
and $p=\operatorname{PD}[L]\in H^n(X;\R)$.  The classical
cohomology is
\begin{equation}\label{prop:even-quadric-topology}
 H^*(X;\R)=\R[h,p]/(h^{n+1},hp,p^2-(-1)^{n/2}h^n),
 \qquad \int_Xh^n=2.
\end{equation}
The line class identifies
$H_2(X;\Z)$ and $\Pi=H_2(X,L;\Z)$ with $\Z$, and
$\mu_L(d)=2nd$.  Let $\tilde h$ be the relative hyperplane class.
The relative bulk classes are
$\Gamma_i=\tilde h^i$, $0\le i\le n$, with
$\rho(\Gamma_i)=h^i$.

Take $U=H^*(X;\R)$ and $W=\operatorname{span}\{\Gamma_0,\ldots,\Gamma_n\}$.
In this case
\[
 U_L=\operatorname{span}\{1,h,\ldots,h^n\},\qquad
 W'=W,\qquad P_\R=0,\qquad\kappa^*=\Id.
\]
The point-like bounding pair is unique up to gauge equivalence
by~\cite{kosloff2025pointlikenoncommutativef}. So, one can define open Gromov--Witten invariants as the coefficients of the enhanced superpotential $\overline\Omega$.

\begin{thm}[Reconstruction from two open invariants]
\label{thm:intro-quadric-reconstruction}
The classical cohomology of $X$, the open and closed axioms and WDVV
relations, and the initial values
\[
 \ogw_{1,1}(\Gamma_n)=-2,\qquad
 \ogw_{1,3}=2(-1)^{n/2+1}
\]
determine all positive-degree genus-zero closed invariants of $X$ and
all enhanced open invariants
$\ogw_{d,k}(\Gamma_{i_1},\ldots,\Gamma_{i_l})$ of $(X,L)$.
\end{thm}

The first value determines the closed invariant
$\GW_1(h^{n-1},h^n)=4$.  If the closed invariants are given, the
second value alone is sufficient to reconstruct the open invariants.
The open reconstruction is proved
in~\cref{cor:even-quadric-open-reconstruction}, and the recovery of
the closed invariants is explained
in~\cref{rem:even-quadric-open-to-closed-reconstruction}.

\subsubsection{Two characters of the quadric}
\label{subsec:intro-quadric-characters}
For this pair $\Lambda_c=\Lambda$.  Set $q=T^1$ and write
\[
 Q_U=\Lambda\otimes\R[[t_0,\ldots,t_n,t_L]],\qquad
 Q_W=\Lambda\otimes\R[[t_0,\ldots,t_n]].
\]
Here $t_i$ is dual to $h^i$ in $U$ and to $\Gamma_i$ in $W$, while
$t_L$ is dual to $p$; thus $|t_i|=2-2i$, $|t_L|=2-n$, and
$\rho^*$ sets $t_L=0$.  The product $\star$ on $U\otimes Q_W$
retains all the ambient bulk variables.
The coefficient $c$ has only even powers of $s$, so the hypothesis
of the character theorem holds and
\[
 \mathcal L(u)=-\partial_uc_0,
 \qquad u\in\mathcal A=U_L\otimes Q_W.
\]

To extend $\mathcal L$ to the class $p$, we prove $\Phi_{ppp}=0$ and
\begin{equation}\label{eq:intro-primitive-square}
 D=\mathcal L(p\star p)=2[s^3/3!]\overline\Omega
   =4\varepsilon q(1+H),\qquad
 \varepsilon=(-1)^{n/2+1},\quad H\in(t_0,\ldots,t_n)Q_W.
\end{equation}
The coefficient $[s^3/3!]\overline\Omega=3![s^3]\overline\Omega$
is the generating function for invariants with three boundary
constraints.  Thus these invariants determine $D$, and hence
$\vartheta=(1+H)^{1/2}$, the formal square root with constant term one.

\begin{thm}[Enumerative formulas for the quadric characters]
\label{thm:intro-quadric-characters}
Adjoin $\zeta$ with $\zeta^2=\varepsilon q$ and $|\zeta|=n$.
The character $\mathcal L$ has two distinct unital degree-zero
extensions
\[
 \mathcal L_\pm:(U\otimes Q_W,\star)\longrightarrow Q_W[\zeta],
 \qquad
 \mathcal L_\pm(\alpha+fp)=\mathcal L(\alpha)\pm2f\zeta\vartheta,
\]
where $\alpha\in\mathcal A$ and $f\in Q_W$.
\end{thm}

This is~\cref{prop:even-quadric-character-extensions}.
At zero bulk, the two extensions take the primitive class
$p=\operatorname{PD}[L]$ to $\mathcal L_\pm(p)=\pm2\zeta$, where
$\zeta^2=(-1)^{n/2+1}q$.  The invariants
with one boundary constraint determine the common restriction to
$\mathcal A$, and those with three boundary constraints determine
the two primitive values.  Composing with $\Id\otimes\rho^*$ gives
characters of the big closed quantum algebra over $Q_U$ that factor
through $t_L=0$.
\subsection{Directions for future research}
\label{subsec:future-directions}
\subsubsection{Lagrangian cobordisms}
\label{subsec:quantum-discriminant-future}
Biran and Membrez define a discriminant for monotone Lagrangian spheres
satisfying additional hypotheses, using Lagrangian Floer homology. They express the discriminant as four times a count
of disks through three points of $L$, with correction terms involving
auxiliary disk counts \cite[Theorem~7.A]{biranmembrez2016cubic}.

Open Gromov--Witten invariants suggest another description.  For $L$ as in
\cref{subsubsec:boundary-character-extensions}, the characters
$\mathcal L_\pm$ send $p=\operatorname{PD}[L]$ to the roots of
$z^2-az-D$, where $a=g^{LL}\Phi_{ppp}$.  We propose identifying
$\mathfrak D_L=a^2+4D$ at zero bulk with four times the refined
discriminant of Biran and Membrez
\cite[Section~6]{biranmembrez2016cubic}.  For the even quadric,
\cref{lm:even-quadric-three-boundary-character} gives
$\mathfrak D_L=8[s^3/3!]\overline\Omega$, a generating function of
invariants with three boundary constraints.  The quadric sphere has
minimal Maslov number $N_L=2n$, so it lies outside their hypotheses
\cite[Section~5.2.1]{biranmembrez2016cubic}.  It was comunicated to me by P. Bar-Lev that the relation between the
discriminant and open Gromov--Witten invariants will be explored further
in~\cite{barlev-discriminant}.

The discriminant constrains Lagrangian cobordisms.  For a connected
monotone spin cobordism satisfying the end and homological hypotheses
of \cite[Theorem~E]{biranmembrez2016cubic}, the numerical discriminants
of all ends coincide, and their common value is a perfect square if
there are at least three ends. 

Hirschi and Hugtenburg establish a cobordism relation for open
Gromov--Witten invariants~\cite{hirschihugtenburg2026open}.  It would
be interesting to reproduce the results of Biran and Membrez using open Gromov--Witten invariants and, more generally, to understand how open Gromov--Witten
invariants enter Lagrangian cobordism theory.

\subsubsection{Arithmetic of OGW}
\label{subsec:arithmetic-theory-open-gromov}
The theory in this paper is defined over a field of characteristic zero.
It would be interesting to define it over localizations of \(\Z\), compatibly with
reduction modulo primes.  The denominator conjecture of
\cite[Sections~1.2.2 and~1.3, Hypothesis~1]{Relativequantum2025}
predicts that \(\ogw_{\beta,k}\) has denominators supported on the
torsion primes of \(H^*(L;\Z)\).  For the Chiang Lagrangian
\(L_\triangle\subset\CP^3\), we have
\(H^2(L_\triangle;\Z)\simeq\Z/4\Z\), so the conjecture reads
\[
 \ogw_{\beta,k}\in\Z[1/2].
\]
The corresponding question is whether its point-like theory can be
defined over \(\Z[1/2]\).

Evans--Lekili \cite{evanslekili2015chiang} prove, in the pearl model
for Floer theory,
\[
 HF(L_\triangle;\F_5)\ne0,
 \qquad HF(L_\triangle;\F_7)=0.
\]
The requirement that the Maslov-two disk count be an eigenvalue of
quantum multiplication by \(c_1(T\CP^3)\) does not rule out
\(HF(L_\triangle;\F_7)\ne0\)
\cite[Remark~1.2.1]{evanslekili2015chiang}.

The open invariant \(\ogw_{2,2}=5/4\), in our sign conventions,
vanishes modulo \(5\) and is nonzero modulo \(7\)
\cite[Table~1]{Relativequantum2025}.
A point-like theory over \(\Z[1/2]\) compatible with reduction would
therefore predict nonzero Floer cohomology modulo \(5\) and vanishing
modulo \(7\), by
\cref{lm:point-like-floer-nonvanishing,lm:point-like-floer-annihilation}.

Changing one of the three initial signs in the Chiang recursion, while
fixing the other two, gives another formal open WDVV solution
\cite[Section~1.2.6 and Corollary~1.3]{Relativequantum2025}, whose recursion gives
\[
 \ogw_{2,2}=\frac{175}{116}\equiv0\pmod7.
\]
Does this solution come from a monotone \(L'\subset\CP^3\) with a
point-like theory over \({\Z[1/(2\cdot29)]}\) compatible with reduction
modulo \(7\)?
The same reasoning would then give nonzero Floer cohomology modulo
\(7\).  A comparison with monotone Floer theory could give a
Lagrangian realization of the characteristic-\(7\) asymptotic spectral
invariant of \cite{kawamotoshelukhin2024spectral}, as \(L_\triangle\)
does in characteristic \(5\)
\cite[Section~4.2]{kawamotoshelukhin2024spectral}.

\subsubsection{Open Gromov--Witten theory with an anti-symplectic involution}
\label{subsec:open-gromov-witten-antysymplectic}
The even quadric suggests a relation between real bounding cochains and
the coefficients needed to define a quantum character.  By
\cref{thm:intro-quadric-characters}, the two extensions of the boundary
character are obtained by adjoining
\[
    \zeta^2=(-1)^{n/2+1}q.
\]
After adjoining the Novikov root \(\sqrt q\), the character values have
real coefficients when \(n\equiv2\pmod4\), whereas
\(n\equiv0\pmod4\) requires the nonreal root \(i\sqrt q\).
These are precisely the dimensions in which the real locus admits no
real point-like bounding cochain in the involution-compatible sense of
\cite{kosloff2025pointlikenoncommutativef}.  Let $\phi$ denote complex
conjugation on the quadric.  In that theory
\(\phi^*s=(-1)^{n/2}s\), so the real condition
\(\phi^*b=-b\) and the normalization \(\int_Lb=s\) are incompatible
when \(n\equiv0\pmod4\), since \(\phi|_L=\Id\).
When \(n\equiv2\pmod4\), real point-like bounding cochains exist by
\cite{kosloff2025pointlikenoncommutativef}.
Here reality is a condition involving the involution; the point-like
bounding pairs over \(\R\) of
\cref{prop:even-quadric-standing-hypotheses} exist in every even dimension
under consideration.

This leads to a complementary question for a Lagrangian fixed by an
anti-symplectic involution.  Under the hypotheses of the character
construction, does the existence of a real point-like bounding cochain
ensure that the boundary character extends to a character of the full
quantum algebra with real coefficients, after adjoining any necessary
Novikov roots?

Moreover, in dimensions where the real point-like condition is
obstructed, can one construct point-like bounding cochains with complex
coefficients and use the involution to weaken the cohomological
assumptions for their existence?  The desired result would be analogous
to Solomon--Tukachinsky's construction of real point-like bounding
chains \cite{solomon2016point,kosloff2025pointlikenoncommutativef}, in which the involution forces
some obstruction classes to vanish and thereby relaxes the
cohomological hypotheses.  The question is whether a suitable
compatibility with complex conjugation can yield similar cancellations
in the remaining dimensions.

\subsection{Acknowledgment}
\label{subsec:acknowledgment}
I would like to thank my advisor, Jake Solomon, for his guidance and support throughout this project. An earlier version of this work relied on an inelegant, ad hoc construction of the non-commutative open WDVV equations. Jake insisted that this abomination has no place in mathematics and illustrated a more natural approach. This led to the current formulation of the non-commutative open WDVV equations and helped me understand the underlying structure of the open-closed theory. I would also like to thank May Sela, Sara B. Tukachinsky, Asaf Kapota, and Paul Seidel for helpful conversations.

The author used large language models for proofreading, improving the exposition, and critically reviewing and refining the arguments.

\section{Algebraic setup for open Gromov--Witten theory}
\label{sec:open-closed-framework}
\subsection{Axiomatic formulation}
\label{subsec:open-closed-background}
We formulate algebraic assumptions on the Fukaya \(A_\infty\) operations
and the associated closed, scalar, and geodesic operations used in
Solomon--Tukachinsky's construction, recasting the geometric setup of
\cref{subsec:intro-geometric-background} in axiomatic form.
One aim of this paper is to construct the enhanced potential and
derive its properties from algebraic data and identities. Throughout, \(\kk\) is a field of characteristic zero.

This framework builds on the open Gromov--Witten theory and
relative quantum cohomology developed
in~\cite{solomon2016differential,solomon2016point,solomon2019relative}, and
extends the non-commutative coefficient framework
of~\cite{kosloff2025pointlikenoncommutativef}.

The axiomatic setting also accommodates two extensions of the framework
of~\cite{solomon2019relative}.  First, the Lagrangian coefficient DGA \(C_L\) of
\cref{dfn:pre-open-closed-pair} is not assumed graded-commutative; only the
trace is cyclic and \(i^*(C_X)\subset\mathcal Z(C_L)\).  Moreover, the open
trace may take values in a commutative coefficient algebra \(K_L\), as in
the local-system model of \cref{ex:local-system-open-trace}.
Second, the open coefficient ring \(R_W\) of
\cref{dfn:filtered-coefficient-algebras} uses the free boundary variable
and Hochschild coinvariants of \cref{subsubsec:intro-coefficients}.

\subsection{Algebraic setup}
We begin with the algebras, restriction map, and traces underlying the
differential-form construction of~\cite{solomon2016differential}.
The following definition makes sense over an arbitrary commutative ring, although
elsewhere we retain the standing assumption that \(\kk\) is a field of
characteristic zero.
\begin{dfn}[Algebras, restriction, and traces]
\label{dfn:pre-open-closed-pair}
Let \(\kk\) be a commutative ring and let \(n\in\mathbb Z\).
The data consist of a unital
\(\kk\)-DGA \((C_X,d_X)\), a commutative \(\kk\)-algebra \(K_L\), a unital
\(K_L\)-DGA \((C_L,d_L)\), trace maps
\[
\int_X:C_X^{2n}\to\Bbbk,
\qquad
\int_L:C_L^n\to K_L,
\]
and a unital \(\kk\)-DGA morphism
\[
i^*:C_X\to C_L,
\]
where \(C_L\) is regarded as a \(\kk\)-DGA by restriction of scalars.  The
closed trace is \(\kk\)-linear, the open trace is \(K_L\)-linear, and both
vanish on exact elements.  They induce perfect pairings
\[
H^p(C_X)\otimes_{\Bbbk}H^{2n-p}(C_X)\to\Bbbk,
\qquad
H^p(C_L)\otimes_{K_L}H^{n-p}(C_L)\to K_L.
\]
The trace on \(C_L\) is cyclic,
\[
\int_L ab=(-1)^{|a||b|}\int_L ba,
\]
and \(i^*(C_X)\subset\mathcal Z(C_L)\), that is,
\[
(i^*\xi)a=(-1)^{|\xi||a|}a(i^*\xi),
\]
for homogeneous \(a,b\in C_L\) and \(\xi\in C_X\).  We use the pairings
\[
\langle\xi,\eta\rangle_X=(-1)^{|\eta|}\int_X\xi\eta\in\Bbbk,
\qquad
\langle a,b\rangle_L=(-1)^{|b|}\int_Lab\in K_L.
\]
Here perfectness of the open pairing means that, for every \(p\), the map
\[
H^p(C_L)\longrightarrow
\operatorname{Hom}_{K_L}\bigl(H^{n-p}(C_L),K_L\bigr),
\qquad
[a]\longmapsto
\left([b]\longmapsto(-1)^{|b|}\int_Lab\right),
\]
is an isomorphism of \(K_L\)-modules.
\end{dfn}

\begin{ex}[Local-system coefficients]\label{ex:local-system-open-trace}
Take \(\kk=\R\) and let \(F\to L\) be a flat complex vector bundle
of complex rank \(r\geq1\), viewed as a real vector bundle of rank
\(2r\).  Its complex-linear endomorphisms form the algebra bundle
\(\mathcal B=\End_{\C}(F)\), with fibre \(M_r(\C)\) and center
\(L\times\C\).  Set
\[
C_L=A^*(L;\mathcal B),
\qquad K_L=\C,
\qquad
\int_L\alpha:=\int_L^{\mathrm{dR}}\tr_{\C}(\alpha),
\]
where the complex matrix trace \(\tr_{\C}(A)=\sum_{j=1}^r A_{jj}\)
is applied to the endomorphism coefficient before integration.
For a complex line, multiplication by \(a+ib\) is represented in
an associated real frame by
\[
A=\begin{pmatrix}a&-b\\ b&a\end{pmatrix},
\qquad
\tr_{\C}(A)=a+ib,
\qquad
\tr_{\R}(A)=2a.
\]
Thus the open trace takes values in \(\C\), although the ground
field is \(\R\).  The complex matrix trace is cyclic and its pairing
is nondegenerate; Poincar\'e duality with local coefficients then
makes the induced cohomological pairing perfect.
Together with \(C_X=A^*(X;\R)\), ordinary integration on \(X\), and
\(i^*\xi=\xi|_L\cdot1_F\), these data satisfy
\cref{dfn:pre-open-closed-pair}.
For background on higher-rank local systems in Lagrangian Floer
theory, see \cite[Sections~2.1.2 and~2.5]{konstantinov2017higher}.
\end{ex}

For the remainder of this paper, we
retain the degree bound
\[
C_L^j=0\quad(j<0\text{ or }j>n).
\]

\section{Structure maps}
\label{sec:structure-maps}

In this section, we axiomatize the open, closed, and scalar operators used
in this paper.  In \cref{subsec:de-rham-model}, we show that the
differential-form operators of \cref{subsubsec:intro-operations} satisfy the
axioms.

\subsection{Axioms}
\label{subsec:structure-maps}
The trace map
\(\mathfrak i_\R\) of \cref{subsubsec:intro-bounding-pairs} generalizes to
the map
\[
 \mathfrak i_K\colon C_X\longrightarrow K_L[-n],\qquad
 \mathfrak i_K(\alpha):=\langle1,i^*\alpha\rangle_L.
\]
Fix also a \(\kk\)-linear subspace
\(\mathcal D_X^{\mathrm{adm}}\subset Z^2(\ker i^*)\) and a period pairing
\((D,B)\mapsto\int_BD\) between \(\mathcal D_X^{\mathrm{adm}}\) and
\(H_2(X,L;\mathbb Z)\), which is \(\kk\)-linear in \(D\) and additive in
\(B\).  Set
\[
 H^2_{\mathrm{adm}}(X,L;\kk)
 :=\operatorname{im}\!\left(
 \mathcal D_X^{\mathrm{adm}}\longrightarrow H^2(\ker\mathfrak i_K)
 \right).
\]
We assume that the period pairing descends to \(\Pi\) and to
\(H^2_{\mathrm{adm}}(X,L;\kk)\).  So, for
\(A\in H^2_{\mathrm{adm}}(X,L;\kk)\) and \(\beta\in\Pi\), the pairing
\(\int_\beta A\) is well defined.  For \(D\in\mathcal D_X^{\mathrm{adm}}\),
we write \(\int_\beta D:=\int_\beta[D]\).  For \(B\in H_2(X;\mathbb Z)\),
we write \(\int_BA:=\int_{\varpi(B)}A\) and
\(\int_BD:=\int_{\varpi(B)}D\).

For a list \(\eta=(\eta_1,\ldots,\eta_l)\) of homogeneous elements of
\(C_X\) and a permutation \(\sigma\in S_l\), define the Koszul sign
\[
s_\sigma(\eta)
:=
\sum_{\substack{r<s\\ \sigma^{-1}(r)>\sigma^{-1}(s)}}
|\eta_r|\,|\eta_s|
\quad(\mathrm{mod}\ 2).
\]
For a partition \(I\sqcup J=\{1,\ldots,l\}\), write
\(\eta_I:=(\eta_i)_{i\in I}\) for the sublist of \(\eta\) with indices in
\(I\), in increasing order, and similarly for \(\eta_J\).  For a list
\(a=(a_1,\ldots,a_k)\) of homogeneous elements of \(C_L\), define
\[
\iota(a,\eta;i,I)
:=
\left(1+\sum_{j\in J}|\eta_j|\right)
\sum_{r=1}^{i-1}(|a_r|+1)
+\sum_{j\in I}|\eta_j|
+\sum_{\substack{m\in I,\ j\in J\\ j<m}}|\eta_m|\,|\eta_j|,
\]
\[
\iota(\eta;I)
:=
\sum_{j\in I}|\eta_j|
+\sum_{\substack{m\in I,\ j\in J\\ j<m}}|\eta_m|\,|\eta_j|.
\]
We omit empty lists of inputs from the notation.

The following definition is modeled on the properties of the operators of
\cite{solomon2016differential} and \cite[Section~2.3]{solomon2019relative}.
We state the axioms separately for each class \(\beta\in\Pi\) and
\(A\in H_2(X;\Z)\), and before extending scalars, whereas the operators of
\cite{solomon2019relative} include the sums over these classes.  In
Appendix~\ref{app:scalar-extension}, we extend the operators to the
coefficient rings of \cref{sec:extension-of-scalars}, multiplying the
coefficients of the boundary inputs in the order of the inputs.  The axioms
then give the corresponding identities for the 
\(A_\infty\) algebras over non-commutative differential graded algebras of this paper.

\begin{dfn}[Structure maps]\label{dfn:structure-maps}
A system of structure maps on the data of \cref{dfn:pre-open-closed-pair}
consists of the following three families of maps, which satisfy the axioms
listed below.

\medskip
\noindent\textbf{\(C_L\)-valued operations.}
For every \(k,l\ge 0\) and \(\beta\in\Pi\), there is a \(\kk\)-linear map
\[
\q_{k,l}^{\beta}\colon C_X^{\otimes l}\otimes C_L^{\otimes k}\to C_L
\]
of degree \(2-k-2l-\mu_L(\beta)\), which is \(K_L\)-multilinear in the
\(C_L\) inputs.  We write its value as
\(\q_{k,l}^{\beta}(a_1,\ldots,a_k;\eta_1,\ldots,\eta_l)\), even though the
\(C_X\) factors precede the \(C_L\) factors in the source.  These maps
satisfy the following axioms.
\begin{axioms}[label=\textup{(O\arabic*)},ref=\textup{O\arabic*}]
\item\label{ax:qkl-structure-equation}
The structure equations hold: for all homogeneous \(a_i\in C_L\) and
\(\eta_j\in C_X\),
\begin{equation}\label{eq:qkl-structure}
\begin{aligned}
0 ={}&
\sum_{r=1}^{l}
(-1)^{1+\sum_{j=1}^{r-1}|\eta_j|}\,
\qop{k}{l}{\beta}\bigl(a_1,\ldots,a_k;\,
   \eta_1,\ldots,d_X\eta_r,\ldots,\eta_l\bigr) \\
&+ \sum_{\substack{\beta_1+\beta_2=\beta\\
                   k_1+k_2=k+1\\
                   1\le i\le k_1\\
                   I\sqcup J=\{1,\ldots,l\}}}\!\!
(-1)^{\iota(a,\eta;i,I)}\,
\qop{k_1}{|I|}{\beta_1}\Bigl(a_1,\ldots,a_{i-1}, \\
&\hspace{6em}
\qop{k_2}{|J|}{\beta_2}\bigl(a_i,\ldots,a_{i+k_2-1};\,\eta_J\bigr),\,
a_{i+k_2},\ldots,a_k;\,\eta_I\Bigr).
\end{aligned}
\end{equation}

\item\label{ax:qkl-energy}
If \(\omega(\beta)<0\), then \(\q_{k,l}^{\beta}=0\).  If
\(\omega(\beta)=0\), then
\[
\q_{k,l}^{\beta}(a_1,\ldots,a_k;\eta_1,\ldots,\eta_l)
=
\begin{cases}
d_La_1,
&(\beta,k,l)=(\beta_0,1,0),\\
(-1)^{|a_1|}a_1\wedge a_2,
&(\beta,k,l)=(\beta_0,2,0),\\
-i^*\eta_1,
&(\beta,k,l)=(\beta_0,0,1),\\
0,
&\text{otherwise.}
\end{cases}
\]
For every \(E>0\), only finitely many \(\beta\in\Pi\) with
\(\omega(\beta)\le E\) satisfy \(\q_{k,l}^{\beta}\ne0\) for some \(k,l\).

\item\label{ax:qkl-closed-symmetry}
The \(C_X\) inputs are graded symmetric: for every \(\sigma\in S_l\),
\[
\q_{k,l}^{\beta}(a_1,\ldots,a_k;\eta_1,\ldots,\eta_l)
=
(-1)^{s_\sigma(\eta)}
\q_{k,l}^{\beta}
(a_1,\ldots,a_k;\eta_{\sigma(1)},\ldots,\eta_{\sigma(l)}).
\]

\item\label{ax:qkl-divisor}
For \(D\in\mathcal D_X^{\mathrm{adm}}\) and \(l\ge 1\),
\[
\q_{k,l}^{\beta}(a_1,\ldots,a_k;D,\eta_2,\ldots,\eta_l)
=
\left(\int_\beta D\right)
\q_{k,l-1}^{\beta}(a_1,\ldots,a_k;\eta_2,\ldots,\eta_l).
\]
By \ref{ax:qkl-closed-symmetry}, the same identity holds with \(D\) in any
\(C_X\) position, since \(|D|=2\).

\item\label{ax:qkl-top-degree}
Except for the zero-energy operations in \ref{ax:qkl-energy} with
\((\beta,k,l)=(\beta_0,2,0)\) or \((\beta_0,0,1)\),
\[
\left\langle
1_L,
\q_{k,l}^{\beta}(a_1,\ldots,a_k;\eta_1,\ldots,\eta_l)
\right\rangle_L=0.
\]

\item\label{ax:qkl-fundamental-class}
For \(l\ge 1\),
\[
\q_{k,l}^{\beta}(a_1,\ldots,a_k;1_X,\eta_2,\ldots,\eta_l)
=
\begin{cases}
-1_L,
&(k,l,\beta)=(0,1,\beta_0),\\
0,
&\text{otherwise.}
\end{cases}
\]
By \ref{ax:qkl-closed-symmetry}, the same identity holds with \(1_X\) in
any \(C_X\) position.

\item\label{ax:qkl-cyclic-symmetry}
For \(k\ge 1\), the \(C_L\) inputs are cyclic with respect to
\(\langle\cdot,\cdot\rangle_L\):
\[
\begin{aligned}
&\left\langle
\q_{k,l}^{\beta}(a_1,\ldots,a_k;\eta_1,\ldots,\eta_l),
a_{k+1}
\right\rangle_L
\\
&\quad=
(-1)^{(|a_{k+1}|+1)\sum_{r=1}^k(|a_r|+1)}
\left\langle
\q_{k,l}^{\beta}(a_{k+1},a_1,\ldots,a_{k-1};\eta_1,\ldots,\eta_l),
a_k
\right\rangle_L .
\end{aligned}
\]

\item\label{ax:qkl-unit}
If one of the \(C_L\) inputs is \(f\in K_L\cdot 1_L\), then the operation
vanishes, except for
\[
\q_{2,0}^{\beta_0}(f,a)=fa,
\qquad
\q_{2,0}^{\beta_0}(a,f)=(-1)^{|a|}fa.
\]
\end{axioms}

\medskip
\noindent\textbf{\(C_X\)-valued operations.}
For every \(l\ge 0\) and \(A\in H_2(X;\mathbb Z)\), there is a map
\[
\q_{\emptyset,l}^{A}\colon C_X^{\otimes l}\to C_X
\]
of degree \(4-2l-2\langle c_1(TX),A\rangle\).  These maps satisfy the
following axioms.
\begin{axioms}[label=\textup{(C\arabic*)},ref=\textup{C\arabic*}]
\item\label{ax:qempty-unit}
For \(l\ge 1\),
\[
\q_{\emptyset,l}^{A}(1_X,\eta_1,\ldots,\eta_{l-1})
=
\begin{cases}
\eta_1,
&(A,l)=(A_0,2),\\
0,
&\text{otherwise.}
\end{cases}
\]

\item\label{ax:qempty-energy}
If \(\omega(A)<0\), then \(\q_{\emptyset,l}^{A}=0\).  If \(\omega(A)=0\), then
\[
\q_{\emptyset,l}^{A}(\eta_1,\ldots,\eta_l)
=
\begin{cases}
\eta_1\wedge\eta_2,
&(A,l)=(A_0,2),\\
0,
&\text{otherwise.}
\end{cases}
\]
For every \(E>0\), only finitely many \(A\in H_2(X;\mathbb Z)\) with
\(\omega(A)\le E\) satisfy \(\q_{\emptyset,l}^{A}\ne0\) for some \(l\).

\item\label{ax:qempty-divisor}
For \(D\in\mathcal D_X^{\mathrm{adm}}\) and \(l\ge 1\),
\[
\q_{\emptyset,l}^{A}(D,\eta_2,\ldots,\eta_l)
=
\begin{cases}
D\wedge\eta_2,
&(A,l)=(A_0,2),\\[2pt]
\left(\int_A D\right)
\q_{\emptyset,l-1}^{A}(\eta_2,\ldots,\eta_l),
&\text{otherwise.}
\end{cases}
\]

\item\label{ax:qempty-chain-map}
For \(l>0\),
\begin{equation}\label{eq:qempty-chain-map}
d_X\q_{\emptyset,l}^{A}(\eta_1,\ldots,\eta_l)
=
\sum_{r=1}^l
(-1)^{\sum_{j=1}^{r-1}|\eta_j|}
\q_{\emptyset,l}^{A}
(\eta_1,\ldots,d_X\eta_r,\ldots,\eta_l).
\end{equation}
Moreover, \(d_X\q_{\emptyset,0}^{A}=0\).

\item\label{ax:qempty-symmetry}
The inputs are graded symmetric: for every \(\sigma\in S_l\),
\[
\q_{\emptyset,l}^{A}(\eta_1,\ldots,\eta_l)
=
(-1)^{s_\sigma(\eta)}
\q_{\emptyset,l}^{A}(\eta_{\sigma(1)},\ldots,\eta_{\sigma(l)}).
\]
\end{axioms}

\medskip
\noindent\textbf{\(K_L\)-valued operations.}
For every \(l\ge 0\) and \(\beta\in\Pi\), there is a map
\[
\q_{-1,l}^{\beta}\colon C_X^{\otimes l}\to K_L
\]
of degree \(3-n-2l-\mu_L(\beta)\).  These maps satisfy the following axioms.
\begin{axioms}[label=\textup{(S\arabic*)},ref=\textup{S\arabic*}]
\item\label{ax:qminus-structure-equation}
The structure equations hold: for all homogeneous \(\eta_j\in C_X\),
\begin{equation}\label{eq:qminus-structure}
\begin{aligned}
0 ={}&
\sum_{r=1}^{l}
(-1)^{1+\sum_{j=1}^{r-1}|\eta_j|}\,
\qop{-1}{l}{\beta}(\eta_1,\ldots,d_X\eta_r,\ldots,\eta_l) \\
&+ \tfrac{1}{2}\!\!\sum_{\substack{\beta_1+\beta_2=\beta\\
                                   I\sqcup J=\{1,\ldots,l\}}}\!\!
(-1)^{\iota(\eta;I)}\,
\bigl\langle
   \qop{0}{|I|}{\beta_1}(\,;\eta_I),\,
   \qop{0}{|J|}{\beta_2}(\,;\eta_J)
\bigr\rangle_L \\
&- \sum_{\substack{A\in H_2(X;\mathbb Z)\\ \varpi(A)=\beta}}
\bigl\langle
   1_L,\, i^*\qop{\emptyset}{l}{A}(\eta_1,\ldots,\eta_l)
\bigr\rangle_L.
\end{aligned}
\end{equation}

\item\label{ax:qminus-symmetry}
The inputs are graded symmetric: for every \(\sigma\in S_l\),
\[
\q_{-1,l}^{\beta}(\eta_1,\ldots,\eta_l)
=
(-1)^{s_\sigma(\eta)}
\q_{-1,l}^{\beta}(\eta_{\sigma(1)},\ldots,\eta_{\sigma(l)}).
\]

\item\label{ax:qminus-fundamental-class}
For \(l\ge 1\),
\[
\q_{-1,l}^{\beta}(1_X,\eta_2,\ldots,\eta_l)=0.
\]
By \ref{ax:qminus-symmetry}, the same identity holds with \(1_X\) in any
position.

\item\label{ax:qminus-divisor}
For \(D\in\mathcal D_X^{\mathrm{adm}}\) and \(l\ge 1\),
\[
\q_{-1,l}^{\beta}(D,\eta_2,\ldots,\eta_l)
=
\left(\int_\beta D\right)
\q_{-1,l-1}^{\beta}(\eta_2,\ldots,\eta_l).
\]

\item\label{ax:qminus-energy}
If \(\omega(\beta)\le 0\), then \(\q_{-1,l}^{\beta}=0\).  For every \(E>0\),
only finitely many \(\beta\in\Pi\) with \(\omega(\beta)\le E\)
satisfy \(\q_{-1,l}^{\beta}\ne0\) for some \(l\).
\end{axioms}
\end{dfn}

We also assume that there are geodesic operators as in
\cite[Section~3]{solomon2019relative}, which satisfy
\cite[Propositions~3.1--3.5, Lemmas~3.7--3.15, and
Corollary~3.16]{solomon2019relative} and the energy axioms of
\cref{dfn:structure-maps}, except for their prescribed values in class zero.
The geodesic operators are used only in the
proof of \cref{thm:flatness-relation}.

\subsection{The de~Rham model}
\label{subsec:de-rham-model}

We now show that the differential-form operators of
\cref{subsubsec:intro-operations} satisfy the axioms of
\cref{dfn:structure-maps}.  Let \(\kk=K_L=\R\), \(C_X=A^*(X)\), and
\(C_L=A^*(L)\).  The traces are given by integration over \(X\) and \(L\), and
\(i^*\) is the restriction of forms to \(L\).  By Stokes' theorem and
Poincar\'e duality, these data satisfy \cref{dfn:pre-open-closed-pair}.  Let
\(\mathcal D_X^{\mathrm{adm}}\) consist of chosen representatives of relative
classes that annihilate \(S_L\), and let the period pairing be given by
integration.

For \(\beta\in\Pi\) and \(A\in H_2(X;\Z)\), denote by \(\q_{k,l}^\beta\),
\(\q_{-1,l}^\beta\), and \(\q_{\emptyset,l}^A\) the summands of class
\(\beta\) and \(A\) in these operators, so that
\[
\q_{k,l}=\sum_{\beta\in\Pi}T^\beta\q_{k,l}^\beta,\qquad
\q_{-1,l}=\sum_{\beta\in\Pi}T^\beta\q_{-1,l}^\beta,\qquad
\q_{\emptyset,l}=\sum_{A\in H_2(X;\Z)}T^{\varpi(A)}\q_{\emptyset,l}^A.
\]

\begin{prop}\label[prop]{prop:de-rham-structure-maps}
Under the assumptions of \cref{subsec:intro-geometric-background}, the
operators \(\q_{k,l}^\beta\), \(\q_{\emptyset,l}^A\), and
\(\q_{-1,l}^\beta\) form a system of structure maps in the sense of
\cref{dfn:structure-maps}.  Moreover, the geodesic operators of
\cite[Section~3]{solomon2019relative} satisfy the assumptions stated after
\cref{dfn:structure-maps}.
\end{prop}

\begin{proof}
We verify the axioms of \cref{dfn:structure-maps}.  Some of the results of
\cite{solomon2016differential,solomon2019relative} cited below are stated
for the sums over $\Pi$, but their proofs apply to each $\beta$ separately.
The degrees of the operators are computed in
\cite[Proposition~3.5]{solomon2016differential} and
\cite[Lemmas~2.9 and~2.15]{solomon2019relative}.

The operators \(\q_{k,l}^\beta\) satisfy the structure equations
\ref{ax:qkl-structure-equation} by
\cite[Proposition~2.4]{solomon2016differential}.  Axioms
\ref{ax:qkl-closed-symmetry}, \ref{ax:qkl-divisor},
\ref{ax:qkl-fundamental-class}, and \ref{ax:qkl-cyclic-symmetry} follow from
\cite[Propositions~3.6, 3.9, 3.7, and~3.3]{solomon2016differential},
respectively, and the unit axiom \ref{ax:qkl-unit} is the case of constant
\(f\) in \cite[Proposition~3.2]{solomon2016differential}.  Axiom
\ref{ax:qkl-top-degree} follows from
\cite[Proposition~3.12]{solomon2016differential}, and from Stokes' theorem
for \(\q_{1,0}^{\beta_0}=d\).

The operators \(\q_{-1,l}^\beta\) satisfy the structure equations
\ref{ax:qminus-structure-equation} by
\cite[Proposition~2.5]{solomon2016differential}.  Indeed, by
\cite[Lemma~2.15]{solomon2019relative},
\[
\bigl\langle1_L,i^*\q_{\emptyset,l}^A(\eta_1,\ldots,\eta_l)\bigr\rangle_L
=(-1)^{\sum_j|\eta_j|}\int_Li^*\q_{\emptyset,l}^A(\eta_1,\ldots,\eta_l),
\]
so the last term of \eqref{eq:qminus-structure} coincides with the last term
in \cite[Proposition~2.5]{solomon2016differential}.  Axioms
\ref{ax:qminus-symmetry}, \ref{ax:qminus-fundamental-class}, and
\ref{ax:qminus-divisor} follow from
\cite[Propositions~3.6, 3.7, and~3.9]{solomon2016differential},
respectively.

The operators \(\q_{\emptyset,l}^A\) satisfy \ref{ax:qempty-unit},
\ref{ax:qempty-divisor}, and \ref{ax:qempty-symmetry} by
\cite[Lemmas~2.14, 2.18, and~2.16]{solomon2019relative}, respectively,
except for the case \((A,l)=(A_0,2)\) of \ref{ax:qempty-divisor}, which
follows from \cite[Lemma~2.17]{solomon2019relative}.  Axiom
\ref{ax:qempty-chain-map} is \cite[Proposition~3.13]{solomon2016differential}.

It remains to verify the energy axioms \ref{ax:qkl-energy},
\ref{ax:qempty-energy}, and \ref{ax:qminus-energy}.  For the classes
\(\beta_0\) and \(A_0\), they follow from
\cite[Proposition~3.8]{solomon2016differential} and
\cite[Lemma~2.17]{solomon2019relative}.  A \(J\)-holomorphic stable map in
any other class is nonconstant and therefore has positive area.  So, the
moduli spaces are empty for all other classes of nonpositive area.  The
finiteness statements follow from Gromov compactness.

Finally, the geodesic operators satisfy
\cite[Propositions~3.1--3.5, Lemmas~3.7--3.15, and
Corollary~3.16]{solomon2019relative} under the assumptions of
\cref{subsec:intro-geometric-background}, and these energy conditions follow as
above.
\end{proof}

\begin{rem}\label[rem]{rem:other-structure-maps}
For data other than the de~Rham model, such as the local systems of
\cref{ex:local-system-open-trace}, one must construct the structure maps,
with compatible orientations, and verify the axioms separately.
\end{rem}

\section{Coefficient rings and bounding cochains}
\label{sec:extension-of-scalars}

\begin{dfn}[Completed Novikov rings]
\label{dfn:novikov-coefficient-rings}
Use the Novikov rings \(\Lambda_c\subset\Lambda\) and positive-area
ideals \(\Lambda_{c,+}\subset\Lambda_+\) of
\cref{subsubsec:intro-coefficients}, now with coefficients in \(\kk\).  Write
\[
 \Pi^{\geq0}:=\{\beta\in\Pi:\omega(\beta)\geq0\}.
\]
For \(A\in H_2(X;\Z)\), write
\(c_1(A):=\langle c_1(TX),A\rangle\).  The standard normalization
\[
 \mu_L(\varpi(A))=2c_1(A),\qquad
 \omega(\varpi(A))=\omega(A)
\]
gives \(|T^{\varpi(A)}|=2c_1(A)\).
\end{dfn}

\begin{dfn}[Formal coefficient algebras]
\label{dfn:filtered-coefficient-algebras}
Extend the formal coefficient conventions of
\cref{subsubsec:intro-coefficients} to \(\kk\).  For a finite-dimensional
graded bulk space \(V\), put
\[
 P_V:=\operatorname{Sym}\bigl(V^\vee[-2]\bigr).
\]
Write \(\kk[[V]]\) for formal power series in the graded variables
\(V^\vee[-2]\).  From now on, tensor products, direct sums, and derivations
involving the coefficient rings are understood in the completed sense
specified in \cref{subsec:filtrations}; completion symbols are suppressed in
the body.  The absolute coefficient algebra is
\[
 Q_V:=\Lambda_c\otimes_\kk\kk[[V]].
\]
A graded linear map \(V'\to V\) induces a homomorphism \(Q_V\to Q_{V'}\).

For the rest of the paper fix \(W\), and put
\begin{equation}\label{eq:s-free-relative-coefficients}
 R_W^0:=K_L\otimes_\kk\Lambda\otimes_\kk\kk[[W]],\qquad
 R_W:=R_W^0\otimes_\kk\kk\langle s\rangle,\quad |s|=1-n,\qquad
 Q_L:=K_L\otimes_\kk Q_W\subset R_W^0.
\end{equation}
The ideals \(\mathfrak m_W\), \(\mI_{Q_W}\), \(\mI_{R_W}\) are as in
\eqref{eq:positive-deformation-ideals}, and \((s)\subset R_W\) is the
two-sided ideal generated by \(s\).

Both \(Q_W\) and \(R_W^0\) are graded-commutative, with \(K_L\)
concentrated in degree zero, and
\[
 Q_W\subset Q_L\subset\mathcal Z(R_W).
\]
We call the exponent of \(s\) the \(s\)-degree.
\end{dfn}

\subsection{Scalar extension of operations}\label{subsec:scalar-extension-operations}

Let \(HH_0(R_W)\) and \([a]\) be as in \cref{subsubsec:intro-coefficients},
over \(\kk\).  Setting \(s=0\) kills graded commutators, since \(R_W^0\) is
graded-commutative, so \(Q_W\subset Q_L\subset HH_0(R_W)\).
Define the scalar extensions
\[
 C_X^{Q_W}:=C_X\otimes_\kk Q_W,
 \qquad
 C_L^{R_W}:=C_L\otimes_{K_L}R_W.
\]
In
\cref{dfn:scalar-extension-structure-maps} of
Appendix~\ref{app:scalar-extension}, the component operators \(\q_{k,l}^\beta\),
\(\q_{\emptyset,l}^A\), and \(\q_{-1,l}^\beta\) are extended to these
modules by the Koszul-sign rules \textup{(E1)}--\textup{(E4)}, with the
trace-valued operators taking values in \(HH_0(R_W)\).

\begin{dfn}
\label{dfn:novikov-summed-structure-maps}
For \(k,l\ge0\), set
\[
\begin{aligned}
\q_{k,l}&:=\sum_{\beta\in\Pi^{\ge0}}T^\beta\q_{k,l}^\beta,\\
\q_{\emptyset,l}&:=\sum_{\substack{A\in H_2(X;\Z)\\\omega(A)\ge0}}
T^{\varpi(A)}\q_{\emptyset,l}^A,\\
\q_{-1,l}&:=\sum_{\beta\in\Pi^{\ge0}}T^\beta\q_{-1,l}^\beta.
\end{aligned}
\]
\end{dfn}

By \cref{ax:qkl-energy,ax:qempty-energy,ax:qminus-energy}, these sums converge in their completed targets; see
\cref{app:derivation-continuity}.  In the closed sum,
\cref{ax:qempty-energy} ensures that only finitely many nonzero terms
contribute to each coefficient after collecting classes with the same
\(\varpi(A)\).

The scalar extension of the pairing is recorded in
\cref{dfn:scalar-extension-interval-pairing}.
\subsection{Bounding cochains}
This subsection uses the \((\gamma,b)\)-deformed operations
\(\q_{k,l}^{\gamma,b}\), \(\q_{\emptyset,l}^{\gamma}\),
and \(\q_{-1,l}^{\gamma,b}\), constructed in
\cref{subsec:bg-deformations}.  By \cref{lm:qkl-bg-insertions}, the
insertion sums defining them converge on the completed modules over
\(Q_W\) and \(R_W\).  Their deformed structure equations are proved in
\cref{lm:bg-deformed-axioms}.

\begin{dfn}[Bounding pairs]\label{dfn:bounding-pair}
The terminology of \cref{subsubsec:intro-bounding-pairs} extends to
the present coefficient algebras as follows.
A pair
\[
(\gamma,b)\in (\mI_{Q_W} C_X^{Q_W})^2
 \oplus \bigl(((s)+\mI_{R_W})C_L^{R_W}\bigr)^1
\]
is called a bounding pair if \(d_X\gamma=0\) and there exists an
element \(c\in \mI_{R_W}\cap\mathcal{Z}(R_W)\) such that
\[
\q_{0,0}^{\gamma,b}=c\cdot 1.
\]
The terms \emph{separated} and \emph{point-like} retain their meanings
from \cref{subsubsec:intro-bounding-pairs}.  
\end{dfn}

\begin{rem}\label{rem:point-like-existence}
In the de~Rham setting, point-like bounding pairs exist when
\(L\) is an even-dimensional rational cohomology sphere.  Indeed,
\cite[Proposition~6.1\textup{(1)}]{kosloff2025pointlikenoncommutativef}
constructs a bounding pair with \(\int_Lb=s\), which gives
\(\langle1,b\rangle_L=-s\) under our pairing convention.
\end{rem}

Gauge equivalence via pseudoisotopy bounding pairs is defined in
\cref{dfn:gauge-equivalence-bounding-pairs}.
\section{Tensor and relative potentials}
\label{sec:tensor-relative-potentials}
\subsection{Derivations on the cone}
\label{subsec:coefficient-actions-cone}
We use the \(Q_W\)-actions and scalar-extension signs of
\cref{rem:scalar-extension-shifts}.

The trace map \(\mathfrak i_K\) defined above has
a \(Q_W\)-linear scalar extension
\[
\mathfrak i_Q\colon C_X^{Q_W}\longrightarrow Q_L[-n],
\qquad
\mathfrak i_Q(\alpha\otimes q)
:=(-1)^{(|\alpha|+1)|q|}
q\,\langle1,i^*\alpha\rangle_L.
\]
Here the product is taken in \(Q_L\).
Via the inclusion \(Q_L\subset HH_0(R_W)\), the same formula gives
\[
\mathfrak{i}\colon C_X^{Q_W}\longrightarrow HH_0(R_W)[-n],
\qquad
\mathfrak i(\alpha\otimes q):=\mathfrak i_Q(\alpha\otimes q),
\]
where the target has trivial differential.  In the cone below we regard the
same underlying map as a degree-one graded \(Q_W\)-linear map
\[
\mathfrak{i}\colon C_X^{Q_W}\longrightarrow HH_0(R_W)[-n-1].
\]
Thus, for homogeneous \(q\in Q_W\) and \(\eta\in C_X^{Q_W}\),
\[
    \mathfrak{i}(q\eta)=(-1)^{|q|}q\mathfrak{i}(\eta).
\]
\begin{lm}\label{lm:i-cochain-map}
The degree-zero version of \(\mathfrak{i}\) is a cochain map.  Equivalently,
the shifted map used in the cone satisfies \(\mathfrak{i}d_X=0\).
\end{lm}
\begin{proof}
Since $HH_0(R_W)[-n]$ carries the trivial differential, it suffices to show
that $\mathfrak{i}(d_X\eta)=0$ for every homogeneous $\eta\in C_X^{Q_W}$.
By $\kk$-linearity, take $\eta=\xi\otimes q$ with $\xi\in C_X$ and
$q\in Q_W$ homogeneous.  Since \(R_W\) has trivial differential and \(i^*\) is
a cochain map,
\[
i^*(d_X\eta)=(d_Li^*\xi)\otimes q.
\]
The formula in \cref{dfn:scalar-extension-interval-pairing} gives
\[
\mathfrak{i}(d_X\eta)
=
(-1)^{(|d_Li^*\xi|+1)|q|}\,
q\,\langle 1,d_Li^*\xi\rangle_L
.
\]
It remains to show that \(\langle 1,d_La\rangle_L=0\) for homogeneous
\(a\in C_L\).  By \cref{ax:qkl-energy}, \(\qop{1}{0}{\beta_0}(a)=d_La\).
Since the case
\((\beta,k,l)=(\beta_0,1,0)\) is not excluded in \cref{ax:qkl-top-degree},
\[
\langle 1,d_La\rangle_L
=
\left\langle 1,\qop{1}{0}{\beta_0}(a)\right\rangle_L
=0,
\]
as required.
\end{proof}
We also use \(\mathfrak i_K\), \(\mathfrak i_Q\), and \(\mathfrak i\)
for their induced maps on cohomology.

Use the cone and differential of \eqref{eq:intro-cone} over the present
coefficient algebras.  Scalar multiplication is given by
\begin{equation}\label{eq:cone-coefficient-action}
 q\cdot(\eta,\xi)=(q\eta,q\xi).
\end{equation}

A graded derivation \(\delta\) satisfies
\[
    \delta(r_1r_2)
    =\delta(r_1)r_2+(-1)^{|\delta||r_1|}r_1\delta(r_2),
\]
on homogeneous elements.

\begin{dfn}[Coefficient derivations]
\label{dfn:coefficient-derivations}
Let \(\mathfrak g\) be the space of continuous graded \(K_L\)-linear
derivations of \(R_W\) preserving \(Q_W\), and put
\[
    \mathfrak g_{\mathrm{for}}
    :=\{\delta\in\mathfrak g\mid\delta|_{\Lambda}=0\}.
\]
For \(q\in Q_W\), set \((q\delta)(r):=q\delta(r)\), and use the bracket
\[
    [\delta,\epsilon]
    :=\delta\epsilon-(-1)^{|\delta||\epsilon|}\epsilon\delta.
\]
These operations preserve \(\mathfrak g\), since
\(Q_W\subset\mathcal Z(R_W)\).  For homogeneous elements,
\[
    [\delta,q\epsilon]
    =\delta(q)\epsilon+(-1)^{|\delta||q|}q[\delta,\epsilon].
\]
\end{dfn}

The derivations \(\partial_u\) of
\cref{subsubsec:intro-coefficients}, extended \(K_L\)-linearly,
belong to \(\mathfrak g_{\mathrm{for}}\).

\begin{rem}\label{rem:euler-not-s-partial-s}
When \(n\) is even, \(s\) is odd and
\[
    \partial_s(s^2)=0,\qquad \mathcal E_s(s^2)=2s^2.
\]
Thus \(\mathcal E_s\ne s\partial_s\), and \(s\partial_s\) is not a
derivation.  When \(n\) is odd, \(s\) is even and
\(\mathcal E_s=s\partial_s\).
\end{rem}

We use the same symbol \(\delta\) for its restrictions to \(Q_W\) and
\(Q_L\); preservation of \(Q_L\) follows from \(K_L\)-linearity.
On an unshifted scalar extension, define
\begin{equation}\label{eq:induced-coefficient-actions}
    \delta(a\otimes r)
    :=(-1)^{|\delta||a|}a\otimes\delta(r).
\end{equation}
By the graded Leibniz rule, \(\delta\) maps graded commutators to graded
commutators, so it induces a map on \(HH_0(R_W)\) given by
\(\delta([r]):=[\delta(r)]\).

For homogeneous \(\delta\in\mathfrak g\), write
\(\delta_X,\delta_L,\bar\delta\) for the operators it induces on
\(C_X^{Q_W},C_L^{R_W},HH_0(R_W)\), respectively, and set
\[
    \delta_{\Cone}(\eta,\xi):=(\delta_X\eta,\bar\delta\xi).
\]
\begin{lm}[Compatibility with the differentials]
\label{lm:coefficient-actions-cone}
For every homogeneous \(\delta\in\mathfrak g\),
\[
\begin{aligned}
d_X\delta_X&=(-1)^{|\delta|}\delta_Xd_X,&
d_L\delta_L&=(-1)^{|\delta|}\delta_Ld_L,\\
\mathfrak i\delta_X&=(-1)^{|\delta|}\bar\delta\,\mathfrak i,&
d_{\Cone}\delta_{\Cone}&=(-1)^{|\delta|}\delta_{\Cone}d_{\Cone}.
\end{aligned}
\]
Consequently, these operators descend to cohomology and preserve
the kernel and image of the induced map \(\mathfrak i\).
\end{lm}

\begin{proof}
Compatibility with \(d_X\) and \(d_L\) follows from
\eqref{eq:induced-coefficient-actions}.

For homogeneous \(\alpha\in C_X\) and \(q\in Q_W\), we compute
\[
\begin{aligned}
\mathfrak i\bigl(\delta_X(\alpha\otimes q)\bigr)
&=(-1)^{|\delta||\alpha|+(|\alpha|+1)(|q|+|\delta|)}
  \delta(q)\,\langle1,i^*\alpha\rangle_L
  \\
&=(-1)^{|\delta|}\bar\delta\,\mathfrak i(\alpha\otimes q).
\end{aligned}
\]
Here \(\delta\) annihilates
\(\langle1,i^*\alpha\rangle_L\in K_L\).  Thus, for
\((\eta,\xi)\in\Cone{\mathfrak i}\),
\[
\begin{aligned}
d_{\Cone}\delta_{\Cone}(\eta,\xi)
&=(d_X\delta_X\eta,\mathfrak i\delta_X\eta)\\
&=(-1)^{|\delta|}(\delta_Xd_X\eta,\bar\delta\,\mathfrak i\eta)\\
&=(-1)^{|\delta|}\delta_{\Cone}d_{\Cone}(\eta,\xi).
\end{aligned}
\]
\end{proof}

For \(M=C_X^{Q_W},C_L^{R_W},HH_0(R_W)\), or
\(\Cone{\mathfrak i}\), write \(\delta_M\) for the operator induced
by \(\delta\).  For homogeneous \(q,m,\delta,\epsilon\), the graded
Leibniz rule and the formulas defining these operators give
\begin{align}
\delta_M(qm)
&=\delta(q)m
  +(-1)^{|\delta||q|}q\,\delta_M(m),\label{eq:coefficient-action-leibniz}\\
(q\delta)_M&=q\delta_M,\label{eq:coefficient-action-q-linearity}\\
[\delta_M,\epsilon_M]_{\mathrm{gr}}
&=[\delta,\epsilon]_M.\label{eq:coefficient-action-commutator}
\end{align}
On the cone, the product \(qm\) is the scalar multiplication defined in \eqref{eq:cone-coefficient-action}.

For a homogeneous linear map \(G\colon M\to N\) between modules to which
\(\delta\) has been extended, define \(\delta G\) by
\begin{equation}\label{eq:operator-derivative-convention}
    (\delta G)(x)
    :=\delta(G(x))-(-1)^{|\delta||G|}G(\delta(x)).
\end{equation}

\subsection{Tensor potential and relative quantum connection}
For a bounding pair \((\gamma,b)\) with Maurer--Cartan coefficient \(c\),
define \(\nun^{\gamma,b}\colon\Cone{\mathfrak i}\to\Cone{\mathfrak i}\) by
\[
\nun^{\gamma,b}(\eta,\xi)=
\bigl(\q_{\emptyset,1}^{\gamma}(\eta),\,
-\q_{-1,1}^{\gamma,b}(\eta)-c\cdot\xi\bigr).
\]

\begin{lm}\label{lm:tensor-potential-second-component}
For any \(\eta\in C_X^{Q_W}\), we have
\[
-\q_{-1,1}^{\gamma,b}(d_X\eta)-c\cdot\mathfrak{i}(\eta)
=
\mathfrak{i}\!\left(\q_{\emptyset,1}^{\gamma}(\eta)\right).
\]
\end{lm}
\begin{proof}
By linearity, it is enough to take \(\eta\) homogeneous.  The deformed
\(l=1\) case of \cref{eq:qminus-structure}, provided by
\cref{lm:bg-deformed-axioms}, gives
\[
\begin{aligned}
0={}&-\q_{-1,1}^{\gamma,b}(d_X\eta)
+\frac{1}{2}\left(
\left\langle
\q_{0,0}^{\gamma,b},\q_{0,1}^{\gamma,b}(\eta)
\right\rangle_L
+(-1)^{|\eta|}
\left\langle
\q_{0,1}^{\gamma,b}(\eta),\q_{0,0}^{\gamma,b}
\right\rangle_L
\right)  \\
&\qquad
-\mathfrak{i}\!\left(\q_{\emptyset,1}^{\gamma}(\eta)\right).
\end{aligned}
\]
We first simplify the middle term.  Since \((\gamma,b)\) is a bounding pair,
\(\q_{0,0}^{\gamma,b}=c\cdot1\), where \(|c|=2\).  Moreover,
\[
\left|\q_{0,1}^{\gamma,b}(\eta)\right|=|\eta|.
\]
The pairing convention and its scalar extension in
\cref{dfn:pre-open-closed-pair,dfn:scalar-extension-interval-pairing} give
\[
\left\langle\q_{0,1}^{\gamma,b}(\eta),1\right\rangle_L
=
(-1)^{|\eta|}
\left\langle1,\q_{0,1}^{\gamma,b}(\eta)\right\rangle_L.
\]
Since \(c\) is even, the same scalar-extension formula therefore yields
\[
\begin{aligned}
&\frac{1}{2}\left(
\left\langle
\q_{0,0}^{\gamma,b},\q_{0,1}^{\gamma,b}(\eta)
\right\rangle_L
+(-1)^{|\eta|}
\left\langle
\q_{0,1}^{\gamma,b}(\eta),\q_{0,0}^{\gamma,b}
\right\rangle_L
\right) \\
&\qquad
=c\cdot
\left\langle1,\q_{0,1}^{\gamma,b}(\eta)\right\rangle_L.
\end{aligned}
\]

It remains to compute the pairing on the right.  Expanding the deformation,
\[
\q_{0,1}^{\gamma,b}(\eta)
=
\sum_{\substack{k\ge0\\ m\ge1}}
\frac{1}{(m-1)!}\,
\q_{k,m}\bigl(
\underbrace{b,\ldots,b}_{k};\eta,
\underbrace{\gamma,\ldots,\gamma}_{m-1}
\bigr).
\]
By \cref{ax:qkl-top-degree}, pairing with \(1\) kills every summand except
the zero-energy exceptions in \cref{ax:qkl-energy}.  Since \(m\ge1\), only
\((\beta,k,m)=(\beta_0,0,1)\) can occur, and
\(\qop{0}{1}{\beta_0}(\,;\eta)=-i^*\eta\).  Hence
\[
\left\langle1,\q_{0,1}^{\gamma,b}(\eta)\right\rangle_L
=-\langle 1,i^*\eta\rangle_L
=-\mathfrak{i}(\eta).
\]
Substituting the two simplifications into the structure equation gives
\[
0=
-\q_{-1,1}^{\gamma,b}(d_X\eta)
-c\cdot\mathfrak{i}(\eta)
-\mathfrak{i}\!\left(\q_{\emptyset,1}^{\gamma}(\eta)\right),
\]
which is equivalent to the desired identity.
\end{proof}

\begin{lm}\label{lm:nun-cochain-map}
The tensor potential \(\nun^{\gamma,b}\) is a degree \(2\) cochain map on
\(\Cone{\mathfrak{i}}\).
For homogeneous \(\delta\in\mathfrak g_{\mathrm{for}}\), its derivative
\(\delta\nun^{\gamma,b}\) is a graded \(Q_W\)-linear cochain endomorphism
of degree \(|\delta|+2\).
\end{lm}
\begin{proof}
The degree of \(\nun^{\gamma,b}\) is given by
\cref{lm:tensor-potential-degree}.  Since this degree is even, the
cochain-map condition is
\(d_{\Cone}\circ\nun^{\gamma,b}
=\nun^{\gamma,b}\circ d_{\Cone}\).
For \((\eta,\xi)\in\Cone{\mathfrak{i}}\), the cone differential gives
\[
d_{\Cone}\bigl(\nun^{\gamma,b}(\eta,\xi)\bigr)
=
\Bigl(d_X\q_{\emptyset,1}^{\gamma}(\eta),\,
\mathfrak{i}\bigl(\q_{\emptyset,1}^{\gamma}(\eta)\bigr)\Bigr),
\]
while
\[
\nun^{\gamma,b}\bigl(d_{\Cone}(\eta,\xi)\bigr)
=
\Bigl(\q_{\emptyset,1}^{\gamma}(d_X\eta),\,
-\q_{-1,1}^{\gamma,b}(d_X\eta)-c\cdot\mathfrak{i}(\eta)\Bigr).
\]
By \cref{eq:qempty-chain-map,lm:bg-deformed-axioms}, the first components
agree:
\[
d_X\q_{\emptyset,1}^{\gamma}(\eta)=\q_{\emptyset,1}^{\gamma}(d_X\eta).
\]
The second components agree by \cref{lm:tensor-potential-second-component}.
Therefore \(d_{\Cone}\circ\nun^{\gamma,b}=\nun^{\gamma,b}\circ d_{\Cone}\).

The derivative
\(\delta\nun^{\gamma,b}=[\delta_{\Cone},\nun^{\gamma,b}]_{\mathrm{gr}}\)
is a graded cochain map by \cref{lm:coefficient-actions-cone}.
Its graded \(Q_W\)-linearity follows from
\eqref{eq:coefficient-action-leibniz}.
\end{proof}

We also use \(\nun^{\gamma,b}\), \(\delta\nun^{\gamma,b}\), and
\(\delta\) for the induced operators on cohomology.

By \cref{thm:tensor-potential-gauge-invariance}, the induced tensor potential
on \(H^*(\Cone{\mathfrak i})\) depends only on the gauge-equivalence class of
the bounding pair \((\gamma,b)\).

\begin{prop}[Commutator identity]
\label[prop]{prop:tensor-potential-derivative-identity}
For homogeneous \(\delta,\epsilon\in\mathfrak g_{\mathrm{for}}\),
\[
 (q\delta)\nun^{\gamma,b}=q\bigl(\delta\nun^{\gamma,b}\bigr)
\qquad(q\in Q_W),
\]
and
\[
\delta\!\left(\epsilon\nun^{\gamma,b}\right)
-(-1)^{|\delta||\epsilon|}
 \epsilon\!\left(\delta\nun^{\gamma,b}\right)
-[\delta,\epsilon]\nun^{\gamma,b}=0.
\]
Here the action on endomorphisms is defined by
\cref{eq:operator-derivative-convention}.
\end{prop}

\begin{proof}
Because \(\nun^{\gamma,b}\) is \(Q_W\)-linear of even degree and
the operators \(\delta_{\Cone}\) satisfy
equation~\eqref{eq:coefficient-action-q-linearity},
its graded commutator with \(q\delta\) is
\(q(\delta\nun^{\gamma,b})\).

For the second assertion, apply the graded Jacobi identity in
\(\End_\kk(\Cone{\mathfrak i})\):
\[
\begin{aligned}
\delta\!\left(\epsilon\nun^{\gamma,b}\right)
-(-1)^{|\delta||\epsilon|}
 \epsilon\!\left(\delta\nun^{\gamma,b}\right)
&=
[\delta,[\epsilon,\nun^{\gamma,b}]_{\mathrm{gr}}]_{\mathrm{gr}}\\
&\quad-
(-1)^{|\delta||\epsilon|}
[\epsilon,[\delta,\nun^{\gamma,b}]_{\mathrm{gr}}]_{\mathrm{gr}}\\
&=[[\delta,\epsilon]_{\mathrm{gr}},
   \nun^{\gamma,b}]_{\mathrm{gr}}\\
&=[\delta,\epsilon]\nun^{\gamma,b},
\end{aligned}
\]
where the last equality is
equation~\eqref{eq:coefficient-action-commutator}.  This calculation includes
the action of each derivation on the coefficients of the other.
\end{proof}

\begin{thm}[Commutation up to chain homotopy]\label{thm:flatness-relation}
Let \((\gamma,b)\) be a bounding pair and let
\(\delta,\epsilon\in\mathfrak g_{\mathrm{for}}\) be homogeneous.
Then \([\delta\nun^{\gamma,b},\epsilon\nun^{\gamma,b}]_{\mathrm{gr}}\) is
chain homotopic to zero.  In particular,
\[
[\delta\nun^{\gamma,b},\epsilon\nun^{\gamma,b}]_{\mathrm{gr}}=0
\quad\text{on}\quad
H^*(\Cone{\mathfrak i}).
\]
\end{thm}
This is \cite[Theorem~5]{solomon2019relative}.  Why its proof remains valid
with our completed, non-commutative coefficients and cone convention is
explained in Appendix~\ref{app:standard-arguments}.

We follow the formulation of the relative quantum connection in
\cite[Section~1.3.6]{solomon2019relative}.  For connection operators, we
use the grading modulo \(2\).

\begin{thm}[Flat relative connection]
\label{thm:flat-relative-connection}
For homogeneous \(\delta\in\mathfrak g_{\mathrm{for}}\), the operator
\begin{equation}\label{eq:relative-connection}
\nabla_\delta
:=\delta+\delta\nun^{\gamma,b}
\end{equation}
on \(\Cone{\mathfrak i}\) satisfies
\(d_{\Cone}\nabla_\delta=(-1)^{|\delta|}\nabla_\delta d_{\Cone}\) and
\[
\nabla_\delta(qm)
=\delta(q)m
+(-1)^{|\delta||q|}q\nabla_\delta m,
\qquad
\nabla_{q\delta}=q\nabla_\delta.
\]
For homogeneous \(\delta,\epsilon\), its curvature is
\begin{equation}\label{eq:relative-curvature}
[\nabla_\delta,\nabla_\epsilon]_{\mathrm{gr}}
-\nabla_{[\delta,\epsilon]}
=[\delta\nun^{\gamma,b},\epsilon\nun^{\gamma,b}]_{\mathrm{gr}},
\end{equation}
which is chain homotopic to zero.
In particular, \(\nabla\) is flat on \(H^*(\Cone{\mathfrak i})\).
\end{thm}

\begin{proof}
The two summands in \eqref{eq:relative-connection} are cochain maps of
degrees \(|\delta|\) and \(|\delta|+2\) by
\cref{lm:coefficient-actions-cone,lm:nun-cochain-map}.  These degrees have
the same parity, which gives the relation with \(d_{\Cone}\).
Since \(\delta\nun^{\gamma,b}\) is graded \(Q_W\)-linear by
\cref{lm:nun-cochain-map}, the Leibniz rule and linearity in \(\delta\)
follow from \eqref{eq:coefficient-action-leibniz},
\eqref{eq:coefficient-action-q-linearity}, and
\cref{prop:tensor-potential-derivative-identity}.
Expanding the curvature, the terms linear in \(\nun^{\gamma,b}\) cancel by
\cref{prop:tensor-potential-derivative-identity}, giving
\eqref{eq:relative-curvature}.  The right-hand side is chain homotopic to
zero by \cref{thm:flatness-relation}, so \(\nabla\) is flat on cohomology.
\end{proof}

\subsection{Relative potential}
For a bounding pair \((\gamma,b)\), use the cochain \(\psi(\gamma,b)\)
and its class \(\Psi(\gamma,b)\) from \cref{subsubsec:intro-potentials},
with \(\gamma_W\) replaced by \(\gamma\).  The following lemma establishes
closedness over the present coefficient algebras.
\begin{lm}\label{lm:relative-potential-closed}
   $d_{\Cone}\psi(\gamma,b)=0$.
\end{lm}
The applicability of the proof of
\cite[Lemma~4.8]{solomon2019relative} in the present setting is explained in
Appendix~\ref{app:standard-arguments}.

By \cref{thm:relative-potential-gauge-invariance}, the class
\(\Psi(\gamma,b)\) depends only on the gauge-equivalence class of the
bounding pair.
\begin{lm}\label{lm:relative-potential-derivative}
Let \(\delta\in\mathfrak g_{\mathrm{for}}\) be homogeneous.
Applying \(\delta\) to the relative potential gives
\[
\delta\psi(\gamma,b)
=
\nun^{\gamma,b}\bigl(\delta\gamma,\delta\langle 1,b\rangle_L\bigr).
\]
\end{lm}
\begin{proof}
By \cref{lm:bg-deformation-qkl} and
\cref{app:derivation-continuity}, we may differentiate the deformation
formulas term by term.  Since \(\delta\in\mathfrak g_{\mathrm{for}}\), it
annihilates the Novikov weights.

For the closed component, \cref{dfn:bg-deformation-qempty} gives
\begin{align}
\delta\q_{\emptyset,0}^{\gamma}
&=
\sum_{m\ge1}\frac{1}{m!}
\sum_{j=1}^{m}
\q_{\emptyset,m}
\bigl(\gamma^{\otimes(j-1)},\delta\gamma,
\gamma^{\otimes(m-j)}\bigr)\notag\\
&=
\sum_{m\ge1}\frac{1}{(m-1)!}
\q_{\emptyset,m}
\bigl(\delta\gamma,\gamma^{\otimes(m-1)}\bigr)\notag\\
&=\q_{\emptyset,1}^{\gamma}(\delta\gamma).
\label{eq:relative-potential-derivative-closed}
\end{align}
The sum over \(j\) comes from differentiating each of the \(m\) copies
of \(\gamma\); rule \textup{(E3)} of
\cref{dfn:scalar-extension-structure-maps} and
\cref{ax:qempty-symmetry} make these \(m\) terms equal.

For the second component, first differentiate the copies of \(\gamma\).
Rules \textup{(E2)} and \textup{(E4)} of
\cref{dfn:scalar-extension-structure-maps}, the pairing of
\cref{dfn:scalar-extension-interval-pairing}, and
\cref{ax:qkl-closed-symmetry,ax:qminus-symmetry} give
\begin{align*}
&
\sum_{\substack{k\ge0\\m\ge1}}
\frac{1}{m!(k+1)}
\sum_{j=1}^{m}
\left\langle
\q_{k,m}\bigl(
b^{\otimes k};
\gamma^{\otimes(j-1)},\delta\gamma,\gamma^{\otimes(m-j)}
\bigr),b
\right\rangle_L
\\
&+
\sum_{m\ge1}\frac{1}{m!}
\sum_{j=1}^{m}
\q_{-1,m}
\bigl(\gamma^{\otimes(j-1)},\delta\gamma,
\gamma^{\otimes(m-j)}\bigr)
\\
&=
\sum_{\substack{k\ge0\\m\ge1}}
\frac{
\left\langle
\q_{k,m}\bigl(
b^{\otimes k};\delta\gamma,\gamma^{\otimes(m-1)}
\bigr),b
\right\rangle_L
}{(m-1)!(k+1)}
\\
&+
\sum_{m\ge1}\frac{1}{(m-1)!}
\q_{-1,m}
\bigl(\delta\gamma,\gamma^{\otimes(m-1)}\bigr)
\\
&=\q_{-1,1}^{\gamma,b}(\delta\gamma).
\end{align*}

It remains to differentiate the \(k+1\) copies of \(b\) in
\(\langle\q_{k,m}(b^{\otimes k};\gamma^{\otimes m}),b\rangle_L\).
By rule \textup{(E1)} of \cref{dfn:scalar-extension-structure-maps}, the
pairing of \cref{dfn:scalar-extension-interval-pairing}, and
\cref{ax:qkl-cyclic-symmetry}, each of the \(k+1\) terms equals
\[
(-1)^{|\delta|}
\left\langle
\q_{k,m}\bigl(b^{\otimes k};\gamma^{\otimes m}\bigr),
\delta b
\right\rangle_L.
\]
Summing over \(k\) and \(m\) gives
\begin{align*}
&
(-1)^{|\delta|}
\sum_{k,m\ge0}
\frac{k+1}{m!(k+1)}
\left\langle
\q_{k,m}\bigl(b^{\otimes k};\gamma^{\otimes m}\bigr),
\delta b
\right\rangle_L\\
&=
(-1)^{|\delta|}
\left\langle
\q_{0,0}^{\gamma,b},\delta b
\right\rangle_L.
\end{align*}
By \cref{dfn:scalar-extension-interval-pairing},
\[
\delta\langle1,b\rangle_L
=(-1)^{|\delta|}\langle1,\delta b\rangle_L.
\]
Since \((\gamma,b)\) is a bounding pair,
\(\q_{0,0}^{\gamma,b}=c\cdot1\), where \(|c|=2\).  Therefore
\[
(-1)^{|\delta|}
\left\langle\q_{0,0}^{\gamma,b},\delta b\right\rangle_L
=(-1)^{|\delta|}\langle c\cdot1,\delta b\rangle_L
=c\cdot\delta\langle1,b\rangle_L.
\]
Adding the two parts gives
\begin{equation}
\delta\q_{-1,0}^{\gamma,b}
=\q_{-1,1}^{\gamma,b}(\delta\gamma)
+c\cdot\delta\langle1,b\rangle_L.
\label{eq:relative-potential-derivative-scalar}
\end{equation}

Finally,
\cref{eq:relative-potential-derivative-closed,eq:relative-potential-derivative-scalar},
the definition of \(\psi\), and
\eqref{eq:tensor-potential-definition} give
\[
\begin{aligned}
\delta\psi(\gamma,b)
&=
\left(
\q_{\emptyset,1}^{\gamma}(\delta\gamma),
-\q_{-1,1}^{\gamma,b}(\delta\gamma)
-c\cdot\delta\langle1,b\rangle_L
\right)\\
&=
\nun^{\gamma,b}
\bigl(\delta\gamma,\delta\langle1,b\rangle_L\bigr),
\end{aligned}
\]
as required.
\end{proof}

\begin{rem}
\label[rem]{rem:odd-boundary-derivative-sign}
Apart from the conventions of \cref{rem:conventions-st}, there is a discrepancy
between our calculation and the calculation in
\cite[Lemma~4.9]{solomon2019relative}.  If \(\partial\) is a homogeneous
derivation, its action on the paired term in \(\q_{-1,0}^{\gamma,b}\) contributes
\[
(-1)^{|\partial|}
\left\langle\q_{0,0}^{\gamma,b},\partial b\right\rangle_L.
\]
This sign factor is omitted in the first line of the calculation with
\(\partial_s\) there.  For the boundary variable \(s\),
\(|\partial_s|=n-1\) and \(|\partial_s b|=n\), so the corrected identity is
\[
\partial_s\psi_{\mathrm{ST}}
=\bigl(0,(-1)^n c\,\partial_s\!\int_L b\bigr).
\]
It agrees with the displayed conclusion of the cited lemma when \(n\) is odd
and has the opposite sign when \(n\) is even.  The final equality
\(\nun_{\mathrm{ST}}^{\gamma,b}(0,1)=(0,-c)\) remains correct.  The proof
above keeps this sign.
\end{rem}

\subsection{Enhanced superpotential}
\label{subsec:enhanced-superpotential}
We construct the splitting used in \cref{subsubsec:intro-potentials},
following \cite[Lemmas~4.10 and~4.11]{solomon2019relative}.
Recall the trace map \(\mathfrak i_K\) from \cref{subsec:structure-maps}
and its extensions \(\mathfrak i_Q\) and \(\mathfrak i\) from
\cref{subsec:coefficient-actions-cone}.
The cone exact sequence \eqref{eq:cone-i-exact-sequence}, with
\(\overline x([r])=[(0,-r)]\) and \(\pi[(\eta,\xi)]=[\eta]\),
holds over the present coefficient algebras by the cohomology long
exact sequence of the cone.

The inclusion
\(Q_L\subset HH_0(R_W)\) identifies
\(\ker\mathfrak i=\ker\mathfrak i_Q\) and
\(\Ker\mathfrak i=\Ker\mathfrak i_Q\).  Lowercase \(\ker\) denotes
the cochain kernel; \(\Ker,\Coker\) retain the cohomological meanings
of \cref{subsubsec:intro-potentials}.  We have a connecting map
\[
 \overline y\colon
 \frac{\mathfrak i_K(C_X^n)}{\operatorname{Im}\mathfrak i_K}
 \longrightarrow H^*(\ker\mathfrak i_K),
 \qquad
 \overline y([\mathfrak i_K(\alpha)])=[d_X\alpha].
\]
Its domain equals \(\Coker\mathfrak i_K\) when the cochain
trace is onto \(K_L\), as in the real de~Rham model; in general it can
be smaller.  Alongside the cone exact sequence
\eqref{eq:cone-i-exact-sequence}, we use
\begin{equation}\label{eq:ker-i-exact-sequence}
 0\longrightarrow
 \frac{\mathfrak i_Q(C_X^{Q_W})}{\operatorname{Im}\mathfrak i_Q}
 \xrightarrow{\,\overline y\otimes\Id_{Q_W}\,}H^*(\ker\mathfrak i)
 \xrightarrow{\,\rho\,}\Ker\mathfrak i_Q
 \longrightarrow0.
\end{equation}
This is the scalar extension of the corresponding sequence over
\(\kk\).  The inclusion-induced map \(\rho\) of
\cref{subsubsec:intro-bounding-pairs} now takes values in \(H^*(C_X)\);
we use the same notation for its scalar extension.
The cochain inclusion \(\eta\mapsto(\eta,0)\) induces an injection
\[
 a\colon H^*(\ker\mathfrak i)\longrightarrow H^*(\Cone{\mathfrak i}),
 \qquad \pi\circ a=\rho.
\]
We use the scalar-extension identifications
\[
 H^*(\ker\mathfrak i)=H^*(\ker\mathfrak i_K)\otimes Q_W,
 \qquad
 \Coker\mathfrak i_Q=(\Coker\mathfrak i_K)\otimes Q_W
 \subset\Coker\mathfrak i,
\]
with the coefficient signs of \eqref{eq:scalar-target-coefficient-actions}.
The last inclusion is induced by \(Q_L\subset HH_0(R_W)\).

Since \(\kk\) is a field, choose a graded \(\kk\)-linear left inverse
\[
 P_K\colon H^*(\ker\mathfrak i_K)\longrightarrow
 \frac{\mathfrak i_K(C_X^n)}{\operatorname{Im}\mathfrak i_K}
 \subseteq\Coker\mathfrak i_K
\]
of \(\overline y\).  In the real de~Rham model, this is the choice
\(P_\R\) of \cref{subsubsec:intro-potentials}.  Set
\[
 P_Q:=P_K\otimes\Id_{Q_W}\colon
 H^*(\ker\mathfrak i)\longrightarrow\Coker\mathfrak i_Q.
\]

\begin{lm}
\label{lm:enhanced-superpotential-choice-from-base}
There is a direct-sum decomposition
\[
 H^*(\Cone{\mathfrak i})
 =\operatorname{Im}\overline x\oplus a(\ker P_Q).
\]
Projection onto the first summand followed by \(\overline x^{-1}\)
defines the unique degree-zero graded \(Q_W\)-linear map
\(P\colon H^*(\Cone{\mathfrak i})\to\Coker\mathfrak i\) satisfying
\[
 P\circ\overline x=\Id,
 \qquad P\circ a=P_Q.
\]
Its kernel is \(a(\ker P_Q)\), and the natural image of
\(P([(\eta,\xi)])\) in \(HH_0(R_W)/Q_L\) is \([-\xi]\).
\end{lm}
\begin{proof}
Since \(P_Q(\overline y\otimes\Id_{Q_W})=\Id\), exactness of
\eqref{eq:ker-i-exact-sequence} gives an isomorphism
\(\rho|_{\ker P_Q}\colon\ker P_Q\to\Ker\mathfrak i_Q\).
Together with \(\pi a=\rho\), the cone exact sequence gives the
direct-sum decomposition.  The identity
\(a((\overline y\otimes\Id_{Q_W})(r))=\overline x(r)\)
then shows that the projection satisfies \(P\circ a=P_Q\).
These restrictions determine \(P\) uniquely.  The maps to
\(HH_0(R_W)/Q_L\) induced by \(P\) and by
\([(\eta,\xi)]\mapsto[-\xi]\) agree on \(\operatorname{Im}\overline x\)
and vanish on \(a(\ker P_Q)\), proving the last assertion.
\end{proof}

Fix this splitting \(P\).  If the natural map
\(\Coker\mathfrak i\to HH_0(R_W)/Q_L\) is an isomorphism, the last
assertion determines \(P\) canonically.

\begin{dfn}
\label{dfn:enhanced-superpotential}
For a bounding pair \((\gamma,b)\), use the enhanced superpotential
\(\overline\Omega(\gamma,b)\) of \cref{subsubsec:intro-potentials},
formed with the splitting \(P\) just constructed.
\end{dfn}
Its natural image in \(HH_0(R_W)/Q_L\) is the class of
\(\q_{-1,0}^{\gamma,b}\), by
\cref{lm:enhanced-superpotential-choice-from-base}.

By \cref{lm:coefficient-actions-cone}, each \(\delta\in\mathfrak g\) induces
operators on
\(\Coker\mathfrak i\) and \(\Coker\mathfrak i_Q\), with
\(\delta[r]=[\delta r]\).
As above, we use the same symbol for these induced operators.
An element of \(R_W\) or \(HH_0(R_W)\) in a formula valued in
\(\Coker\mathfrak i\) denotes its natural quotient class.

\begin{lm}\label{lm:splitting-commutes-derivations}
Let \(P\) be the splitting chosen above.  For every
homogeneous \(\delta\in\mathfrak g\), one has
\[
    P\circ\delta
    =
    \delta\circ P
    \colon
    H^*(\Cone{\mathfrak{i}})
    \longrightarrow
    \Coker\mathfrak{i}.
\]
\end{lm}
\begin{proof}
Use the decomposition of \cref{lm:enhanced-superpotential-choice-from-base}.
The action of \(\delta\) on coefficients gives
\[
 \delta\overline x(c)=\overline x(\delta c),
 \qquad \delta a(u)=a(\delta u).
\]
Since \(P\circ\overline x=\Id\), the first identity proves the claim on
\(\operatorname{Im}\overline x\).
For homogeneous \(u_0\in H^*(\ker\mathfrak i_K)\) and \(q\in Q_W\),
\eqref{eq:tensor-to-hh0} gives
\[
 P_Q(u_0\otimes q)=(-1)^{(n+1)|q|}[P_K(u_0)q].
\]
If \(P_K(u_0)\ne0\), then \(|u_0|=n+1\), so
\eqref{eq:induced-coefficient-actions} gives
\(P_Q\delta=\delta P_Q\).
Thus \(\ker P_Q\) is preserved by \(\delta\), and both
\(P\delta\) and \(\delta P\) vanish on \(a(\ker P_Q)\).
\end{proof}

\begin{rem}
\label{rem:conventions-st}
For the reader comparing with \cite{solomon2019relative}, we collect here
all differences of convention; every adaptation of an argument of
\cite{solomon2019relative} below refers back to this remark.
\begin{enumerate}
\item Our map \(\mathfrak i(\eta)=\langle 1,i^*\eta\rangle_L\) replaces
the map \(\mathrm i(\eta)=(-1)^{n+|\eta|}\int_L i^*\eta\) of
\cite[Section~1.3.1]{solomon2019relative}.
\item Our tensor potential \eqref{eq:tensor-potential-definition}
differs from \cite[Eq.~(17)]{solomon2019relative} by the factor \((-1)^n\)
in the sign of the \(\q^{\gamma,b}_{-1,1}(\eta)\)-term.
\item Our scalar term \(\q_{-1,0}^{\gamma,b}\) carries no factor
\((-1)^n\), whereas \cite[Eq.~(42)]{solomon2019relative} includes that
factor.
\item The displayed identification
\(\overline{x}(r)=[(0,-r)]\) has the same form as
\cite[Eq.~(36)]{solomon2019relative}, but the cokernel coordinate transported
between the two cone conventions is
\(r_{\mathrm{ours}}=(-1)^n r_{\mathrm{ST}}\).
\end{enumerate}
Precisely, set
\[
F_n(\eta,\xi):=(\eta,(-1)^n\xi),
\qquad
J_n([r]):=[(-1)^n r].
\]
Since \(\mathfrak i=(-1)^n\mathrm i_{\mathrm{ST}}\), including after scalar
extension, one has
\[
d_{\Cone}^{\mathrm{ours}}F_n=F_n d_{\Cone}^{\mathrm{ST}},
\qquad
\nun_{\mathrm{ours}}^{\gamma,b}
=F_n\nun_{\mathrm{ST}}^{\gamma,b}F_n^{-1},
\qquad
\psi_{\mathrm{ours}}=F_n\psi_{\mathrm{ST}}.
\]
In particular, the \(\q^{\gamma,b}_{-1,1}(\eta)\)-term in the second
component acquires the factor \((-1)^n(-1)^{n+1}=-1\), while the
\(c\cdot\xi\)-term is unchanged because \((-1)^{2n}=1\).  On the long exact
sequence,
\[
H^*(F_n)\circ\overline{x}_{\mathrm{ST}}
=\overline{x}_{\mathrm{ours}}\circ J_n.
\]
Thus, for compatible splittings,
\[
P_{\mathrm{ours}}\circ H^*(F_n)=J_n\circ P_{\mathrm{ST}},
\qquad
\overline\Omega_{\mathrm{ours}}=(-1)^n\overline\Omega_{\mathrm{ST}}.
\]

\end{rem}

\subsection{Enhanced axioms}
\label{subsec:enhanced-axioms}

\begin{assum}[Point-like bounding pair]
\label[assum]{assum:point-like-bounding-pair}
Assume that there exists a point-like bounding pair \((\gamma_W,b)\)
with \(\gamma_W=\Gamma_W\), where \(\Gamma_W\) is the bulk
representative specified below.  Fix such a pair for this subsection.
\end{assum}

We establish identities for the enhanced superpotential of this pair
and deduce its coefficient axioms.
The variables \(t_1,\ldots,t_N\) are homogeneous coordinates on the
previously fixed bulk space \(W\).  For the enhanced
equations, assume that \(W\) admits the decomposition
\[
    W=\kk\cdot1_X\oplus W_{\mathrm{rel}},
\]
where the classes in \(W_{\mathrm{rel}}\) have cocycle representatives in
\(\ker i^*\).  Assume that
\[
 W^2\subset H^2_{\mathrm{adm}}(X,L;\kk).
\]
Thus, for \(A\in W^2\) and \(\beta\in\Pi\), the pairing
\(\int_\beta A\) is well defined.

Take the basis \(w_1,\ldots,w_N\) adapted to this decomposition,
with \(w_1=1_X\), and choose cocycle representatives, denoted by
the same symbols, in \(\ker i^*\) for \(i>1\) and in
\(\mathcal D_X^{\mathrm{adm}}\) in degree two.  Form the bulk
representative \(\gamma_W=\Gamma_W\) as in
\cref{subsubsec:intro-bounding-pairs} from these cocycles.
Require
\[
    \left.[T^{\beta_0}]b\right|_{t=0}
    \in C_L\otimes_\kk s\kk\langle s\rangle.
\]

Write \(P_Q=P_K\otimes\Id_{Q_W}\) for the scalar extension fixed in
\cref{lm:enhanced-superpotential-choice-from-base}, and use the natural map
from \(\Coker\mathfrak i_Q\) to
\(\Coker\mathfrak i\) without further notation.  The product
\(A\smile B\) below is represented by the product in \(C_X\).  It belongs to
\(\ker\mathfrak i_Q\) under the preceding choice of representatives, so
\(P_Q(A\smile B)\) is defined.

\begin{dfn}[Divisor derivations]
\label{notn:novikov-divisor-fields}
For \(A\in H^2_{\mathrm{adm}}(X,L;\kk)\), define \(E_A\in\mathfrak g\)
termwise by
\[
    E_A(T^\beta)=\left(\int_\beta A\right)T^\beta,
    \qquad E_A(t_i)=E_A(s)=0.
\]
Additivity of the periods gives the Leibniz rule, and preservation of
energy gives continuity.  Thus \(E_A\) is a degree-zero derivation
preserving \(Q_W\); it belongs to \(\mathfrak g_{\mathrm{for}}\) exactly
when all periods vanish.
\end{dfn}

The derivations
appearing below descend to \(HH_0(R_W)\), the cone, and its cokernel, and they
commute with \(P\) by
\cref{lm:splitting-commutes-derivations}.

\begin{lm}\label{lm:relative-potential-divisor}
For \(A\in W^2\),
\[
 (\partial_A-E_A)\Psi(\gamma_W,b)
 =a(A\smile\Gamma_W).
\]
\end{lm}
\begin{proof}
Let
\[
 D:=\partial_A\gamma_W\in\mathcal D_X^{\mathrm{adm}}
\]
be the representative determined by the chosen coordinates.  Then
\(E_A\gamma_W=0\).
By the definition of \(a\), it suffices to prove the cochain identity
\begin{equation}
 (\partial_A-E_A)\psi(\gamma_W,b)
 =(D\wedge\gamma_W,0).
 \label{eq:relative-divisor-cochain-target}
\end{equation}
Since \(\partial_A\in\mathfrak g_{\mathrm{for}}\),
\(\partial_A\gamma_W=D\), and
\(\partial_A\langle1,b\rangle_L=0\) by
the assumption that \((\gamma_W,b)\) is point-like, and seprated. So,
\cref{lm:relative-potential-derivative} and
\eqref{eq:tensor-potential-definition} give
\begin{align}
 \partial_A\q_{\emptyset,0}^{\gamma_W}
 &=\q_{\emptyset,1}^{\gamma_W}(D),
 \label{eq:relative-divisor-closed-coordinate}\\
 \partial_A\q_{-1,0}^{\gamma_W,b}
 &=\q_{-1,1}^{\gamma_W,b}(D).
 \label{eq:relative-divisor-scalar-coordinate}
\end{align}

It remains to compute the action of \(E_A\) on the Novikov weights.  For
the closed component,
\begin{align}
 E_A\q_{\emptyset,0}^{\gamma_W}
 &=
 \sum_{\substack{B\in H_2(X;\Z)\\\omega(B)\ge0}}\sum_{m\ge0}
 \frac{T^{\varpi(B)}}{m!}
 \left(\int_BD\right)
 \q_{\emptyset,m}^{B}(\gamma_W^{\otimes m})\notag\\
 &=
 \sum_{\substack{B\in H_2(X;\Z)\\\omega(B)\ge0}}\sum_{m\ge0}
 \frac{T^{\varpi(B)}}{m!}\,
 \q_{\emptyset,m+1}^{B}
 \bigl(D,\gamma_W^{\otimes m}\bigr)
 -D\wedge\gamma_W\notag\\
 &=\q_{\emptyset,1}^{\gamma_W}(D)-D\wedge\gamma_W.
 \label{eq:relative-divisor-closed-novikov}
\end{align}
The subtracted term is the exception in
\cref{ax:qempty-divisor},
\(\q_{\emptyset,2}^{A_0}(D,\gamma_W)=D\wedge\gamma_W\), which 
\(E_A\) evaluates to zero.  Equations
\eqref{eq:relative-divisor-closed-coordinate} and
\eqref{eq:relative-divisor-closed-novikov} therefore give
\begin{equation}
 (\partial_A-E_A)\q_{\emptyset,0}^{\gamma_W}
 =D\wedge\gamma_W.
 \label{eq:relative-divisor-closed-component}
\end{equation}

Let \(c\) be the Maurer--Cartan coefficient.  For the $HH_0$ component, the
terms in which \(E_A\) acts on the
boundary insertions and the second argument of the pairing combine, by
\cref{ax:qkl-cyclic-symmetry}
and the Maurer--Cartan equation, to
\[
 \left\langle\q_{0,0}^{\gamma_W,b},E_Ab\right\rangle_L
 =c\cdot E_A\langle1,b\rangle_L
 =c\cdot E_A(-s)=0.
\]
Here we used \eqref{eq:point-like-bounding-pair} and \(E_A(s)=0\) from
\cref{notn:novikov-divisor-fields}.

The remaining terms come from the action of \(E_A\) on the Novikov weights:
\begin{align*}
 &\sum_{\substack{\beta\in\Pi^{\ge0}\\k,m\ge0}}
 \frac{T^\beta}{m!(k+1)}
 \left(\int_\beta D\right)
 \left\langle
 \q_{k,m}^{\beta}
 \bigl(b^{\otimes k};\gamma_W^{\otimes m}\bigr),b
 \right\rangle_L\\
 &\quad+
 \sum_{\substack{\beta\in\Pi^{\ge0}\\m\ge0}}
 \frac{T^\beta}{m!}
 \left(\int_\beta D\right)
 \q_{-1,m}^{\beta}(\gamma_W^{\otimes m})\\
 &=
 \sum_{\substack{\beta\in\Pi^{\ge0}\\k,m\ge0}}
 \frac{T^\beta}{m!(k+1)}
 \left\langle
 \q_{k,m+1}^{\beta}
 \bigl(b^{\otimes k};D,\gamma_W^{\otimes m}\bigr),b
 \right\rangle_L\\
 &\quad+
 \sum_{\substack{\beta\in\Pi^{\ge0}\\m\ge0}}
 \frac{T^\beta}{m!}
 \q_{-1,m+1}^{\beta}
 \bigl(D,\gamma_W^{\otimes m}\bigr)\\
 &=\q_{-1,1}^{\gamma_W,b}(D).
\end{align*}
The second equality is exactly
\cref{ax:qkl-divisor,ax:qminus-divisor}.  Together with the vanishing
boundary contribution above, it gives
\begin{equation}
 E_A\q_{-1,0}^{\gamma_W,b}
 =\q_{-1,1}^{\gamma_W,b}(D).
 \label{eq:relative-divisor-scalar-novikov}
\end{equation}
Equations \eqref{eq:relative-divisor-scalar-coordinate} and
\eqref{eq:relative-divisor-scalar-novikov} imply
\begin{equation}
 (\partial_A-E_A)\q_{-1,0}^{\gamma_W,b}=0.
 \label{eq:relative-divisor-scalar-component}
\end{equation}

Combining \eqref{eq:relative-divisor-closed-component} and
\eqref{eq:relative-divisor-scalar-component} with the definition of
\(\psi\) proves \eqref{eq:relative-divisor-cochain-target}.
Since \(i^*D=0\), one has \(i^*(D\wedge\gamma_W)=0\); and since
\(d_XD=d_X\gamma_W=0\), the product is a cocycle in
\(\ker\mathfrak i_Q\).  Consequently,
\[
 [(D\wedge\gamma_W,0)]
 =a([D\wedge\gamma_W])
 =a(A\smile\Gamma_W).
\]
Passing to cohomology in \eqref{eq:relative-divisor-cochain-target} gives
the result.
\end{proof}

\begin{thm}[Enhanced superpotential equations]
\label{thm:enhanced-superpotential-axioms}
Let \(c\in \mI_{R_W}\cap\mathcal Z(R_W)\) be the Maurer--Cartan coefficient of
\((\gamma_W,b)\).  The enhanced superpotential satisfies the following
equations.
\begin{enumerate}[label=\textup{(\alph*)}]
\item\label{item:enhanced-degree}
\emph{Homogeneity:}
\(\overline\Omega\) is homogeneous of degree \(3-n\) in the unshifted
coefficient grading.

\item\label{item:enhanced-s-equation}
\emph{\(s\)-equation:} for every homogeneous derivation
\(\delta\) vanishing on \(R_W^0\),
\[
    \delta\overline\Omega
    =
    [-c\cdot\delta(s)]
    \qquad\text{in }\Coker\mathfrak i.
\]
In particular,
\[
    \mathcal E_s\overline\Omega
    =
    [-cs].
\]

\item\label{item:enhanced-unit}
\emph{Unit / fundamental class equation:}
\[
    \partial_1\overline\Omega
    =
    T^{\beta_0}\bigl(P_Q(\Gamma_W)+s\bigr).
\]

\item\label{item:enhanced-divisor}
\emph{Divisor equation:} if \(A\in W^2\), then
\[
    (\partial_A-E_A)\overline\Omega
    =P_Q(A\smile\Gamma_W).
\]

\item\label{item:enhanced-zero}
\emph{Zero-energy equation:}
\[
    [T^{\beta_0}]\overline\Omega
    =
    t_1s+\frac12P_Q(\Gamma_W\smile\Gamma_W).
\]
\end{enumerate}
\end{thm}

\begin{proof}
We first prove \ref{item:enhanced-degree}.  By
\cref{dfn:bg-deformation-qminusone}, the terms that define
\(\q_{-1,0}^{\gamma_W,b}\) are of the form
\[
 \left\langle
 \q_{k,m}(b^{\otimes k};\gamma_W^{\otimes m}),b
 \right\rangle_L
 \qquad\text{or}\qquad
 \q_{-1,m}(\gamma_W^{\otimes m}).
\]
The degree conventions in
\cref{dfn:structure-maps,dfn:novikov-summed-structure-maps} give
\[
    |\q_{k,m}|=2-k-2m,
    \qquad
    |\q_{-1,m}|=3-n-2m.
\]
By \cref{dfn:bounding-pair},
\(|b|=1\) and \(|\gamma_W|=2\), while the scalar-extended pairing has
degree \(-n\) by \cref{dfn:scalar-extension-interval-pairing}.  Thus the
two types of terms have degree
\[
    (2-k-2m)+k+2m+1-n=3-n,
    \qquad
    (3-n-2m)+2m=3-n,
\]
respectively.  Similarly, the degree convention in
\cref{dfn:structure-maps,dfn:novikov-summed-structure-maps} and the closed
deformation formula \cref{dfn:bg-deformation-qempty} give
\(|\q_{\emptyset,0}^{\gamma_W}|=4\).  Therefore
\[
 \psi=(\q_{\emptyset,0}^{\gamma_W},
       -\q_{-1,0}^{\gamma_W,b})
\]
has degree \(4\) in \(\Cone{\mathfrak i}\), whose second summand is shifted
by \(-n-1\).  By
\cref{lm:enhanced-superpotential-choice-from-base}, the splitting \(P\) has
degree zero
when the shifts in the splitting sequence are restored.  Hence
\(\overline\Omega=P[\psi]\) has degree \(3-n\) in the unshifted
coefficient grading.

Let \(\delta\) be a homogeneous derivation vanishing on \(R_W^0\).
It acts trivially on
\(\gamma_W\), and
\(\delta\langle1,b\rangle_L=-\delta s\) by
\eqref{eq:point-like-bounding-pair}.  Hence
\cref{lm:relative-potential-derivative} gives
\[
    \delta\psi
    =\nun^{\gamma_W,b}(0,-\delta s)
    =(0,c\cdot\delta s).
\]
The last equality is the definition of the tensor potential
\eqref{eq:tensor-potential-definition}, where \(c\in\mathcal Z(R_W)\)
satisfies the Maurer--Cartan equation in \cref{dfn:bounding-pair},
\(\q_{0,0}^{\gamma_W,b}=c\cdot1_L\).
With \(\overline x(r)=[(0,-r)]\), this says
\[
    \delta\Psi
    =\overline x\!\left([-c\cdot\delta(s)]\right).
\]
By
\cref{lm:splitting-commutes-derivations,lm:enhanced-superpotential-choice-from-base}
and the identity \(P\overline x=\Id\),
\[
    \delta\overline\Omega
    =P(\delta\Psi)=[-c\cdot\delta(s)].
\]
Taking \(\delta=\mathcal E_s\) proves \ref{item:enhanced-s-equation}.

For the unit equation, \cref{lm:relative-potential-derivative} gives
\[
    \partial_1\psi
    =\left(
       \q_{\emptyset,1}^{\gamma_W}(1_X),
       -\q_{-1,1}^{\gamma_W,b}(1_X)
     \right),
\]
where \(\partial_1\gamma_W=1_X\) and
\(\partial_1\langle1,b\rangle_L=0\) by
the separatedness assumption on \((\gamma_W,b)\).  Using the deformation formulas
\cref{dfn:bg-deformation-qempty,dfn:bg-deformation-qminusone}, the closed
unit \cref{ax:qempty-unit}, the fundamental-class
\cref{ax:qkl-fundamental-class,ax:qminus-fundamental-class}, and
\eqref{eq:point-like-bounding-pair}, we obtain
\[
 \q_{\emptyset,1}^{\gamma_W}(1_X)=T^{\beta_0}\gamma_W,
 \qquad
 \q_{-1,1}^{\gamma_W,b}(1_X)
 =T^{\beta_0}\langle-1_L,b\rangle_L
 =T^{\beta_0}s.
\]
Consequently,
\[
    \partial_1\Psi
    =T^{\beta_0}\bigl(a(\Gamma_W)+\overline x(s)\bigr).
\]
Applying \(P\) gives
\[
    \partial_1\overline\Omega
    =T^{\beta_0}\bigl(P_Q(\Gamma_W)+s\bigr),
\]
which proves \ref{item:enhanced-unit}.

Applying \(P\) to the identity in \cref{lm:relative-potential-divisor} gives
\[
 (\partial_A-E_A)\overline\Omega
 =P_Q(A\smile\Gamma_W).
\]
This proves \ref{item:enhanced-divisor}.

It remains to prove the zero-energy equation.  If
\(A\in W_{\mathrm{rel}}\), then \(i^*A=0\).  The closed formula
\cref{dfn:bg-deformation-qempty} and the closed zero-energy
\cref{ax:qempty-energy} give the first equality below.  The scalar
 formula \cref{dfn:bg-deformation-qminusone}, together with 
\cref{ax:qkl-energy,ax:qminus-energy}, gives the second:
\[
 [T^{\beta_0}]\q_{\emptyset,1}^{\gamma_W}(A)
 =A\smile\Gamma_W,
 \qquad
 [T^{\beta_0}]\q_{-1,1}^{\gamma_W,b}(A)
 =\left\langle-i^*A,[T^{\beta_0}]b\right\rangle_L=0.
\]
Thus
\cref{lm:relative-potential-derivative} give
\[
    [T^{\beta_0}]\partial_A\psi
    =
    (A\smile\Gamma_W,0).
\]
Taking cohomology classes in the preceding equality and applying \(P\),
we obtain
\[
    [T^{\beta_0}]\partial_A\overline\Omega
    =
    P\bigl([(A\smile\Gamma_W,0)]\bigr)
    =
    P_Q(A\smile\Gamma_W).
\]
The last equality follows from
\(P\circ a=P_Q\) in \cref{lm:enhanced-superpotential-choice-from-base}.
For \(A=1_X\), the unit equation gives
\[
    [T^{\beta_0}]\partial_1\overline\Omega
    =
    P_Q(\Gamma_W)+s.
\]
Since \(P_Q=P_K\otimes\Id_{Q_W}\) and \(\Gamma_W\) has even degree,
\[
 \partial_A\left(\frac12P_Q(\Gamma_W\smile\Gamma_W)\right)
 =\frac12P_Q(A\smile\Gamma_W+\Gamma_W\smile A)
 =P_Q(A\smile\Gamma_W)
\]
for every homogeneous \(A\in W\).  Together with
\(\partial_{w_i}(t_1s)=\delta_{i1}s\), the preceding identities give
\[
 \partial_{w_i}\left(
 [T^{\beta_0}]\overline\Omega-t_1s
 -\frac12P_Q(\Gamma_W\smile\Gamma_W)
 \right)=0,
 \qquad i=1,\ldots,N.
\]
Thus the expression is independent of \(t_1,\ldots,t_N\).
It remains to prove that its value at \(t=0\) is zero.

By \cref{dfn:bounding-pair,dfn:filtered-coefficient-algebras},
\[
 c\in \mI_{R_W}=\Lambda_+R_W+\mathfrak m_WR_W.
\]
The first summand has positive Novikov energy, and every term in the second
contains a bulk variable.  It follows directly that
\[
 \left.[T^{\beta_0}]c\right|_{t=0}=0.
\]
Taking the coefficient of \(T^{\beta_0}\) in the \(s\)-equation and then
setting \(t=0\) gives
\[
 \mathcal E_s
 \left(
  \left.[T^{\beta_0}]\overline\Omega\right|_{t=0}
 \right)=0.
\]
By \cref{lm:s-degree-one-detector}, the kernel of \(\mathcal E_s\) on
\(\Coker\mathfrak i\) is its \(s\)-degree-zero component.  Hence the last
display shows that
\(\left.[T^{\beta_0}]\overline\Omega\right|_{t=0}\) has \(s\)-degree zero.

Finally, the hypothesis
\[
 \left.[T^{\beta_0}]b\right|_{t=0}
 \in C_L\otimes_\kk s\kk\langle s\rangle
\]
implies
\(\left.[T^{\beta_0}]b\right|_{t=s=0}=0\), and
\(\gamma_W|_{t=0}=0\).  The closed zero-energy
\cref{ax:qempty-energy} gives
\[
 \left.[T^{\beta_0}]
 \q_{\emptyset,0}^{\gamma_W}\right|_{t=s=0}=0.
\]
The scalar zero-energy \cref{ax:qminus-energy} gives
\(\q_{-1,l}^{\beta_0}=0\).  By
\cref{dfn:bg-deformation-qminusone}, every remaining term in
\(\q_{-1,0}^{\gamma_W,b}\) is paired with \(b\).  The
open energy \cref{ax:qkl-energy}
implies that a
term of total class \(\beta_0\) must use the \(T^{\beta_0}\)-coefficient of
that final copy.  This coefficient vanishes at \(t=s=0\), so
\[
 \left.[T^{\beta_0}]\psi\right|_{t=s=0}=0,
 \qquad
 \left.[T^{\beta_0}]\overline\Omega\right|_{t=s=0}=0.
\]
It follows that
\(\left.[T^{\beta_0}]\overline\Omega\right|_{t=0}=0\).  Since both terms
on the right-hand side of \ref{item:enhanced-zero} also vanish at \(t=0\),
the equation follows.
\end{proof}

Let
\[
 \varepsilon_{1_X}\colon
 W=\kk\cdot1_X\oplus W_{\mathrm{rel}}\longrightarrow\kk
\]
be the projection onto the coefficient of \(1_X\).

\begin{notn}[Enhanced open coefficients]
\label{notn:enhanced-ogw-coefficients}
For \(\beta\in\Pi^{\geq0}\), \(k\geq0\), and homogeneous
\(A_1,\ldots,A_l\in W\), use the coefficient definition
\eqref{eq:intro-open-invariants}.  Coefficient extraction uses the
\(s\)-degree projections of \cref{lm:s-degree-decomposition}; coefficients
killed by the quotient are assigned zero.  For example, if
\(|s|\) is odd, positive even powers vanish in \(HH_0(R_W)\), although no
relation \(s^2=0\) is imposed in \(R_W\).
\end{notn}

\begin{prop}[Enhanced axioms]
\label[prop]{prop:enhanced-ogw-axioms}
\label[prop]{prop:noncommutative-ogw-axioms}
The coefficients of the enhanced superpotential satisfy:
\begin{enumerate}[label=\textup{(\alph*)}]
\item \textup{(Degree)}
\[
    \ogw_{\beta,k}(A_1,\ldots,A_l)=0
\]
unless
\[
    n-3+\mu_L(\beta)+k+2l
    =
    kn+\sum_{j=1}^l|A_j|.
\]

\item \textup{(Unit / fundamental class)}
\[
\ogw_{\beta,k}(1_X,A_1,\ldots,A_{l-1})
=
\begin{cases}
1,&(\beta,k,l)=(\beta_0,1,1),\\
P_K(A_1),&(\beta,k,l)=(\beta_0,0,2),\\
0,&\text{otherwise.}
\end{cases}
\]

\item \textup{(Zero)}
\[
\ogw_{\beta_0,k}(A_1,\ldots,A_l)
=
\begin{cases}
\varepsilon_{1_X}(A_1),&(k,l)=(1,1),\\
(-1)^{|A_1||A_2|}P_K(A_1\smile A_2),&(k,l)=(0,2),\\
0,&\text{otherwise.}
\end{cases}
\]

\item \textup{(Divisor)}
If \(|A_l|=2\), then
\[
\ogw_{\beta,k}(A_1,\ldots,A_l)
=
\begin{cases}
P_K(A_1\smile A_2),
&(\beta,k,l)=(\beta_0,0,2),\\[2pt]
\displaystyle
\left(\int_\beta A_l\right)
\ogw_{\beta,k}(A_1,\ldots,A_{l-1}),
&\text{otherwise.}
\end{cases}
\]
\end{enumerate}
If \(\Coker\mathfrak i_Q=0\), then
\(P_K=0\), so the second exceptional terms in the unit
and zero axioms vanish.
\end{prop}

\begin{proof}
By \cref{thm:enhanced-superpotential-axioms},
\(|\overline\Omega|=3-n\).  Since \(|\partial_A|=|A|-2\),
comparing degrees in \eqref{eq:intro-open-invariants} proves the degree axiom.
For the unit axiom, differentiate
\(T^{\beta_0}(P_Q(\Gamma_W)+s)\) in
\(A_1,\ldots,A_{l-1}\) and set \(t=0\).
Only \(s\) for \(l=1\) and \(P_K(A_1)\) for \(l=2\) remain.
For the zero axiom, \(t_1s\) gives
\(\varepsilon_{1_X}(A_1)\) when \((k,l)=(1,1)\), while
\[
\partial_{A_2}\partial_{A_1}\tfrac12P_Q(\Gamma_W\smile\Gamma_W)
=(-1)^{|A_1||A_2|}P_K(A_1\smile A_2)
\]
gives the case \((k,l)=(0,2)\).
All other coefficients vanish.

For the divisor axiom, differentiate the divisor equation in
\(A_1,\ldots,A_{l-1}\), set \(t=0\), and compare coefficients.
Here \(E_{A_l}(T^\beta)=(\int_\beta A_l)T^\beta\), while
\(P_Q(A_l\smile\Gamma_W)\) gives
\(P_K(A_1\smile A_2)\) only when
\((\beta,k,l)=(\beta_0,0,2)\).
Finally, \(\Coker\mathfrak i_Q=(\Coker\mathfrak i_K)\otimes Q_W\),
so its vanishing implies \(P_K=0\).
\end{proof}

\subsection{Wall-crossing formula}
\label{subsec:wall-crossing-formula}

We now state the wall-crossing formula of
\cite[Theorem~6]{solomon2019relative} with our conventions.  Assume
that \(\mathfrak i_K\) is zero on cohomology, and choose
\(\theta_\diamond\in C_X^n\) such that
\[
    \mathfrak i_K(\theta_\diamond)=1_{K_L}.
\]
In the scalar de~Rham case, such a \(\theta_\diamond\) always exists,
and \(\mathfrak i_\R\) is zero on cohomology if and only if \([L]=0\).
Under these assumptions,
\(\Coker\mathfrak i_K=K_L\) and
\(\Coker\mathfrak i=HH_0(R_W)\).  Define
\[
    \Gamma_\diamond
    :=
    [d_X\theta_\diamond]
    =\overline y\bigl([1_{K_L}]\bigr)
    \in H^{n+1}(\ker\mathfrak i_K).
\]
Suppose that \(\Gamma_\diamond\in W\), and let \(t_\diamond\) be its
coordinate.  Choose a cocycle representative
\(\zeta_\diamond\in\ker i^*\) as in \cref{subsec:enhanced-axioms},
and use it in \(\gamma_W\), so that
\[
    \partial_{\Gamma_\diamond}\gamma_W=\zeta_\diamond,
    \qquad
    \partial_{\Gamma_\diamond}s=0.
\]
Thus
\[
    |\Gamma_\diamond|=n+1,
    \qquad
    |t_\diamond|=1-n=|s|.
\]
Since \(W\) has only even-degree classes, \(\Gamma_\diamond\in W\)
requires \(n=\dim L\) to be odd.  The following theorem therefore
assumes that \(n\) is odd.

\begin{thm}[Wall-crossing]
\label{thm:wall-crossing}
Suppose that \(\mathfrak i_K\) vanishes on cohomology, that \(W\) is
concentrated in even parity, that
\(1_{K_L}\in\operatorname{Im}(\mathfrak i_K\colon C_X^n\to K_L)\), that
\(n\) is odd, that \(b\) is point-like, and that \(\Gamma_\diamond\in W\).  Then,
for every \(\beta\in\Pi^{\geq0}\), \(k,l\geq0\), and homogeneous
\(A_1,\ldots,A_l\in W\), the enhanced open invariants satisfy
\[
    \ogw_{\beta,k+1}(A_1,\ldots,A_l)
    =
    \ogw_{\beta,k}
    (\Gamma_\diamond,A_1,\ldots,A_l).
\]
\end{thm}

\begin{proof}
Using the chosen lift \(\theta_\diamond\), in the cone one has
\[
    (d_X\theta_\diamond,0)
    =
    d_{\Cone}(\theta_\diamond,0)+(0,-1_{K_L}).
\]
It follows that the two natural images of the unit agree,
\[
    a(\Gamma_\diamond)
    =
    \overline x\bigl([1_{K_L}]\bigr)
    \qquad\text{in }H^*(\Cone{\mathfrak i}).
\]

The point-like equation
\(\langle1,b\rangle_L=-s\) gives
\(\partial_{\Gamma_\diamond}\langle1,b\rangle_L=0\).  Hence
\cref{lm:relative-potential-derivative}, and the definition of the tensor potential
give
\[
\begin{aligned}
    \partial_{\Gamma_\diamond}\Psi
    &=
    \nun^{\gamma_W,b}\bigl(a(\Gamma_\diamond)\bigr)\\
    &=
    \nun^{\gamma_W,b}
    \left(\overline x\bigl([1_{K_L}]\bigr)\right)
    =
    \overline x([-c]).
\end{aligned}
\]
The splitting \(P\) commutes with coefficient derivations and satisfies
\(P\overline x=\Id\), so
\[
    \partial_{\Gamma_\diamond}\overline\Omega=[-c].
\]
On the other hand, the \(s\)-equation in
\cref{thm:enhanced-superpotential-axioms}, applied to
\(\delta=\partial_s\), gives
\[
    \partial_s\overline\Omega=[-c]
    =\partial_{\Gamma_\diamond}\overline\Omega.
\]
Applying the derivations
\(\partial_{A_1},\ldots,\partial_{A_l}\), setting the bulk coordinates to zero,
and taking the coefficient of \(T^\beta s^k/k!\) gives
the result.
\end{proof}
\section{The action on \texorpdfstring{\(V\)}{V}, open WDVV,
and the boundary character}
\label{sec:noncommutative-relative-quantum}

For the remainder of the main text, assume that the bulk spaces \(U\) and
\(W\), the subspace \(U_L\), and the \(s\)-free coefficient algebras
\(Q_W\) and \(R_W^0\) are concentrated in even parity.  
\subsection{Setup}
\label{subsec:noncommutative-open-wdvv-setup}

Continue to use \(\rho\) from \eqref{eq:ker-i-exact-sequence} and the
chosen splitting \(P\), and hence \(P_K\), from
\cref{lm:enhanced-superpotential-choice-from-base}.

Use the closed quantum product and WDVV notation of
\cref{subsubsec:intro-closed-potential,subsubsec:intro-open-wdvv}, now for
the open--closed pair \((C_X,C_L)\) over \(\kk\).  Thus
\(U\subset H^*(C_X)\), \(\gamma_U\in C_X^{Q_U}\), and
\(\star_U\) is defined over the coefficient algebra
\(Q_U\) of \cref{dfn:filtered-coefficient-algebras}.

\begin{assum}[Closed Frobenius subalgebra]
\label[assum]{assum:closed-frobenius-sector}
Assume that \(1_X\in U\), that the restriction
\(\langle\,\cdot\,,\,\cdot\,\rangle_X|_U\) is nondegenerate, and that
\(U\otimes Q_U\) is closed under \(\star_U\).
\end{assum}

\begin{assum}[Bounding pair]
\label[assum]{assum:universal-bounding-pair}
Assume that the fixed data satisfy \(W=\rho^{-1}(U)\), that
\((\gamma_W,b)\) is a separated bounding pair over \(W\), as in
\cite[Assumption~(A.2)]{solomon2019relative}, and that
\(\gamma_W=\Gamma_W\).
\end{assum}

As in \cref{subsubsec:intro-open-wdvv}, put
\(U_L=U\cap\operatorname{Im}\rho\) and \(W'=\ker(P_K|_W)\).
The sequence \eqref{eq:ker-i-exact-sequence} before scalar extension,
together with
\(P_K\overline y=\Id\) and \(W=\rho^{-1}(U)\), gives
\[
 \rho(W)=U_L,
 \qquad
 W=W'\oplus\operatorname{Im}\overline y,
 \qquad
 \rho|_{W'}\colon W'\xrightarrow{\ \simeq\ }U_L,
\]
and
\[
 \operatorname{Im}\rho=\Ker\mathfrak i_K,
 \qquad
 U_L=U\cap\Ker\mathfrak i_K.
\]
With \(\mathcal A=U_L\otimes Q_W\) as in
\cref{subsubsec:intro-character}, put
\(V:=\pi^{-1}(U\otimes Q_W)\).

\begin{lm}
\label{lm:V-splitting}
\[
 \operatorname{Im}\overline x\oplus a(W'\otimes Q_W)=V.
\]
\end{lm}
\begin{proof}
This is the exact-sequence argument of
\cite[Lemma~5.11]{solomon2019relative}.  Since \(P\) splits the cone exact
sequence and \(\ker P=a(\ker P_Q)\),
\[
 H^*(\Cone{\mathfrak i})
 =\operatorname{Im}\overline x\oplus a(\ker P_Q).
\]
Moreover, \(\operatorname{Im}\overline x\subset V\), and
\[
 V\cap\ker P
 =a\!\left(\rho^{-1}(U\otimes Q_W)\cap\ker P_Q\right)
 =a(W'\otimes Q_W).
\]
Here we used \(\pi a=\rho\) and
\(P_Q=P_K\otimes\Id_{Q_W}\).  The result follows.
\end{proof}

The restriction of \(\pi\) to \(a(W'\otimes Q_W)\) is an isomorphism onto
\(\mathcal A\).  Hence
\[
 0\longrightarrow\Coker\mathfrak i
 \xrightarrow{\ \overline x\ }V
 \xrightarrow{\ \pi|_V\ }\mathcal A
 \longrightarrow0
\]
is split exact.  The chosen splitting \(P\) determines the section
\[
 \sigma
 :=a\circ(\rho|_{W'\otimes Q_W})^{-1}
 =(\pi|_{a(W'\otimes Q_W)})^{-1}.
\]
It satisfies
\[
 \pi|_V\circ\sigma=\Id_{\mathcal A},
 \qquad
 P\circ\sigma=0,
 \qquad
 V=\operatorname{Im}\overline x\oplus\sigma(\mathcal A).
\]

\begin{notn}[Coordinates on \(W\) and \(U\)]
\label[notn]{notn:closed-potential-coordinate-bases}
The coefficient maps \(\rho^*\) and \(\kappa^*\) are those of
\cref{subsubsec:intro-open-wdvv}; here
\(\ker(\rho|_W)=\operatorname{Im}\overline y\).
The latter map extends to \(R_W\) and \(\Coker\mathfrak i\).
Since \(U_L=U\cap\Ker\mathfrak i_K\), one has
\(\mathcal A=U\otimes Q_W\) when \(\mathfrak i_K\) vanishes on cohomology.

Set \(\gamma_{W'}=\kappa^*\gamma_W\).  Since
\(\rho|_{W'}\colon W'\to U_L\) is an isomorphism, choose a closed
bulk form \(\gamma_U\) extending \(\gamma_{W'}\), so that
\begin{equation}\label{eq:wdvv-compatible-W-representatives}
    \rho^*\gamma_U
    =
    \gamma_{W'}
    \qquad\text{in }C_X^{Q_W}.
\end{equation}

Use the compatible bases \(\{\Gamma_m\}_{m\in I_L}\) and
\(\{\Delta_l\}_{l\in I_U}\) fixed in \cref{subsubsec:intro-open-wdvv}.
For \(m\in I_L\), the coordinate \(t_m\) dual to \(\Gamma_m\) vanishes
on \(\operatorname{Im}\overline y\) and is the pullback under \(\rho^*\)
of the coordinate dual to \(\Delta_m\).  In the abstract pairing notation,
\[
    g_{lm}:=\langle\Delta_l,\Delta_m\rangle_X,
    \qquad
    (g^{lm}):=(g_{lm})^{-1}
    \qquad(l,m\in I_U);
\]
the inverse exists by \cref{assum:closed-frobenius-sector}.
The section \(\sigma\) satisfies
\(\sigma(\Delta_m)=a(\Gamma_m)\) for \(m\in I_L\).

Write \(\partial_l=\partial_{\Delta_l}\) on \(Q_U\) for \(l\in I_U\)
and, when \(l\in I_L\), also \(\partial_l=\partial_{\Gamma_l}\) on
\(Q_W\).

As in the Introduction, closed operations in the relative theory are
pulled back along \(\rho^*\), and \(\star=\rho^*(\star_U)\).  We leave
\(\rho^*\) visible on other closed tensors when their source matters.
Because \(U\) is a \(\kk\)-vector space, base change is free:
\begin{equation}\label{eq:closed-base-change}
    (U\otimes Q_U)\otimes_{Q_U,\rho^*}Q_W\simeq U\otimes Q_W ,
\end{equation}
so the pulled-back Frobenius algebra is \((U\otimes Q_W,\star)\).
\end{notn}

Via \(\rho|_{W'}\), identify \(\mathcal A\) with \(W'\otimes Q_W\) and
write \(\partial_u\) for the corresponding derivation.  For
\(u\in U_L\), one has \(\partial_u\rho^*=\rho^*\partial_u\).

\begin{lm}
\label{lm:closed-operations-subalgebra-U}
For closed \(\eta_1,\ldots,\eta_l\in C_X^{Q_W}\),
\[
    \left[
    \rho^*\q_{\emptyset,l}^{\gamma_U}
    (\eta_1,\ldots,\eta_l)
    \right]
    =
    \left[
    \q_{\emptyset,l}^{\gamma_W}
    (\eta_1,\ldots,\eta_l)
    \right]
    \in H^*(C_X^{Q_W}).
\]
\end{lm}
\begin{proof}
The proof of \cite[Lemma~5.1]{solomon2019relative} applies
to each class in \(\operatorname{Im}\overline y\), using
\eqref{eq:wdvv-compatible-W-representatives} and the closed
chain-map identity of \cref{lm:bg-deformed-axioms}.
\end{proof}

Recall that \(\overline u=(\rho\otimes\Id_{Q_W})(u_W)\in\mathcal A\)
for \(u=(u_W,x_u)\in(W\otimes Q_W)\oplus R_W\).

\begin{lm}
\label{lm:tensor-potential-preserves-V}
Let \(\partial_u\in\mathfrak g_{\mathrm{for}}\) be homogeneous.  Then
\[
    \partial_u\nun^{\gamma_W,b}(V)\subset V.
\]
Moreover, for \(v\in V\),
\[
    \pi\!\left(
    (\partial_u\nun^{\gamma_W,b})(v)
    \right)
    =
    \pi(v)\,\star\,\overline u,
\]
so the closed projection depends only on the \(W\)-component \(u_W\).
\end{lm}
\begin{proof}
This is the projection calculation of
\cite[Lemma~5.2]{solomon2019relative}.
Split \(\partial_u=\partial_{u_W}+\partial_{(0,x_u)}\).
Since \(\partial_{(0,x_u)}|_{R_W^0}=0\) while \(\gamma_W\) and the closed
component of the tensor potential have coefficients in \(Q_W\), the first
component of \(\partial_{(0,x_u)}\nun^{\gamma_W,b}\) vanishes.  Hence
\[
    \pi\!\left(
    (\partial_{(0,x_u)}\nun^{\gamma_W,b})(v)
    \right)
    =0\in U\otimes Q_W,
\]
so \((\partial_{(0,x_u)}\nun^{\gamma_W,b})(v)\in\ker\pi\subset V\).
Since \(\overline u\) depends only on \(u_W\), it remains to treat
\(\partial_{u_W}\), where \(u_W\in W\otimes Q_W\).
By graded \(Q_W\)-linearity it is enough to take \(u_W\in W\).
Let \((\eta,\xi)\) be a cone cocycle representing \(v\).  The first
component of \((\partial_{u_W}\nun^{\gamma_W,b})(v)\) gives
\[
\begin{aligned}
    \pi\!\left((\partial_{u_W}\nun^{\gamma_W,b})(v)\right)
    &=
    \left[
      \q_{\emptyset,2}^{\gamma_W}
      (\eta,\partial_{u_W}\gamma_W)
    \right] \\
    &=
    \pi(v)\,\star\,\overline u,
\end{aligned}
\]
where the second equality is
\cref{lm:closed-operations-subalgebra-U}.  Both factors on the right belong
to \(U\otimes Q_W\).  By \cref{assum:closed-frobenius-sector} and base
change \eqref{eq:closed-base-change}, their \(\star\)-product also belongs to
\(U\otimes Q_W\).  Hence
\((\partial_u\nun^{\gamma_W,b})(v)\in V\), and adding the two parts
proves the lemma. 
\end{proof}

For a homogeneous derivation \(\partial_\upsilon\in\mathfrak g_{\mathrm{for}}\),
define
\[
    T_\upsilon
    :=
    \left.H^*\!\left(\partial_\upsilon\nun^{\gamma_W,b}\right)\right|_V
    \colon V\longrightarrow V.
\]
Here \(\partial_\upsilon\nun^{\gamma_W,b}\) is the operator defined by
\eqref{eq:operator-derivative-convention}.  The map \(T_\upsilon\) is graded
\(Q_W\)-linear of degree \(|\partial_\upsilon|+2\).

\begin{cor}
\label{cor:liftable-quantum-subalgebra}
The submodule \(\mathcal A=U_L\otimes Q_W\) is a unital
\(Q_W\)-subalgebra of \((U\otimes Q_W,\star)\).
\end{cor}
\begin{proof}
For homogeneous \(a,\alpha\in\mathcal A\), choose \(v\in V\) with
\(\pi(v)=a\).
By \cref{lm:tensor-potential-preserves-V},
\[
    a\,\star\,\alpha=\pi(T_\alpha v)\in\pi(V)=\mathcal A.
\]
The unit has a relative lift, so \(1_X\in\mathcal A\).
\end{proof}

For the following two lemmas, we work in the geometric setting of
\cref{subsec:intro-geometric-background}.

For the closed potential \(\Phi_U\) of \eqref{eq:closed-potential},
the gradient of the full potential evaluated at \(\gamma_U\) is given
by \cite[Eq.~(5)]{solomon2019relative}:
\[
    \nabla\Phi_U
    =\left[\q_{\emptyset,0}^{\gamma_U}\right]
    =\sum_{k\geq0}\frac{1}{k!}
      \left[\q_{\emptyset,k}(\gamma_U^{\otimes k})\right]
    \in H^*(C_X)\otimes Q_U.
\]
For \(\alpha\in U\), it satisfies
\[
    \partial_\alpha\Phi_U
    =\left\langle\nabla\Phi_U,\alpha\right\rangle_X
    =\left\langle
       \left[\q_{\emptyset,0}^{\gamma_U}\right],\alpha
      \right\rangle_X.
\]
We use the pulled-back, \(Q_W\)-trilinear tensor
\(\Phi_{\alpha\beta\chi}\) introduced in
\cref{subsubsec:intro-open-wdvv}.

\begin{lm}\label{lm:closed-potential-product-pairing}
For \(\alpha,\beta,\chi\in U\otimes Q_W\), we have
\[
    \left\langle\beta\star\alpha,\chi\right\rangle_X
    =\Phi_{\alpha\beta\chi}
    =\left\langle\beta,\alpha\star\chi\right\rangle_X.
\]
\end{lm}
\begin{proof}
This is the closed-component calculation in the proof of
\cite[Lemma~5.12]{solomon2019relative}, pulled back along \(\rho^*\)
and extended \(Q_W\)-trilinearly.  The second equality follows from
symmetry of the closed potential tensor and of the pairing.
\end{proof}

Let \(\mathcal A^\perp\) denote the orthogonal complement of
\(\mathcal A\) in \(U\otimes Q_W\) with respect to
\(\langle\,\cdot\,,\,\cdot\,\rangle_X\).

\begin{lm}\label{lm:liftable-orthogonal-submodule}
\(\mathcal A^\perp\) is an \(\mathcal A\)-submodule of
\((U\otimes Q_W,\star)\).
If \(p=\operatorname{PD}[L]\ne0\) lies in \(U\), then
\(\mathcal A^\perp=Q_Wp\), and there is a unique unital
degree-zero \(Q_W\)-algebra homomorphism
\[
 \lambda\colon\mathcal A\longrightarrow Q_W
\]
such that \(p\star u=\lambda(u)p\) for all \(u\in\mathcal A\).
\end{lm}
\begin{proof}
Let \(x\in\mathcal A^\perp\) and \(u,v\in\mathcal A\).
By \cref{cor:liftable-quantum-subalgebra,lm:closed-potential-product-pairing},
\[
 \langle x\star u,v\rangle_X
 =\langle x,u\star v\rangle_X=0.
\]
Thus \(x\star u\in\mathcal A^\perp\).

We have \(\mathfrak i_\R(a)=\langle p,a\rangle_X\) for \(a\in U\).  Since \(U_L=U\cap\Ker\mathfrak i_\R\), we have
\[
 \mathcal A=(Q_Wp)^\perp.
\]
Nondegeneracy of the pairing on \(U\) gives
\(\mathcal A^\perp=Q_Wp\).
Since \(p\ne0\), for each \(u\in\mathcal A\) there is a unique
\(\lambda(u)\in Q_W\) such that \(p\star u=\lambda(u)p\).
For \(u,v\in\mathcal A\), associativity gives
\[
 \lambda(u\star v)p
 =p\star(u\star v)
 =(p\star u)\star v
 =\lambda(u)\lambda(v)p.
\]
Hence \(\lambda(u\star v)=\lambda(u)\lambda(v)\).
The unit, linearity, and degree follow from those of \(\star\).
\end{proof}

\subsection{The action of the tensor potential on
\texorpdfstring{\(V\)}{V}}

The following facts replace the commutative scalar calculation in
\cite[Lemma~5.8]{solomon2019relative}.

\begin{lm}
\label{lm:mc-coefficient-actions-coker}
Let \(c\in\mathcal Z(R_W)\) be the Maurer--Cartan coefficient of
\((\gamma_W,b)\), and let
\(\partial_u\in\mathfrak g_{\mathrm{for}}\) be homogeneous.
Then \(\partial_uc\in\mathcal Z(R_W)\), and
\[
 c\cdot\operatorname{Im}\mathfrak i\subset\operatorname{Im}\mathfrak i,
 \qquad
 (\partial_uc)\cdot\operatorname{Im}\mathfrak i
 \subset\operatorname{Im}\mathfrak i.
\]
Thus multiplication by \(c\) and \(\partial_uc\) is well defined on
\(\Coker\mathfrak i\).
\end{lm}
\begin{proof}
Since \(c\in\mathcal Z(R_W)\) and \(|c|=2\), we have \(ca=ac\)
for every homogeneous \(a\in R_W\).  Applying \(\partial_u\) gives
\[
 (\partial_uc)a+c\,\partial_ua
 =(\partial_ua)c+(-1)^{|\partial_u||a|}a(\partial_uc).
\]
The terms \(c\,\partial_ua\) and \((\partial_ua)c\) agree.  Hence
\((\partial_uc)a=(-1)^{|\partial_u||a|}a(\partial_uc)\), so
\(\partial_uc\in\mathcal Z(R_W)\).

For closed \(\eta\in C_X^{Q_W}\),
\cref{lm:tensor-potential-second-component} gives
\[
    -c\cdot\mathfrak i(\eta)
    =
    \mathfrak i\!\left(
    \q_{\emptyset,1}^{\gamma_W}(\eta)
    \right).
\]
Thus \(c\cdot\operatorname{Im}\mathfrak i\subset\operatorname{Im}\mathfrak i\).
By \cref{lm:coefficient-actions-cone},
\(\partial_u(\operatorname{Im}\mathfrak i)\subset\operatorname{Im}\mathfrak i\).
For \(r\in\operatorname{Im}\mathfrak i\), we therefore have
\[
 (\partial_uc)\cdot r
 =\partial_u(c\cdot r)-c\cdot\partial_ur
 \in\operatorname{Im}\mathfrak i.
\]
\end{proof}

\begin{lm}
\label{lm:tensor-potential-coker-action}
For every \(r\in\Coker\mathfrak i\) and every homogeneous
\(\partial_u\in\mathfrak g_{\mathrm{for}}\),
\[
    \nun^{\gamma_W,b}\bigl(\overline x(r)\bigr)
    =
    \overline x(-c\cdot r),
    \qquad
    \bigl(\partial_u\nun^{\gamma_W,b}\bigr)
    \bigl(\overline x(r)\bigr)
    =
    \overline x\bigl(-(\partial_uc)\cdot r\bigr).
\]
\end{lm}
\begin{proof}
If \(\widetilde r\in HH_0(R_W)\) represents \(r\), then
\(\overline x(r)\) is represented by \((0,-\widetilde r)\), and the
definition of the tensor potential gives
\[
    \nun^{\gamma_W,b}(0,-\widetilde r)
    =(0,c\widetilde r)
    =\overline x(-c\cdot r).
\]
The action on the right is well defined by
\cref{lm:mc-coefficient-actions-coker}.

For the second identity, \eqref{eq:operator-derivative-convention} and
\(\lvert\nun^{\gamma_W,b}\rvert=2\) give
\(\partial_u\nun^{\gamma_W,b}
=\partial_u\circ\nun^{\gamma_W,b}-\nun^{\gamma_W,b}\circ\partial_u\).
Since \(\overline x(r)\) is represented by \((0,-\widetilde r)\), we have
\(\partial_u\overline x(r)=\overline x(\partial_ur)\).  Hence
\[
\begin{aligned}
    \bigl(\partial_u\nun^{\gamma_W,b}\bigr)
    \bigl(\overline x(r)\bigr)
    &=
    \partial_u\!\left(\overline x(-c\cdot r)\right)
    -
    \nun^{\gamma_W,b}\!\left(\overline x(\partial_ur)\right)\\
    &=
    \overline x\!\left(
        -(\partial_uc)\cdot r-c\cdot\partial_ur
    \right)
    -
    \overline x\!\left(-c\cdot\partial_ur\right)\\
    &=
    \overline x\!\left(-(\partial_uc)\cdot r\right).
\end{aligned}
\]
Here \(|c|=2\),
and multiplication by \(\partial_uc\) is well defined on the cokernel by
\cref{lm:mc-coefficient-actions-coker}.
\end{proof}

\begin{rem}
This lemma is the non-commutative analogue of
\cite[Lemma~5.8]{solomon2019relative}.  In our setting, the calculation uses
\(c,\partial_uc\in\mathcal Z(R_W)\) and the two inclusions in
\cref{lm:mc-coefficient-actions-coker}.
\end{rem}

Define
\[
    \chi_\upsilon\colon\Coker\mathfrak i\longrightarrow\Coker\mathfrak i,
    \qquad
    \chi_\upsilon(r):=-(\partial_\upsilon c)\cdot r,
\]
and, for \(\alpha\in\mathcal A\), let
\(m_\alpha\colon\mathcal A\to\mathcal A\) be given by
\(m_\alpha(a):=a\star\alpha\).  This is well defined by
\cref{cor:liftable-quantum-subalgebra}.
The maps \(\chi_\upsilon\) and \(m_{\overline\upsilon}\) are graded
\(Q_W\)-linear of degree
\(|\partial_\upsilon|+2\).

\begin{prop}[Triangular block form]
\label[prop]{prop:tensor-potential-block-form}
Let \(\partial_\upsilon\in\mathfrak g_{\mathrm{for}}\) be homogeneous and set
\(j_\upsilon:=P\circ T_\upsilon\circ\sigma\colon\mathcal A\to\Coker\mathfrak i\).
Relative to \(V=\operatorname{Im}\overline x\oplus\sigma(\mathcal A)\),
\[
T_\upsilon=
\begin{pmatrix}
\chi_\upsilon&j_\upsilon\\
0&m_{\overline\upsilon}
\end{pmatrix}.
\]
\end{prop}
\begin{proof}
Every \(v\in V\) decomposes as
\[
    v=\overline x(Pv)+\sigma(\pi v),
\]
since \(P\overline x=\Id\), \(P\sigma=0\), \(\pi\overline x=0\) and
\(\pi\sigma=\Id\).
By \cref{lm:tensor-potential-coker-action},
\(T_\upsilon\overline x(r)=\overline x(\chi_\upsilon(r))\), which is the
first column.
For \(\alpha\in\mathcal A\), \cref{lm:tensor-potential-preserves-V} gives
\(\pi T_\upsilon\sigma(\alpha)=\alpha\star\overline\upsilon\), so
\[
    T_\upsilon\sigma(\alpha)
    =\overline x\bigl(j_\upsilon(\alpha)\bigr)
    +\sigma\bigl(m_{\overline\upsilon}(\alpha)\bigr),
\]
which is the second column.
\end{proof}

\subsection{Open WDVV}
\label{subsec:open-wdvv}

\begin{thm}[Open WDVV in block form]
\label[thm]{thm:structural-open-wdvv}
Let \(\partial_u,\partial_v\in\mathfrak g_{\mathrm{for}}\) be homogeneous
derivations.  Then
\[
    \chi_u\circ j_v+j_u\circ m_{\overline v}
    =
    (-1)^{|\partial_u||\partial_v|}
    \bigl(\chi_v\circ j_u+j_v\circ m_{\overline u}\bigr)
\]
as an identity of graded \(Q_W\)-linear maps
\(\mathcal A\to\Coker\mathfrak i\) of degree
\(|\partial_u|+|\partial_v|+4\).
\end{thm}
\begin{proof}
By \cref{thm:flatness-relation},
\[
    T_uT_v
    =
    (-1)^{|\partial_u||\partial_v|}T_vT_u
    \qquad\text{on }V,
\]
since \(|T_\upsilon|=|\partial_\upsilon|+2\).
By \cref{prop:tensor-potential-block-form},
\[
    P T_uT_v\sigma
    =\chi_u\circ j_v+j_u\circ m_{\overline v}.
\]
Composing the commutation identity with \(P\) and \(\sigma\) gives the result.
\end{proof}

Abbreviate \(\overline\Omega=\overline\Omega(\gamma_W,b)\).
We work in the geometric setting of
\cref{subsec:intro-geometric-background}.

\begin{lm}\label{lm:enhanced-potential-second-derivative}
Let \(w\in W\otimes1\) and let
\(\partial_v\in\mathfrak g_{\mathrm{for}}\) be homogeneous.
Then
\[
    P T_v(a(w))=\partial_v\partial_w\overline\Omega.
\]
\end{lm}
\begin{proof}
Since \(\partial_w\gamma_W=w\) for the chosen cocycle representative
of \(w\), and the bounding pair is separated,
\cref{lm:relative-potential-derivative} gives
\[
    \partial_w\Psi
    =\nun^{\gamma_W,b}(a(w)).
\]
Since \(w\) is constant, \(\partial_v a(w)=0\).  Using
\(\overline\Omega=P\Psi\) and
\cref{lm:splitting-commutes-derivations}, we obtain
\[
\begin{aligned}
    P T_v(a(w))
    &=P\!\left(\partial_v\bigl(\nun^{\gamma_W,b}(a(w))\bigr)\right)\\
    &=P(\partial_v\partial_w\Psi)
      =\partial_v\partial_w(P\Psi)
      =\partial_v\partial_w\overline\Omega.
\end{aligned}
\]
\end{proof}

The following is the analogue of \cite[Lemma~5.12]{solomon2019relative}
for the cokernel formed from \(HH_0(R_W)\).
\begin{lm}\label{lm:tensor-potential-coordinate-decomposition}
Let \(w\in W\otimes1\) and let
\(\partial_v\in\mathfrak g_{\mathrm{for}}\) be homogeneous.
Then
\[
    T_v(a(w))
    =\overline x\!\left(\partial_v\partial_w\overline\Omega\right)
    +\sum_{\substack{l\in I_U\\m\in I_L}}
      \Phi_{\overline v\,\overline w\,\Delta_l}\,
      g^{lm}\,a(\Gamma_m).
\]
\end{lm}
\begin{proof}
By \cref{lm:tensor-potential-preserves-V}, \(T_v(a(w))\in V\).
Write its decomposition as
\[
    T_v(a(w))=\overline x(r)+\sum_{m\in I_L}r^m a(\Gamma_m),
    \qquad r\in\Coker\mathfrak i,\quad r^m\in Q_W.
\]
Applying \(P\) to this decomposition gives \(r=PT_v(a(w))\), so
\(r=\partial_v\partial_w\overline\Omega\) by
\cref{lm:enhanced-potential-second-derivative}.

Applying \(\pi\) to the decomposition and pairing with \(\Delta_l\),
for \(l\in I_U\), we obtain from \cref{lm:tensor-potential-preserves-V} and
\cref{lm:closed-potential-product-pairing}
\[
    \sum_{m\in I_L}r^m g_{ml}
    =\left\langle\pi T_v(a(w)),\Delta_l\right\rangle_X
    =\left\langle\overline w\star\overline v,\Delta_l\right\rangle_X
    =\Phi_{\overline v\,\overline w\,\Delta_l}.
\]
Multiplication by the inverse metric therefore gives
\[
    r^m=\sum_{l\in I_U}
    \Phi_{\overline v\,\overline w\,\Delta_l}\,g^{lm},
\]
which proves the lemma.
\end{proof}

\begin{proof}[Proof of \cref{thm:intro-open-wdvv}]
We follow Solomon--Tukachinsky's proof of Theorem~3 in
\cite{solomon2019relative}.  We compute
\(P T_uT_v(a(w))\) and \(P T_vT_u(a(w))\), then use
\cref{thm:flatness-relation} to relate them.

Since \(T_u\) is graded \(Q_W\)-linear of degree
\(|\partial_u|+2\) and \(P\) has degree zero, for homogeneous \(q\in Q_W\)
we have
\[
    P T_u\bigl(q\,a(\Gamma_m)\bigr)
    =(-1)^{|\partial_u||q|}q\,P T_u(a(\Gamma_m)).
\]
By \cref{lm:tensor-potential-coordinate-decomposition,lm:tensor-potential-coker-action},
we compute
\[
\begin{aligned}
    P T_uT_v(a(w))
    &=- (\partial_uc)\cdot\partial_v\partial_w\overline\Omega
      \\
    &\quad+\sum_{\substack{l\in I_U\\m\in I_L}}
       (-1)^{|\partial_u||\Phi_{\overline v\,\overline w\,\Delta_l}|}
       \Phi_{\overline v\,\overline w\,\Delta_l}\,
       g^{lm}\,P T_u(a(\Gamma_m))\\
    &=- (\partial_uc)\cdot\partial_v\partial_w\overline\Omega
      \\
    &\quad+\sum_{\substack{l\in I_U\\m\in I_L}}
       (-1)^{|\partial_u||\Phi_{\overline v\,\overline w\,\Delta_l}|}
       \Phi_{\overline v\,\overline w\,\Delta_l}\,
       g^{lm}\,\partial_u\partial_{\Gamma_m}\overline\Omega.
\end{aligned}
\]
The sign in the first equality comes from taking the coefficient
\(\Phi_{\overline v\,\overline w\,\Delta_l}\in Q_W\) outside \(T_u\),
using its graded \(Q_W\)-linearity.
The second equality uses
\cref{lm:enhanced-potential-second-derivative} with \(w=\Gamma_m\).
Similarly,
\[
\begin{aligned}
    P T_vT_u(a(w))
    &=- (\partial_vc)\cdot\partial_u\partial_w\overline\Omega\\
    &\quad+\sum_{\substack{l\in I_U\\m\in I_L}}
       (-1)^{|\partial_v||\Phi_{\overline u\,\overline w\,\Delta_l}|}
       \Phi_{\overline u\,\overline w\,\Delta_l}\,
       g^{lm}\,\partial_v\partial_{\Gamma_m}\overline\Omega.
\end{aligned}
\]
By \cref{thm:flatness-relation}, on cohomology we have
\[
    P T_uT_v(a(w))
    =(-1)^{|\partial_u||\partial_v|}P T_vT_u(a(w)).
\]
Comparing the two calculations proves the asserted equation: since
\(Q_W\) is even, the coefficient signs are trivial, and symmetry of
\(\Phi\) and \(g^{lm}\) gives \eqref{eq:intro-open-wdvv}.
\end{proof}

\subsection{The boundary open character}
\label{sec:open-character}

For the remainder of this section, impose
\cref{assum:point-like-bounding-pair} on the bounding pair of
\cref{assum:universal-bounding-pair}.  Fix such a point-like pair
\((\gamma_W,b)\), and let \(c\) be its Maurer--Cartan coefficient.
Assume also the remaining hypotheses of \cref{subsec:enhanced-axioms},
so that its enhanced equations apply.

Evaluation at \(s=0\) means the projection \(R_W\to R_W^0\) from
\eqref{eq:s-free-relative-coefficients}; the \(s\)-degree projections and
the identification in \(s\)-degree one are those of
\cref{lm:s-degree-decomposition,lm:s-degree-one-detector}.
For the Maurer--Cartan coefficient,
\([s]c:=\left.(\partial_sc)\right|_{s=0}\in R_W^0\) denotes its
\(s\)-linear coefficient.

\begin{lm}
\label{lm:boundary-euler-block}
For \(u,\alpha\in \mathcal A\),
\[
 j_{\mathcal E_s}(\alpha\,\star\,u)
 =
 -(\partial_uc)\cdot j_{\mathcal E_s}(\alpha)
 +(\mathcal E_sc)\cdot j_u(\alpha).
\]
\end{lm}
\begin{proof}
For \(\mathcal E_s\) and \(u\), the four blocks in
\cref{thm:structural-open-wdvv} are
\[
 \overline{\mathcal E_s}=0,
 \qquad
 m_{\overline{\mathcal E_s}}=m_0=0,
 \qquad
 m_u(\alpha)=\alpha\,\star\,u,
\]
and
\[
 \chi_{\mathcal E_s}(r)=-(\mathcal E_sc)\cdot r,
 \qquad
 \chi_u(r)=-(\partial_uc)\cdot r.
\]
Substitution gives the result.
\end{proof}

\begin{lm}
\label{lm:boundary-euler-unit}
One has
\[
 j_{\mathcal E_s}(1_X)=[s].
\]
\end{lm}
\begin{proof}
Observe that \(P_K(1_X)=0\); hence \(1_X\in W'\) and
\(\sigma(1_X)=a(1_X)\).  Therefore
\[
\begin{aligned}
 j_{\mathcal E_s}(1_X)
 &=P T_{\mathcal E_s}\sigma(1_X)\\
 &=P\bigl(\mathcal E_s\nun^{\gamma_W,b}\bigr)a(1_X)\\
 &=P\,\mathcal E_s\bigl(\nun^{\gamma_W,b}a(1_X)\bigr)\\
 &=P\,\mathcal E_s(\partial_1\Psi)\\
 &=\mathcal E_s(\partial_1\overline\Omega)=[s].
\end{aligned}
\]
The fourth equality follows from
\cref{lm:relative-potential-derivative}, since
\(\partial_1\gamma_W=1_X\) and
\(\partial_1\langle1,b\rangle_L=0\).
The last equality is the enhanced unit equation in
\cref{thm:enhanced-superpotential-axioms}
\textup{\ref{item:enhanced-unit}}.
\end{proof}

Recall the map \(\mathcal L(u)=-\kappa^*(\partial_uc_0)\) from
\cref{subsubsec:intro-character}, where \(c_0=c|_{s=0}\).

\begin{thm}[The open character]
\label{thm:open-character}
If
\[
 \kappa^*\!\left(\left.(\partial_sc)\right|_{s=0}\right)=0,
\]
then
\[
 \mathcal L\colon
 \bigl(\mathcal A,\star\bigr)\longrightarrow\kappa^*R_W^0
\]
is a unital degree-zero character, and
\[
 \mathcal L(qu)=\kappa^*(q)\mathcal L(u)
 \qquad(q\in Q_W).
\]
\end{thm}
\begin{proof}
Put
\[
 J(\alpha):=\bigl[\kappa^*j_{\mathcal E_s}(\alpha)\bigr]_{s^1}.
\]
The hypothesis gives \(\kappa^*([s]c)=0\); hence
\[
 \kappa^*\!\left(\partial_u([s]c)\right)
 =\partial_{\kappa^*u}\!\left(\kappa^*([s]c)\right)=0.
\]
Applying \(\kappa^*\) to \cref{lm:boundary-euler-block} and taking \(s\)-degree one gives
\[
 J(\alpha\star u)=\mathcal L(u)J(\alpha).
\]
By \cref{lm:boundary-euler-unit}, \(J(1_X)=[s]\).  Therefore for every $u\in \mathcal A$,
\[
 J(u)=\mathcal L(u)[s].
\]
For \(u,v\in\mathcal A\), it follows that
\[
 \mathcal L(u\star v)[s]
 =J(v\star u)=\mathcal L(u)J(v)=\mathcal L(u)\mathcal L(v)[s].
\]
The result follows from \cref{lm:s-degree-one-detector}; linearity and degree
are immediate from the definition.
\end{proof}

\begin{rem}
\label{rem:sheridan-comparison}
\Cref{fig:sheridan-character-comparison} compares \cref{thm:open-character}
with Sheridan's proof that the closed--open map is
multiplicative~\cite[Proposition~2.1]{sheridan2016fano}.

\begin{figure}[htbp]
  \centering
  \includegraphics[width=\linewidth]{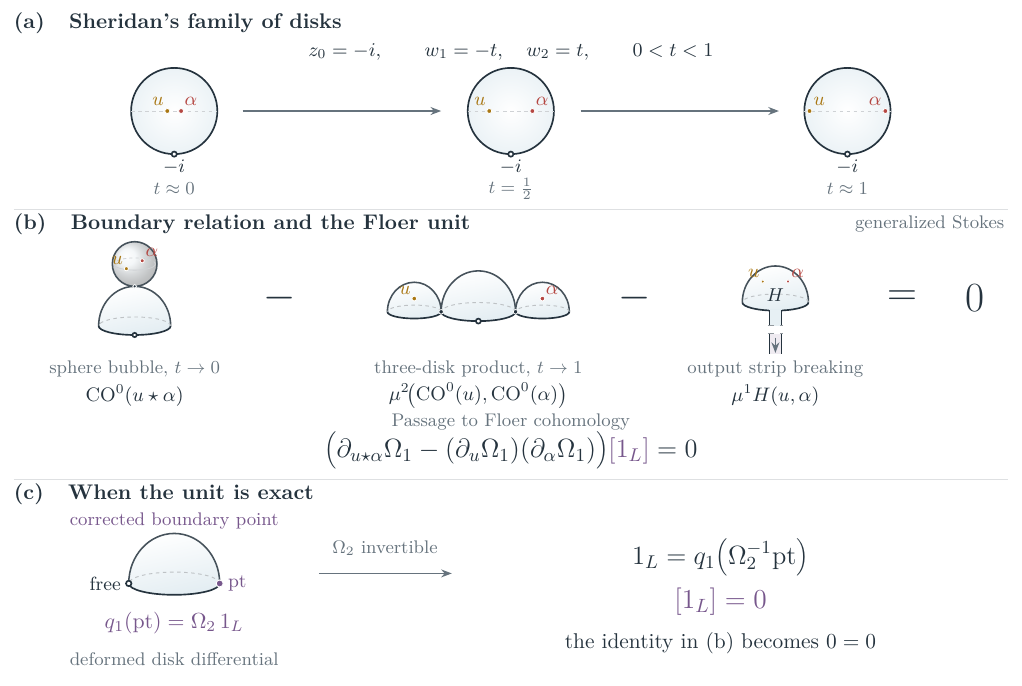}
  \caption{%
    Panels (a) and (b) show Sheridan's proof that the closed--open map is
    multiplicative~\cite[Proposition~2.1 and Figure~2]{sheridan2016fano}.
    The disks in (a) have an output at $-i$ and interior inputs $u$ and
    $\alpha$ at $-t$ and $t$.
    As $t\to0$, a sphere carrying $u$ and $\alpha$ bubbles off; as $t\to1$,
    the disk breaks into three disks; and a strip can break off at the output.
    After deformation by the bounding cochain, the resulting relation in Floer
    cohomology is the identity
    $\mathcal L(u\star\alpha)=\mathcal L(u)\mathcal L(\alpha)$ of
    \cref{thm:open-character} multiplied by the class $[1_L]$ of the unit.
    In (c), the cochain $\mathrm{pt}=(-1)^n\kappa^*((\partial_sb)|_{s=0})$
    satisfies $\q_1(\mathrm{pt})=\Omega_2\,1_L$, as in the proof of
    \cref{lm:point-like-floer-annihilation}.
    So, if $\Omega_2$ is invertible, then $[1_L]=0$ and the relation in (b)
    holds trivially.
    It is instructive to compare this figure with
    \cref{fig:intro-character-strata} in the overview, which illustrates the
    proof of \cref{thm:open-character} in geometric terms.}
  \label{fig:sheridan-character-comparison}
\end{figure}
\end{rem}

\subsubsection{Extensions of the boundary character}
\label{subsubsec:boundary-character-extensions}

In this subsubsection, work in the scalar de~Rham setting of
\cref{subsec:intro-geometric-background}.  Suppose that $n$ is even and
\[
 H^*(L;\R)\cong H^*(S^n;\R).
\]
Take $U=H^{\mathrm{even}}(X;\R)$ and $p=\operatorname{PD}[L]$.
By \cite[Lemma~5.6]{hugtenburg2024examples},
\[
 \langle p,p\rangle_X=2(-1)^{n/2}.
\]
Thus $\mathfrak i_\R$ is onto on cohomology, so $P_\R=0$, and
\[
 U_L=p^\perp,\qquad W'=W,\qquad\kappa^*=\Id,\qquad
 U\otimes Q_W=\mathcal A\oplus Q_Wp.
\]
Choose the bases of \cref{notn:closed-potential-coordinate-bases}
with $\Delta_L=p$ and $I_U=I_L\sqcup\{L\}$.  Then
\[
 g^{iL}=0\quad(i\in I_L),\qquad
 g^{LL}=\frac{(-1)^{n/2}}2,
\]
and $\rho^*$ sets $t_L=0$.
By~\cite[Proposition~6.1\textup{(1)}]
{kosloff2025pointlikenoncommutativef}, the bulk form
$\gamma_W=\Gamma_W$ admits a bounding cochain $b$ with $\int_Lb=s$.
Thus $\langle1,b\rangle_L=-s$, so $(\gamma_W,b)$ is point-like.
Fix such a pair and let $c$ be its Maurer--Cartan coefficient.
Since $R_W^0$ is even and $|s|$ is
odd, the degree-two coefficient $c$ has only even powers of $s$.
Hence \cref{thm:open-character} gives the character
$\mathcal L(u)=-\partial_uc_0$ on $\mathcal A$.

\begin{lm}
\label{lm:boundary-primitive-products}
The character $\lambda$ of \cref{lm:liftable-orthogonal-submodule}
equals $\mathcal L$.  In particular, $\mathcal L$ takes values in
$Q_W$, and for $u\in\mathcal A$,
\[
\begin{aligned}
 p\star u&=\mathcal L(u)p,\\
 p\star p&=2(-1)^{n/2}
 \sum_{l,m\in I_L}g^{lm}\mathcal L(\Delta_l)\Delta_m
 +g^{LL}\Phi_{ppp}p.
\end{aligned}
\]
\end{lm}
\begin{proof}
We first prove $\lambda(u)=-\partial_uc_0$ for $u\in\mathcal A$.
Represent $p$ by a fixed closed form, also denoted by $p$.
Taking $\eta=p$ and $s=0$ in
\cref{lm:tensor-potential-second-component} gives
\[
 \mathfrak i_Q\!\left([\q_{\emptyset,1}^{\gamma_W}(p)]\right)
 =-c_0\mathfrak i_Q(p).
\]
Write $u=\sum_{i\in I_L}f_i\Delta_i$.  Then
$\partial_u=\sum_{i\in I_L}f_i\partial_{t_i}$ and
$[\partial_u\gamma_W]=u$.  Since $p$ is independent of the $t_i$,
the deformation formula in \cref{dfn:bg-deformation-qempty} gives
\[
\begin{aligned}
 \partial_u\q_{\emptyset,1}^{\gamma_W}(p)
 &=\partial_u\sum_{l\ge0}\frac1{l!}
   \q_{\emptyset,l+1}(p,\gamma_W^{\otimes l})\\
 &=\sum_{l\ge1}\frac{l}{l!}
   \q_{\emptyset,l+1}
   (p,\partial_u\gamma_W,\gamma_W^{\otimes(l-1)})\\
 &=\q_{\emptyset,2}^{\gamma_W}(p,\partial_u\gamma_W).
\end{aligned}
\]
The factor $l$ comes from differentiating each of the $l$ copies of
$\gamma_W$; \cref{ax:qempty-symmetry} makes these $l$ terms equal.
By the definition of $\star$ in
\cref{subsubsec:intro-closed-potential} and
\cref{lm:closed-operations-subalgebra-U},
\[
 [\q_{\emptyset,2}^{\gamma_W}(p,\partial_u\gamma_W)]
 =p\star u.
\]
By \cref{lm:coefficient-actions-cone}, the even derivation
$\partial_u$ commutes with $\mathfrak i_Q$ on cohomology.
Moreover, $\partial_u\mathfrak i_Q(p)=0$, since $p$ is fixed.
Applying $\partial_u$ to
$\mathfrak i_Q([\q_{\emptyset,1}^{\gamma_W}(p)])
=-c_0\mathfrak i_Q(p)$ and using
$p\star u=\lambda(u)p$ from
\cref{lm:liftable-orthogonal-submodule} therefore gives
\[
 \lambda(u)\mathfrak i_Q(p)
 =\mathfrak i_Q(p\star u)
 =-(\partial_uc_0)\mathfrak i_Q(p).
\]
Since $\mathfrak i_Q(p)=2(-1)^{n/2}$, it follows that
$\lambda(u)=\mathcal L(u)\in Q_W$.

For $v\in\mathcal A$, \cref{lm:closed-potential-product-pairing} gives
\[
 \langle p\star p,v\rangle_X
 =\langle p,p\star v\rangle_X
 =2(-1)^{n/2}\mathcal L(v).
\]
Together with $\langle p\star p,p\rangle_X=\Phi_{ppp}$, this
determines $p\star p$.
\end{proof}

Define
\begin{equation}\label{eq:boundary-character-discriminant-coefficient}
\begin{aligned}
 D&:=\mathcal L\bigl(p\star p-g^{LL}\Phi_{ppp}p\bigr)\\
  &=2(-1)^{n/2}\sum_{l,m\in I_L}
       g^{lm}(\partial_lc_0)(\partial_mc_0)\in Q_W.
\end{aligned}
\end{equation}

\begin{lm}
\label{lm:boundary-three-constraint-equation}
In $R_W^0$, one has
\[
 D\,\partial_i\bigl([s^3/3!]\overline\Omega\bigr)
 =\bigl([s^3/3!]\overline\Omega\bigr)\partial_iD,
 \qquad i\in I_L.
\]
\end{lm}
\begin{proof}
Since $W'=W$, we have $\kappa^*=\Id$ and
$\mathcal L(u)=-\partial_uc_0$, where $c_0=c|_{s=0}$.
Put
\begin{align*}
 \alpha&:=p\star p-g^{LL}\Phi_{ppp}p\\
 &\overset{\ref{lm:boundary-primitive-products}}{=}
 2(-1)^{n/2}\sum_{l,m\in I_L}
 g^{lm}\mathcal L(\Delta_l)\Delta_m\in\mathcal A.
\end{align*}
For $i\in I_L$, associativity gives
\begin{align*}
 \alpha\star\Delta_i
 &=p\star(p\star\Delta_i)-g^{LL}\Phi_{ppp}(p\star\Delta_i)\\
 &\overset{\ref{lm:boundary-primitive-products}}{=}
 \mathcal L(\Delta_i)\alpha
 =-(\partial_ic_0)\alpha.
\end{align*}
Since $\partial_ic_0\in Q_W$ and $j_{\mathcal E_s}$ is $Q_W$-linear,
\begin{align}
 \partial_i(c-c_0)\cdot j_{\mathcal E_s}(\alpha)
 &=(\partial_ic)\cdot j_{\mathcal E_s}(\alpha)
   +j_{\mathcal E_s}(\alpha\star\Delta_i)\notag\\
 &\overset{\ref{lm:boundary-euler-block}}{=}
 (\mathcal E_sc)\cdot j_{\Delta_i}(\alpha).
 \label{eq:boundary-extension-three-constraint-wdvv}
\end{align}

Substituting the expression for $\alpha$ into the $Q_W$-linear map
$j_{\mathcal E_s}$ and using $\sigma(\Delta_m)=a(\Gamma_m)$, we compute
\begin{align*}
 [j_{\mathcal E_s}(\alpha)]_{s^1}
 &\overset{\ref{lm:enhanced-potential-second-derivative}}{=}
 2(-1)^{n/2}\sum_{l,m\in I_L}g^{lm}\mathcal L(\Delta_l)
 [\mathcal E_s\partial_m\overline\Omega]_{s^1}\\
 &\overset{\ref{thm:enhanced-superpotential-axioms}\textup{\ref{item:enhanced-s-equation}}}{=}
 -2(-1)^{n/2}\sum_{l,m\in I_L}
 g^{lm}\mathcal L(\Delta_l)(\partial_mc_0)[s]\\
 &=2(-1)^{n/2}\sum_{l,m\in I_L}
 g^{lm}\mathcal L(\Delta_l)\mathcal L(\Delta_m)[s]\\
 &\overset{\ref{eq:boundary-character-discriminant-coefficient}}{=}D[s].
\end{align*}
Similarly, $Q_W$-linearity of $j_{\Delta_i}$ gives
\begin{align*}
 [j_{\Delta_i}(\alpha)]_{s^1}
 &\overset{\ref{lm:enhanced-potential-second-derivative}}{=}
 2(-1)^{n/2}\sum_{l,m\in I_L}g^{lm}\mathcal L(\Delta_l)
 [\partial_i\partial_m\overline\Omega]_{s^1}\\
 &\overset{\ref{thm:enhanced-superpotential-axioms}\textup{\ref{item:enhanced-s-equation}}}{=}
 2(-1)^{n/2}\sum_{l,m\in I_L}
 g^{lm}\mathcal L(\Delta_l)\partial_i\mathcal L(\Delta_m)[s]\\
 &=(-1)^{n/2}\partial_i\!\left(
 \sum_{l,m\in I_L}g^{lm}\mathcal L(\Delta_l)\mathcal L(\Delta_m)
 \right)[s]\\
 &\overset{\ref{eq:boundary-character-discriminant-coefficient}}{=}
 \tfrac12(\partial_iD)[s].
\end{align*}
The third equality follows from the Leibniz rule, since $g^{lm}$ is
constant and symmetric.

Since $c$ has only even powers of $s$,
\begin{align*}
 -\tfrac12D\,\partial_i\bigl([s^3/3!]\overline\Omega\bigr)[s^3]
 &\overset{\ref{thm:enhanced-superpotential-axioms}\textup{\ref{item:enhanced-s-equation}}}{=}
 D\,\partial_i([s^2]c)[s^3]\\
 &=\bigl[\partial_i(c-c_0)\cdot j_{\mathcal E_s}(\alpha)\bigr]_{s^3}\\
 &\overset{\ref{eq:boundary-extension-three-constraint-wdvv}}{=}
 \bigl[(\mathcal E_sc)\cdot j_{\Delta_i}(\alpha)\bigr]_{s^3}\\
 &=2[s^2]c\cdot\tfrac12(\partial_iD)[s^3]\\
 &\overset{\ref{thm:enhanced-superpotential-axioms}\textup{\ref{item:enhanced-s-equation}}}{=}
 -\tfrac12\bigl([s^3/3!]\overline\Omega\bigr)\partial_iD[s^3].
\end{align*}
By \cref{lm:s-degree-one-detector}, the map $q\mapsto q[s^3]$ is
injective on $R_W^0$, so this proves the lemma.
\end{proof}

Let
\[
 Q'_W:=Q_W[\xi]/\bigl(\xi^2-(g^{LL}\Phi_{ppp})^2-4D\bigr),
 \qquad |\xi|=n.
\]
Define $Q_W$-linear maps
\[
 \mathcal L_\pm\colon U\otimes Q_W\longrightarrow Q'_W
\]
by
\[
 \mathcal L_\pm(\alpha+fp)
 =\mathcal L(\alpha)+\frac f2
 \bigl(g^{LL}\Phi_{ppp}\pm\xi\bigr),
 \qquad \alpha\in\mathcal A,\quad f\in Q_W,
\]
where $\xi$ also denotes its class in $Q'_W$.

\begin{prop}
\label[prop]{prop:boundary-character-extensions}
The maps $\mathcal L_\pm$ are unital degree-zero characters of
$(U\otimes Q_W,\star)$ extending $\mathcal L$.
\end{prop}
\begin{proof}
Let $B$ be a graded commutative $Q_W$-algebra.  Since
$U\otimes Q_W=\mathcal A\oplus Q_Wp$, a degree-zero $Q_W$-linear
extension of $\mathcal L$ to $U\otimes Q_W$ is uniquely determined
by its value $z\in B^n$ on $p$.
For $\alpha,\beta\in\mathcal A$ and $f,g\in Q_W$,
\cref{lm:boundary-primitive-products} gives
\[
 (\alpha+fp)\star(\beta+gp)
 =\alpha\star\beta+fg(p\star p)
  +\bigl(f\mathcal L(\beta)+g\mathcal L(\alpha)\bigr)p.
\]
Since $\mathcal L$ is a character, multiplicativity is equivalent
to
\[
 z^2-g^{LL}\Phi_{ppp}z-D=0.
\]
The extension preserves the unit because $\mathcal L(1_X)=1$.
Taking $B=Q'_W$, the elements
$z=\tfrac12(g^{LL}\Phi_{ppp}\pm\xi)$ satisfy this equation and
have degree $n$.  The corresponding extensions are
$\mathcal L_\pm$.
\end{proof}

\begin{rem}
\label{rem:boundary-three-constraint-proportionality}
If $D|_{t=0}$ is invertible after an extension of the Novikov
coefficients, then, after completing in the bulk variables,
\cref{lm:boundary-three-constraint-equation} gives
\[
 [s^3/3!]\overline\Omega
 =\frac{\left.([s^3/3!]\overline\Omega)\right|_{t=0}}
        {D|_{t=0}}\,D.
\]
Indeed, $D$ is then a unit, and the quotient
$([s^3/3!]\overline\Omega)/D$ is independent of the bulk variables.
Thus the proportionality factor is determined by the zero-bulk
values.  For the even quadric, we will prove
$\Phi_{ppp}=0$ and $D=2[s^3/3!]\overline\Omega$ in
\cref{subsec:even-quadric-boundary-character}.
\end{rem}

\subsection{The boundary character on the Euler element}
\label{subsec:euler-spectral}

We continue with the point-like pair \((\gamma_W,b)\), its Maurer--Cartan
coefficient \(c\), and the map \(\mathcal L\) defined in the preceding
subsection.  We compute its value on the Euler element of \(\mathcal A\).
We write
\[
    R_{W'}^0:=\kappa^*R_W^0
    :=\operatorname{Im}\!\left(
      \kappa^*\colon R_W^0\longrightarrow R_W^0
    \right).
\]

\subsubsection{The \texorpdfstring{\(s\)}{s}-degree-zero Maurer--Cartan
coefficient}

Write \(c_0=c|_{s=0}\).  Since evaluation at \(s=0\) preserves the
degree of \(c\), we have \(|c_0|=2\).

The divisor equation descends to \(c_0\) as follows.

\begin{lm}
\label{lm:mc-divisor-at-boundary-zero}
For every \(A\in W^2\),
\[
    (\partial_A-E_A)c_0=0.
\]
\end{lm}
\begin{proof}
\cref{thm:enhanced-superpotential-axioms}
\textup{\ref{item:enhanced-s-equation}} gives
\(\mathcal E_s\overline\Omega=[-cs]\).  So,
\[
    [\overline\Omega]_{s^1}
    =[\mathcal E_s\overline\Omega]_{s^1}
    =-c_0[s].
\]
\cref{thm:enhanced-superpotential-axioms}
\textup{\ref{item:enhanced-divisor}} gives
\[
    (\partial_A-E_A)\overline\Omega
    =P_Q(A\smile\Gamma_W).
\]
The right-hand side has \(s\)-degree zero.  Since \(\partial_A\) and
\(E_A\) preserve \(s\)-degree, taking the \(s\)-degree-one component gives
\[
    -(\partial_A-E_A)c_0[s]=0.
\]
Multiplication by \([s]\) is injective on \(R_W^0\) by
\cref{lm:s-degree-one-detector}, which proves the claim.
\end{proof}

\subsubsection{The Euler element and its relative lift}

We assume that \(c_1(TX)\) has the following relative lift.

\begin{assum}[Relative lift of \(c_1(TX)\)]
\label{assum:monotone-divisor-direction}
There is a class \(A_\omega\in(W')^2\) with a representative in
\(\mathcal D_X^{\mathrm{adm}}\) such that
\begin{equation}\label{eq:monotone-divisor-direction}
    \rho(A_\omega)=c_1(TX),
    \qquad
    \int_\beta A_\omega=\frac{\mu_L(\beta)}2
    \quad(\beta\in\Pi).
\end{equation}
\end{assum}

\begin{rem}
In the de~Rham setting, the relative Maslov class
\(\mu_L/2\in H^2(X,L;\R)\) maps to \(c_1(TX)\) and annihilates
\(S_L\).  When \(0\ne c_1(TX)\in U\), choose its divisor
representative from this relative class and choose
\(\mathcal D_X^{\mathrm{adm}}\) compatibly.  Since \(n>1\), we have
\((W')^2=W^2\), so this gives
\cref{assum:monotone-divisor-direction}.  This is a choice of the bulk
representative; it requires no proportionality between area and Maslov
index on all relative homology classes.

A sufficient condition for using the particular representative
\(\lambda^{-1}\omega\) is
\[
    \omega(\beta)=\frac{\lambda}{2}\mu_L(\beta)
    \quad(\beta\in\Pi),\qquad\lambda>0.
\]
In this case \(A_\omega=[\lambda^{-1}\omega]\), with admissibility
following from \(S_L\subset\ker\omega\).
\end{rem}

The Euler element \(E\) of \eqref{eq:euler-element} corresponds to the
Euler vector field of big quantum cohomology \cite{dubrovin1996geometry}
under the identification \(\partial_{t_m}\leftrightarrow\Delta_m\),
after setting the closed bulk coordinates outside \(U_L\) to zero.
We lift it to \(W'\otimes Q_W\) using the isomorphism
\(\rho:W'\to U_L\).

\begin{dfn}[The relative Euler lift]
\label{dfn:liftable-euler-element}
Under \cref{assum:monotone-divisor-direction}, the Euler element
\(E\in\mathcal A^2\) of \eqref{eq:euler-element} has a relative lift and an
associated derivation
\begin{align}
    \widetilde E
    &:=A_\omega+
    \sum_{m\in I_L}
    \left(1-\frac{|\Delta_m|}{2}\right)t_m\Gamma_m,
    \notag\\
    \partial_E:=\partial_{\widetilde E}
    &=\partial_{A_\omega}+
    \sum_{m\in I_L}
    \left(1-\frac{|\Delta_m|}{2}\right)
    t_m\partial_{\Gamma_m}.
    \label{eq:euler-direction}
\end{align}
\end{dfn}

By construction, \(\rho(\widetilde E)=E\).  Moreover,
\(|t_m|=2-|\Delta_m|\), and hence
\(1-|\Delta_m|/2=|t_m|/2\) is half the degree of \(t_m\).
These weights give the homogeneity relation needed to evaluate the
character on \(E\).

\begin{lm}
\label{lm:euler-superpotential-homogeneity}
Suppose \cref{assum:monotone-divisor-direction} holds.  Then
\[
    \kappa^*(\partial_Ec_0)=\kappa^*c_0.
\]
\end{lm}
\begin{proof}
Write
\[
    \kappa^*c_0=\sum_{\beta,a}c_{\beta,a}T^\beta t^a,
    \qquad
    t^a=\prod_{m\in I_L}t_m^{a_m},
    \qquad c_{\beta,a}\in K_L.
\]
Since \(K_L\) is concentrated in degree zero and \(|\kappa^*c_0|=2\),
\[
    c_{\beta,a}\ne0
    \quad\Longrightarrow\quad
    \mu_L(\beta)+\sum_{m\in I_L}a_m|t_m|=2.
\]
Recall that \(\kappa^*\) fixes the Novikov variables and the coordinates
\(t_m\), \(m\in I_L\), and sets the complementary coordinates to zero.
Thus it commutes with \(E_{A_\omega}\) and with
\(\partial_{\Gamma_m}\) for \(m\in I_L\).
By \cref{lm:mc-divisor-at-boundary-zero} and
\eqref{eq:euler-direction},
\[
\begin{aligned}
    \kappa^*(\partial_Ec_0)
    &=\sum_{\beta,a}
    \left(
      \int_\beta A_\omega+
      \frac12\sum_{m\in I_L}a_m|t_m|
    \right)c_{\beta,a}T^\beta t^a\\
    &=\kappa^*c_0,
\end{aligned}
\]
where the last equality uses \eqref{eq:monotone-divisor-direction}.
\end{proof}

We can now identify the value of the boundary character on the Euler
element.

\begin{thm}[The boundary character on the Euler element]
\label{thm:euler-character}
Suppose \cref{assum:monotone-divisor-direction} holds and
\[
    \kappa^*\!\left(
      \left.(\partial_sc)\right|_{s=0}
    \right)=0.
\]
Then
\[
    \mathcal L(E)=-\kappa^*c_0.
\]
\end{thm}
\begin{proof}
The derivation \(\partial_E\) annihilates \(s\).  Therefore, by the
definition of \(\mathcal L\) and
\cref{lm:euler-superpotential-homogeneity},
\[
    \mathcal L(E)
    =-\kappa^*\!\left(
      \left.(\partial_Ec)\right|_{s=0}
    \right)
    =-\kappa^*(\partial_Ec_0)
    =-\kappa^*c_0.
\]
\end{proof}

Under the hypotheses of \cref{thm:euler-character},
\cref{thm:open-character} makes \(\mathcal L\) a character.  Hence
\[
    \mathcal L(\alpha\star E)
    =-\kappa^*c_0\cdot\mathcal L(\alpha)
    \qquad(\alpha\in\mathcal A).
\]
The content of \cref{thm:euler-character} is the identification of this
Euler eigenvalue with the \(s\)-degree-zero Maurer--Cartan coefficient.

\subsection{Floer cohomology implications}
Continue with the point-like bounding pair \((\gamma_W,b)\) and its
Maurer--Cartan coefficient \(c\).

\begin{lm}\label{lm:point-like-floer-nonvanishing}
Assume \(\left.(\partial_sc)\right|_{s=0}=0\).  Then the Floer cohomology
group
\[
HF\bigl(L,\q_1^{\gamma_W,b|_{s=0}};R_W^0\bigr)
\]
is nonzero.
\end{lm}
\begin{proof}
The structure axiom \ref{ax:qkl-structure-equation} with \(k=1\), \(l=0\),
and the unit axiom \ref{ax:qkl-unit} hold after deformation by
\cref{lm:bg-deformed-axioms}.  Since
\(\q_0^{\gamma_W,b|_{s=0}}=c|_{s=0}\cdot1_L\) with
\(c|_{s=0}\in\mathcal Z(R_W^0)\),
the two curvature terms cancel, so \(\q_1^{\gamma_W,b|_{s=0}}\) squares to zero.
Applying \(\partial_s\) to the Maurer--Cartan equation
\(\q_0^{\gamma_W,b}=c\cdot1_L\) and setting \(s=0\) gives
\[
 \pm\q_1^{\gamma_W,b|_{s=0}}
 \left(\left.(\partial_sb)\right|_{s=0}\right)
 =\left.(\partial_sc)\right|_{s=0}\cdot1_L=0.
\]
Thus \(\left.(\partial_sb)\right|_{s=0}\) is a degree \(n\) Floer cocycle.
The point-like condition
\(\langle1,b\rangle_L=-s\) and the definition of the extension of scalars of the pairing give
\[
 \left\langle1,\left.(\partial_sb)\right|_{s=0}\right\rangle_L=(-1)^n.
\]
By \cref{lm:bg-deformed-axioms}, the cyclic-symmetry axiom
\ref{ax:qkl-cyclic-symmetry} and the unit axiom \ref{ax:qkl-unit} also hold
for the \((\gamma_W,b|_{s=0})\)-deformed operations.  The unit axiom gives
\(\q_1^{\gamma_W,b|_{s=0}}(1_L)=0\); cyclic symmetry with \(k=1\), \(l=0\),
therefore gives
\[
 \left\langle1,\q_1^{\gamma_W,b|_{s=0}}(x)\right\rangle_L
 =\pm\left\langle\q_1^{\gamma_W,b|_{s=0}}(1_L),x\right\rangle_L=0
\]
for every homogeneous \(x\). Hence \(\left.(\partial_sb)\right|_{s=0}\)
cannot be exact, and its class proves the nonvanishing assertion.
\end{proof}
\begin{lm}\label{lm:point-like-floer-annihilation}
We have
\[
 \left.(\partial_sc)\right|_{s=0}\cdot HF\!\left(
 L,\q_1^{\gamma_W,b|_{s=0}};
 R_W^0
 \right)=0.
\]
\end{lm}
\begin{proof}
Applying \(\partial_s\) to the Maurer--Cartan equation of \cref{dfn:bounding-pair},
using the sign rule in \cref{rem:scalar-extension-shifts}, and setting
\(s=0\) gives
\[
 (-1)^{n-1}\q_1^{\gamma_W,b|_{s=0}}
 \left(\left.(\partial_sb)\right|_{s=0}\right)
 =\left.(\partial_sc)\right|_{s=0}\cdot1_L.
\]
Let \(x\) be a homogeneous Floer cocycle. Apply the structure axiom
\ref{ax:qkl-structure-equation}, in its deformed form from
\cref{lm:bg-deformed-axioms}, with \(k=2\), \(l=0\), and inputs
\(\left.(\partial_sb)\right|_{s=0}\) and \(x\). Substituting
\(\q_0^{\gamma_W,b|_{s=0}}=c|_{s=0}\cdot1_L\), we obtain
\[
\begin{aligned}
0={}&
\q_1^{\gamma_W,b|_{s=0}}
\left(\q_2^{\gamma_W,b|_{s=0}}
\left(\left.(\partial_sb)\right|_{s=0},x\right)\right)
\\
&+\q_2^{\gamma_W,b|_{s=0}}
\left(\q_1^{\gamma_W,b|_{s=0}}
\left(\left.(\partial_sb)\right|_{s=0}\right),x\right)
\\
&+(-1)^{n+1}\underbrace{\q_2^{\gamma_W,b|_{s=0}}
\left(\left.(\partial_sb)\right|_{s=0},
\q_1^{\gamma_W,b|_{s=0}}(x)\right)}_{=0}
\\
&+\underbrace{\q_3^{\gamma_W,b|_{s=0}}
\left(c|_{s=0}\cdot1_L,
\left.(\partial_sb)\right|_{s=0},x\right)}_{=0}
\\
&+(-1)^{n+1}\underbrace{\q_3^{\gamma_W,b|_{s=0}}
\left(\left.(\partial_sb)\right|_{s=0},
c|_{s=0}\cdot1_L,x\right)}_{=0}
\\
&+(-1)^{n+|x|}\underbrace{\q_3^{\gamma_W,b|_{s=0}}
\left(\left.(\partial_sb)\right|_{s=0},
x,c|_{s=0}\cdot1_L\right)}_{=0}.
\end{aligned}
\]
The third term is zero because \(\q_1^{\gamma_W,b|_{s=0}}(x)=0\).
The last three terms are zero by the unit axiom \ref{ax:qkl-unit},
after scalar extension and deformation as in
\cref{dfn:scalar-extension-structure-maps,lm:bg-deformed-axioms}.
Substituting the identity obtained by applying \(\partial_s\) to the
Maurer--Cartan equation into the two remaining terms, and then using
\(\q_2^{\gamma_W,b|_{s=0}}(1_L,x)=x\) from the same unit axiom, gives
\[
 \begin{gathered}
 \q_1^{\gamma_W,b|_{s=0}}
 \left(\q_2^{\gamma_W,b|_{s=0}}
 \left(\left.(\partial_sb)\right|_{s=0},x\right)\right)
 \\=(-1)^n\q_2^{\gamma_W,b|_{s=0}}
 \left(\left.(\partial_sc)\right|_{s=0}\cdot1_L,x\right)
 =(-1)^n\left.(\partial_sc)\right|_{s=0}x.
 \end{gathered}
\]
Thus \(\left.(\partial_sc)\right|_{s=0}x\) is exact for every Floer
cocycle \(x\), proving the assertion over \(R_W^0\).
\end{proof}

\begin{rem}
\label{rem:point-like-obstruction-orders}
In even dimension, the existence of a point-like bounding pair is a
strictly stronger condition than unobstructedness in the ordinary sense.  Indeed,
\(\left.(\partial_sc)\right|_{s=0}=0\) automatically when \(n\) is even, so
by \cref{lm:point-like-floer-nonvanishing} the existence of a point-like
bounding pair forces
\[
 HF\bigl(L,\q_1^{\gamma_W,b|_{s=0}};R_W^0\bigr)\ne0 .
\]
Consequently, an even-dimensional monotone Lagrangian with vanishing self-Floer
cohomology admits no point-like bounding pair.
\end{rem}

\section{Open Gromov--Witten invariants of the even-dimensional quadric}
\label{sec:even-quadric-background}

We use the even quadric pair \((X,L)\), with \(n>2\) even, and
the orientation, relative spin structure, and cohomological notation of
\cref{subsec:intro-even-quadric}.
We apply the flatness equation of
\cref{thm:structural-open-wdvv,thm:intro-open-wdvv} to this pair.
The resulting recursions determine the enhanced open invariants from
\(\ogw_{1,3}\), computed in \cref{lm:even-quadric-seed}.
The case \(n=2\) is not covered.

\subsection{Cohomology and bounding pairs}

We use the scalar de~Rham model of \cref{subsec:de-rham-model}, under the
hypotheses of \cref{prop:de-rham-structure-maps}.
All numbered references to \cite{hugtenburg2024examples} below refer to
arXiv:2402.03209v1.
The classical ring \eqref{prop:even-quadric-topology} and relative basis
recalled in \cref{subsec:intro-even-quadric} follow from
\cite[Propositions~5.8--5.10 and Section~7.1]{hugtenburg2024examples}.
Take \(S_L=0\).  By \cite[Lemma~5.12]{hugtenburg2024examples},
the identification of absolute and relative degrees by the line class
makes \(\varpi=\Id\).  By
\cite[Remark~5.13 and Lemma~5.14]{hugtenburg2024examples},
\begin{equation}\label{lm:even-quadric-classes}
 \int_d\tilde h=d,\qquad c_1(TX)(d)=nd,\qquad
 \mu_L(d)=2nd,\qquad w_{\mathfrak{s}}=0.
\end{equation}
Put \(\Delta_i=h^i\) and \(\Delta_L=p\).
The ring \eqref{prop:even-quadric-topology} gives the metric and its inverse:
\begin{equation}\label{eq:even-quadric-metric}
 g_{ij}=2\delta_{i,n-j},\qquad g_{iL}=g^{iL}=0,\qquad
 g^{i,n-i}=\tfrac12,\qquad g^{LL}=\tfrac{(-1)^{n/2}}{2}.
\end{equation}

The relative complex \(\widehat A^*(X,L)\) of
\cite[Section~2.2]{hugtenburg2024examples} is precisely
\(\ker\mathfrak i_\R\).  The cone construction of
\eqref{eq:intro-cone} gives
\[
 \Cone{\mathfrak i_\R}=A^*(X)\oplus\R[-n-1],\qquad
 d_{\Cone}(\eta,\xi)=(d\eta,\mathfrak i_\R(\eta)).
\]
Since \(\mathfrak i_\R\) is surjective on cochains, the inclusion
\(\eta\mapsto(\eta,0)\) induces an isomorphism
\[
 H^*(\ker\mathfrak i_\R)\xrightarrow{\ \sim\ }
 H^*(\Cone{\mathfrak i_\R}).
\]
Thus the cited relative basis supplies the bulk classes used here;
\(\Gamma_0=1_X\), and \(\Gamma_1\) is represented in
\(\mathcal D_X^{\mathrm{adm}}\) by \(a_\omega^{-1}\omega\), where
\(a_\omega=\omega(1)>0\) is the area of the line class.

\begin{notn}[Coefficients]\label{notn:even-quadric-coefficients}
Set \(q=T^1\), so that \(\Lambda=\R[[q]]\) with \(|q|=2n\), and put
\[
 Q_U:=\Lambda[[t_0,\ldots,t_n,t_L]],\qquad
 Q_W:=\Lambda[[t_0,\ldots,t_n]],\qquad
 R_W:=Q_W\langle s\rangle,
\]
where \(|t_i|=2-2i\), \(|t_L|=2-n\), \(|s|=1-n\), and \(R_W\) is
completed as in \eqref{eq:restricted-s-completion}.
These are the rings of \cref{subsubsec:intro-coefficients} for
\(U=H^*(X;\R)\) and \(W=\operatorname{span}\{\Gamma_0,\ldots,\Gamma_n\}\),
with \(t_i\) dual to \(h^i\) and to \(\Gamma_i\), and \(t_L\) dual to \(p\).
Note that the unit coordinate is \(t_0\) here, rather than the \(t_1\) of
\cref{subsec:enhanced-axioms}.

Represent \(\Gamma_i\) and \(h^i=\rho(\Gamma_i)\) by the same closed form,
and set
\[
 \gamma_U=\sum_{i=0}^n h^it_i+p\,t_L,\qquad
 \gamma_W=\Gamma_W=\sum_{i=0}^n\Gamma_it_i.
\]
Then \(\rho^*\gamma_U=\gamma_W\) as cochains, as required in
\eqref{eq:wdvv-compatible-W-representatives}.
\end{notn}

\begin{prop}[Bounding pair and \(U_L\)]
\label[prop]{prop:even-quadric-standing-hypotheses}
There is a unique gauge class of point-like unit bounding pairs for the
bulk class \(\Gamma_W\) over the above coefficient system.
It has a representative with \(\gamma_W=\Gamma_W\) and
\(\langle1,b\rangle_L=-s\).  The identifications of
\(U_L,W',P_\R,\kappa^*\) in \cref{subsec:intro-even-quadric} hold,
\(\Coker\mathfrak i_Q=0\), and \(\mathcal A\) is a unital subalgebra
of \((U\otimes Q_W,\star)\).
\end{prop}
\begin{proof}
We apply \cite[Proposition~6.1\textup{(1)}]
{kosloff2025pointlikenoncommutativef} with \(a=s\).
The cohomological hypothesis holds since \(L\simeq S^n\).
Thus there exists a bounding pair with \(\gamma_W=\Gamma_W\) and
\(\int_Lb=s\).  By \cite[Theorem~2 and equation~(6)]
{kosloff2025pointlikenoncommutativef}.
By the scalar-extension pairing convention of
\cref{dfn:scalar-extension-interval-pairing}, for
\(\eta\in A^n(L)\) with \(\int_L\eta=1\),
\[
 \langle\eta\otimes s,1\rangle_L=s,\qquad
 \langle1,\eta\otimes s\rangle_L=-s.
\]
Hence the constructed pair satisfies \(\langle1,b\rangle_L=-s\).

By \eqref{prop:even-quadric-topology}, 
\(\mathfrak i_\R(p)=2(-1)^{n/2}\) and \(\mathfrak i_\R(h^i)=0\).
Consequently \(\mathfrak i_\R\) is onto on cohomology, so
\(\Coker\mathfrak i_\R=\Coker\mathfrak i_Q=0\).
Hence \(P_\R=0\).  Since \(U=H^*(X;\R)\), the definitions above give
\(U_L=\Ker\mathfrak i_\R\), \(W'=W=\rho^{-1}(U)\), and
\(\kappa^*=\Id\).  Closed WDVV
\cite[Theorem~11]{hugtenburg2024examples} and the metric
\eqref{eq:even-quadric-metric} give
\cref{assum:closed-frobenius-sector}.  The bounding pair above satisfies
\cref{assum:universal-bounding-pair}.  By
\cref{cor:liftable-quantum-subalgebra}, \(\mathcal A\) is closed under
\(\star\).
\end{proof}

\subsection{Coefficients and initial values}

For an ordered tuple \(I=(i_1,\ldots,i_l)\), write
\(\Gamma_I=(\Gamma_{i_1},\ldots,\Gamma_{i_l})\), and similarly for
\(\Delta_I\); put \(\partial_I=\partial_{i_l}\cdots\partial_{i_1}\),
with \(\partial_\emptyset=\Id\).
By \cref{lm:s-degree-one-detector}, as completed \(Q_W\)-modules,
\begin{equation}\label{eq:even-quadric-cyclic-coefficients}
 \Coker\mathfrak i=HH_0(R_W)/Q_W
 =\widehat{\bigoplus}_{\substack{k\ge1\\k\text{ odd}}}Q_W[s^k].
\end{equation}
Thus the invariants vanish for positive even \(k\).  They also vanish
for \(k=0\), since \([L]\ne0\) forces the \(s^0\) component of
\(\Coker\mathfrak i\) to vanish.  By
\cref{lm:enhanced-superpotential-choice-from-base},
the image of \(\overline\Omega\) in \(HH_0(R_W)/Q_W\) is
\([\q_{-1,0}^{\gamma_W,b}]\).  Thus its coefficients with \(k>0\)
are the coefficients \(\ogw\) of that scalar potential, with the
normalization of \eqref{eq:intro-open-invariants}.

The closed degree axiom
\cite[Theorem~10\textup{(a)}, equation~(8)]{hugtenburg2024examples} and
\cref{prop:enhanced-ogw-axioms} give
\begin{equation}\label{lm:even-quadric-degree-constraints}
\begin{aligned}
 \GW_d(\Delta_I)\ne0\quad&\Longrightarrow\quad
 2n(d+1)-6+2l=\sum_{a=1}^l|\Delta_{i_a}|,\\
 \ogw_{d,k}(\Gamma_{i_1},\ldots,\Gamma_{i_l})\ne0
 \quad&\Longrightarrow\quad
 n-3+2nd+k+2l=nk+2\sum_{a=1}^l i_a.
\end{aligned}
\end{equation}
We use the zero, unit and divisor axioms of
\cref{prop:enhanced-ogw-axioms} in the recursions below.

\begin{prop}[Closed initial values]
\label[prop]{prop:even-quadric-closed-initial-values}
The nonzero classical three-point values are
\[
 \GW_0(h^i,h^j,h^a)=2\quad(i+j+a=n),\qquad
 \GW_0(p,p,1)=2(-1)^{n/2},
\]
up to permutation.  In degree one,
\[
 \GW_1(h^{n-1},h^n)=4,\qquad
 \GW_1(h^i,h^j,h^a)=
 \begin{cases}
 0,&(i,j,a)=(0,n,n),\\
 4,&i>0,\ a=n,\\
 8,&i>0,\ a<n,
 \end{cases}
\]
where \(0\le i\le j\le a\le n\) and \(i+j+a=2n\).
These are \cite[Proposition~5.8, Corollary~5.16, and Lemma~5.17]
{hugtenburg2024examples}.
\end{prop}

\subsection{Recursions}

To express the terms involving \(c\) in WDVV as open invariants, use the
\(s\)-equation of \cref{thm:enhanced-superpotential-axioms}
\textup{\ref{item:enhanced-s-equation}},
\([-cs]=\mathcal E_s\overline\Omega\).
All coefficients in \(Q_W\) are even, and \(|c|=2\) by
\cref{dfn:bounding-pair,dfn:structure-maps}, so \(c\) contains only even
powers of \(s\).  Since \(\mathcal E_s\) multiplies the coefficient of
\([s^k]\) by \(k\), by equation~\eqref{eq:s-euler-counts} of
\cref{lm:s-degree-one-detector}, comparison of
coefficients gives, for every ordered tuple \(I\), including the empty tuple,
\begin{equation}\label{eq:even-quadric-central-coefficient}
 \left[T^d\frac{s^{k-1}}{(k-1)!}\right]
 \left.\partial_I c\right|_{t=0}
 =-\ogw_{d,k}(\Gamma_I),
 \qquad d\ge0,\quad k\ge1\text{ odd}.
\end{equation}
The coefficient on the left is taken in \(R_W\); products involving
\(c\) use the central action on \(\Coker\mathfrak i\) recalled in
\cref{subsubsec:intro-coefficients}.

\begin{thm}[Open WDVV relations]
\label{thm:even-quadric-open-wdvv-relations}
For an ordered tuple \(J=(j_1,\ldots,j_q)\) and
\(A\subseteq[q]:=\{1,\ldots,q\}\), write
\[
    \Gamma_{J_A}:=(\Gamma_{j_a})_{a\in A},
    \qquad
    \Delta_{J_A}:=(\Delta_{j_a})_{a\in A},
\]
where the entries retain their original order.  A sum over
\(A\sqcup B=[q]\) runs over all complementary ordered subtuples; empty
strings of insertions are omitted, and
\(\Gamma_J:=\Gamma_{J_{[q]}}\), \(\Delta_J:=\Delta_{J_{[q]}}\).
All curve degrees and boundary summation indices below are nonnegative.
Set \(\Gamma_a=\Delta_a=0\) for \(a\notin\{0,\ldots,n\}\), and
set \(\binom ab=0\) unless \(0\le b\le a\).  Finally, set
\[
    \ogw_{d,k}:=\ogw_{d,k}(\,).
\]
Then the enhanced non-commutative open WDVV equation
\cref{thm:intro-open-wdvv} gives the following three
coefficient identities.
\begin{enumerate}[label=\textup{(\arabic*)},ref=\textup{\arabic*}]
\item If \(d>0\), \(k\ge1\), and \(2\le j\le n\), then
\begin{equation}\label{eq:even-quadric-wdvv-single}
\begin{aligned}
\ogw_{d,k}(\Gamma_j)
={}&
-\frac12
\sum_{\substack{\widehat d+d_0=d\\ d_0<d}}
\ \sum_{a=0}^{n}
\GW_{\widehat d}
   (\Delta_{j-1},\Delta_1,\Delta_a)\,
\ogw_{d_0,k}(\Gamma_{n-a})
\\
&-
\sum_{\substack{d_1+d_2=d\\ k_1+k_2=k-1}}
\binom{k-1}{k_1}
\ogw_{d_1,k_1}(\Gamma_{j-1},\Gamma_1)\,
\ogw_{d_2,k_2+2}
\\
&+
\sum_{\substack{d_1+d_2=d\\ k_1+k_2=k-1}}
\binom{k-1}{k_1}
\ogw_{d_1,k_1+1}(\Gamma_{j-1})\,
\ogw_{d_2,k_2+1}(\Gamma_1).
\end{aligned}
\end{equation}
Thus an invariant with one interior insertion is expressed in terms of closed
invariants and open invariants of smaller Novikov degree.  The cases \(j=0\)
and \(j=1\) are determined by \cref{prop:enhanced-ogw-axioms}.

\item Let \(l\ge2\), \(d,k\ge0\), and, after applying the unit and divisor
axioms, assume
\[
    2\le j_1\le j_2\le\cdots\le j_l\le n.
\]
Put \(J=(j_3,\ldots,j_l)\).  Then
\begin{equation}\label{eq:even-quadric-wdvv-multiple}
\begin{aligned}
&\ogw_{d,k}
  (\Gamma_{j_1},\Gamma_{j_2},\Gamma_J)
\\
={}&
\ogw_{d,k}
  (\Gamma_{j_2+1},\Gamma_{j_1-1},\Gamma_J)
\\
&+
\frac12
\sum_{\substack{\widehat d+d_0=d\\ d_0<d}}
\ \sum_{A\sqcup B=[l-2]}
\ \sum_{a=0}^{n}
\Bigl[
  \GW_{\widehat d}
    (\Delta_1,\Delta_{j_2},\Delta_{J_A},\Delta_a)\,
  \ogw_{d_0,k}
    (\Gamma_{n-a},\Gamma_{j_1-1},\Gamma_{J_B})
\\[-2pt]
&\hspace{12em}
 -
  \GW_{\widehat d}
    (\Delta_1,\Delta_{j_1-1},\Delta_{J_A},\Delta_a)\,
  \ogw_{d_0,k}
    (\Gamma_{n-a},\Gamma_{j_2},\Gamma_{J_B})
\Bigr]
\\
&+
\sum_{\substack{d_1+d_2=d\\ k_1+k_2=k\\
                 A\sqcup B=[l-2]}}
\binom{k}{k_1}
\Bigl[
  \ogw_{d_1,k_1}
    (\Gamma_1,\Gamma_{j_2},\Gamma_{J_A})\,
  \ogw_{d_2,k_2+1}
    (\Gamma_{j_1-1},\Gamma_{J_B})
\\[-2pt]
&\hspace{12em}
 -
  \ogw_{d_1,k_1}
    (\Gamma_1,\Gamma_{j_1-1},\Gamma_{J_A})\,
  \ogw_{d_2,k_2+1}
    (\Gamma_{j_2},\Gamma_{J_B})
\Bigr].
\end{aligned}
\end{equation}
The first term on the right is the only term of the same Novikov degree
that is not a product of lower terms; it replaces \(j_1\) by \(j_1-1\).

\item If \(d>0\) and \(k\ge1\), then
\begin{equation}\label{eq:even-quadric-wdvv-no-interior}
\begin{aligned}
0={}&
\frac12
\sum_{\widehat d+d_0=d+1}
\ \sum_{a=0}^{n}
\GW_{\widehat d}
  (\Delta_{n/2},\Delta_{n/2+1},\Delta_a)\,
\ogw_{d_0,k}(\Gamma_{n-a})
\\
&+
\sum_{\substack{d_1+d_2=d+1\\ k_1+k_2=k-1}}
\binom{k-1}{k_1}
\ogw_{d_1,k_1+2}\,
\ogw_{d_2,k_2}(\Gamma_{n/2},\Gamma_{n/2+1})
\\
&-
\sum_{\substack{d_1+d_2=d+1\\ k_1+k_2=k-1}}
\binom{k-1}{k_1}
\ogw_{d_1,k_1+1}(\Gamma_{n/2})\,
\ogw_{d_2,k_2+1}(\Gamma_{n/2+1}).
\end{aligned}
\end{equation}
The terms containing the unknown invariant with no interior insertions are
\[
    \left(
      4d+(k-1)\ogw_{1,1}(\Gamma_{n/2},\Gamma_{n/2+1})
    \right)\ogw_{d,k}
    =
    2(2d-k+1)\ogw_{d,k},
\]
where
\(\ogw_{1,1}(\Gamma_{n/2},\Gamma_{n/2+1})=-2\) by
\cref{thm:even-quadric-basic-values}.  All other terms in
\eqref{eq:even-quadric-wdvv-no-interior} have smaller open Novikov degree or
contain an interior insertion of degree greater than \(2\).  By the degree
constraint \eqref{lm:even-quadric-degree-constraints},
\(2(2d-k+1)\) vanishes only for \((d,k)=(1,3)\), whose value is
computed in \cref{lm:even-quadric-seed}.
\end{enumerate}
\end{thm}

We prove \cref{thm:even-quadric-open-wdvv-relations} in four lemmas,
corresponding to \cite[Lemmas~6.4--6.7]{hugtenburg2024examples}.
We first record the coefficient calculation used below,
following \cite[Theorem~1 and Section~6]{Relativequantum2025}.
The derivations \(\partial_{t_i}\) and \(\mathcal E_s\) are even, so the sign in
\cref{thm:structural-open-wdvv} is positive.
In \eqref{eq:intro-open-wdvv}, one index of each \(g^{lm}\) lies in
\(I_L\), and \(g^{iL}=0\) for \(i\in I_L\) by
\eqref{eq:even-quadric-metric}.  Hence the terms with index \(L\)
vanish, and the closed invariants in
\cref{thm:even-quadric-open-wdvv-relations} have only the classes
\(h^i\) as insertions.

For the boundary factors, \(s^e\cdot[s^f]=[s^{e+f}]\) for even \(e\), so
\[
 [s^k/k!]\left(\frac{s^e}{e!}\frac{[s^f]}{f!}\right)
 =\binom{k}{e}\qquad(e+f=k,\ e\text{ even},\ f\text{ odd}).
\]
After substituting for \(c\) by
\eqref{eq:even-quadric-central-coefficient}, the factors from
\(-(\mathcal E_sc)\partial_i\partial_j\overline\Omega\)
and \(-\mathcal E_s\partial_i\overline\Omega\cdot\partial_jc\),
respectively, and dividing by \(k\), are the following, where
\(a,b\) are the odd numbers of boundary insertions in the two invariants:
\[
 \frac1k\frac{k!(b-1)}{a!(b-1)!}
 =\binom{k-1}{a}\quad(a+b-1=k),\qquad
 \frac1k\frac{k!a}{a!(b-1)!}
 =\binom{k-1}{a-1}.
\]
The first expression is zero for \(b=1\).
These yield the two \(\binom{k-1}{k_1}\) sums in
\eqref{eq:even-quadric-wdvv-single} and
\eqref{eq:even-quadric-wdvv-no-interior}.  For the bulk derivations,
the same coefficient calculation gives \(\binom{k}{k_1}\), as in
\eqref{eq:even-quadric-wdvv-multiple}.

\begin{lm}
\label[lm]{lm:even-quadric-single-insertion-recursion}
For \(d>0\), \(k\ge1\), and \(2\le j\le n\), relation
\eqref{eq:even-quadric-wdvv-single} holds.  Every nonzero summand on
its right involves only open invariants of Novikov degree less than \(d\).
\end{lm}
\begin{proof}
Use the proof of \cite[Lemma~6.4]{hugtenburg2024examples}, replacing
its WDVV equation~(22) by \cref{thm:intro-open-wdvv} with
\[
 \partial_u=\mathcal E_s,\qquad v=\Gamma_1,\qquad w=\Gamma_{j-1}.
\]
Set \(t=0\), extract \([T^d s^k/k!]\), and divide by \(k\).
The coefficient calculation above gives
\eqref{eq:even-quadric-wdvv-single}.  The desired invariant
comes from the classical contraction
\[
 \tfrac12\GW_0(\Delta_{n-j},\Delta_1,\Delta_{j-1})=1.
\]
The degree-zero vanishing argument following equations~(44)--(46)
in the cited proof applies verbatim: the open factors have no
insertion, two insertions, or a single nonunit insertion.
Our zero axiom in \cref{prop:enhanced-ogw-axioms} gives the same
vanishings, so all remaining open factors have degree less than \(d\).
\end{proof}

\begin{lm}
\label[lm]{lm:even-quadric-multiple-insertion-recursion}
For \(d,k\ge0\), \(l\ge2\), and
\(2\le j_1\le\cdots\le j_l\le n\), relation
\eqref{eq:even-quadric-wdvv-multiple} holds.  Apart from its first term,
every nonzero summand on the right involves only open invariants of
Novikov degree less than \(d\).  The first term replaces \(j_1\) by
\(j_1-1\), or vanishes if \(j_2=n\).
\end{lm}
\begin{proof}
Use the proof of \cite[Lemma~6.5]{hugtenburg2024examples}, replacing
its WDVV equation~(21) by \cref{thm:intro-open-wdvv} with
\[
 u=\Gamma_{j_2},\qquad v=\Gamma_{j_1-1},\qquad w=\Gamma_1.
\]
Apply \(\partial_J\), set \(t=0\), and extract \([T^d s^k/k!]\).
The coefficient calculation above gives
\eqref{eq:even-quadric-wdvv-multiple}.
The separation of the two closed terms and the degree-zero
vanishing argument following equations~(47)--(52) in the cited
proof apply verbatim.  Indeed, \cref{ax:qempty-energy} leaves only
\(A=\emptyset\) in closed degree zero, and
\cref{prop:enhanced-ogw-axioms} kills every degree-zero open factor
in the product sums.  This gives the asserted bounds and leaves
the displayed replacement of \(j_1\) by \(j_1-1\).
\end{proof}

\begin{lm}
\label[lm]{thm:even-quadric-basic-values}
\label[lm]{lm:even-quadric-ogw-11}
\label[lm]{lm:even-quadric-middle-pair-value}
In the conventions of this paper,
\[
 \ogw_{1,1}(\Gamma_n)=-2,\qquad
 \ogw_{1,1}(\Gamma_{n/2},\Gamma_{n/2+1})=-2.
\]
\end{lm}
\begin{proof}
Follow the proof of \cite[Lemma~5.18]{hugtenburg2024examples},
replacing its WDVV equation~(22) by
\cref{thm:intro-open-wdvv} with
\(\partial_u=\mathcal E_s\), \(v=\Gamma_1\), and \(w=\Gamma_{n-1}\).
Set \(t=0\) and take the coefficient of \(Ts\).  Equivalently, specialize
\eqref{eq:even-quadric-wdvv-single} to \(d=k=1\), \(j=n\).
Every product of two open invariants has a degree-zero factor with
no interior insertion, two interior insertions, or a single nonunit
insertion.  These factors vanish by the zero axiom.
The closed degree constraint leaves just two terms:
\[
\begin{aligned}
0={}&\frac12\GW_0(\Delta_1,\Delta_{n-1},\Delta_0)\,
                 \ogw_{1,1}(\Gamma_n)
\\
 &+\frac12\GW_1(\Delta_1,\Delta_{n-1},\Delta_n)\,
                 \ogw_{0,1}(\Gamma_0).
\end{aligned}
\]
By \cref{prop:even-quadric-closed-initial-values}, the two closed
invariants are \(2\) and \(4\), respectively.  Our zero axiom gives
\(\ogw_{0,1}(\Gamma_0)=1\).  Hence
\[
 0=\frac12\cdot2\,\ogw_{1,1}(\Gamma_n)
   +\frac12\cdot4\cdot1
  =\ogw_{1,1}(\Gamma_n)+2,
\]
which proves the first value.

For the second value, follow the insertion reduction of
\cite[Lemma~6.5]{hugtenburg2024examples}, used there to obtain
Lemma~6.6.  Here their WDVV equation~(21) is replaced by
\eqref{eq:even-quadric-wdvv-multiple}, proved above from
\cref{thm:intro-open-wdvv}.  Put
\[
 A_a=\ogw_{1,1}(\Gamma_a,\Gamma_{n+1-a}),\qquad 1\le a\le n/2.
\]
For \(2\le a\le n/2\), use \(d=k=1\), \(j_1=a\),
\(j_2=n+1-a\), and \(J=\emptyset\).  In our WDVV equation, take
\[
 u=\Gamma_{n+1-a},\qquad v=\Gamma_{a-1},\qquad w=\Gamma_1,
\]
and we extract \([Ts]\) at \(t=0\).
The degree-zero vanishing argument in the cited proof applies
verbatim: every remaining closed or open product has a degree-zero
open factor with two insertions or a single nonunit insertion.
Consequently the two classical terms give
\[
 \frac12\GW_0(\Delta_1,\Delta_{a-1},\Delta_{n-a})\,A_a
 =\frac12\GW_0(\Delta_1,\Delta_{n+1-a},\Delta_{a-2})\,A_{a-1}.
\]
Both closed coefficients are \(2\), so \(A_a=A_{a-1}\).
The divisor axiom and the first value give
\[
 \ogw_{1,1}(\Gamma_{n/2},\Gamma_{n/2+1})=A_{n/2}=\cdots=A_1
 =\ogw_{1,1}(\Gamma_n)=-2.\qedhere
\]
\end{proof}

\begin{rem}
\label{rem:even-quadric-basic-values-comparison}
\Cref{thm:even-quadric-basic-values} gives \(\ogw_{1,1}(\Gamma_n)=-2\),
whereas \cite[Lemma~5.18]{hugtenburg2024examples} gives
\(\ogw_{1,1}(h^n)=2\).  This sign difference already appears in the
zero axiom:
\cite[Theorem~12\textup{(c)}, equation~(19)]{hugtenburg2024examples}
uses \(\ogw_{0,1}(1)=-1\), whereas here
\(\ogw_{0,1}(\Gamma_0)=1\).  Thus the same two-term calculation ends
with \(\ogw_{1,1}(h^n)-2=0\) in their Lemma~5.18 and with
\(\ogw_{1,1}(\Gamma_n)+2=0\) here.  The graded sign responsible for
our normalization is explained in \cref{rem:odd-boundary-derivative-sign}.
The divisor axiom and insertion reduction carry this sign to the second
value.  The middle pair here is
\((\Gamma_{n/2},\Gamma_{n/2+1})\); in their odd-dimensional Lemma~6.6,
it is \((\Gamma_{(n+1)/2},\Gamma_{(n+1)/2})\).
\end{rem}

\begin{lm}
\label[lm]{lm:even-quadric-no-interior-recursion}
For \(d>0\) and \(k\ge1\), relation
\eqref{eq:even-quadric-wdvv-no-interior} holds.  Its terms containing
\(\ogw_{d,k}\) have total coefficient \(2(2d-k+1)\).
All other nonzero summands involve open invariants of smaller Novikov
degree or of degree \(d\) with an insertion \(\Gamma_j\), \(j\ge2\).
Subject to \(2nd=k(n-1)+3-n\), the coefficient vanishes only for
\((d,k)=(1,3)\).
\end{lm}
\begin{proof}
Follow the proof of \cite[Lemma~6.7]{hugtenburg2024examples},
replacing its WDVV equation~(22) by
\cref{thm:intro-open-wdvv} with
\[
 \partial_u=\mathcal E_s,\qquad v=\Gamma_{n/2+1},\qquad w=\Gamma_{n/2}.
\]
The two interior indices sum to \(n+1\).
Set \(t=0\), extract \([T^{d+1}s^k/k!]\), and divide by \(k\).
Using \eqref{eq:even-quadric-central-coefficient} for the product
signs, separate the summands as in equations~(53)--(57) of the
cited proof:
\begin{align}
0={}&
\frac12\GW_1(\Delta_{n/2},\Delta_{n/2+1},\Delta_{n-1})\,
\ogw_{d,k}(\Gamma_1)
\label{eq:even-quadric-no-interior-closed-leading}
\\
&+\frac12
\sum_{\substack{\widehat d+d_0=d+1\\d_0<d}}
\ \sum_{a=0}^{n}
\GW_{\widehat d}(\Delta_{n/2},\Delta_{n/2+1},\Delta_a)\,
\ogw_{d_0,k}(\Gamma_{n-a})
\label{eq:even-quadric-no-interior-closed-remainder}
\\
&+(k-1)\ogw_{d,k}\,
\ogw_{1,1}(\Gamma_{n/2},\Gamma_{n/2+1})
\label{eq:even-quadric-no-interior-open-leading}
\\
&+\sum_{\substack{d_1+d_2=d+1\\1\le d_1<d,\ 1\le d_2\le d\\
                  k_1+k_2=k-1}}
\binom{k-1}{k_1}
\ogw_{d_1,k_1+2}\,
\ogw_{d_2,k_2}(\Gamma_{n/2},\Gamma_{n/2+1})
\label{eq:even-quadric-no-interior-pair-remainder}
\\
&-\sum_{\substack{d_1+d_2=d+1\\1\le d_1,d_2\le d\\
                  k_1+k_2=k-1}}
\binom{k-1}{k_1}
\ogw_{d_1,k_1+1}(\Gamma_{n/2})\,
\ogw_{d_2,k_2+1}(\Gamma_{n/2+1}).
\notag
\end{align}
Here \(k_1,k_2\ge0\), and empty sums are zero.
The degree-zero vanishing argument following equations~(56)--(57)
in the cited proof applies verbatim: a degree-zero open factor has
no insertion, two insertions, or a single nonunit insertion.
Thus each open degree is at most \(d\).

For line~\eqref{eq:even-quadric-no-interior-closed-remainder}, the
closed degree constraint gives
\[
 \GW_{\widehat d}(\Delta_{n/2},\Delta_{n/2+1},\Delta_a)\ne0
 \quad\Longrightarrow\quad a=n\widehat d-1.
\]
Since \(0\le a\le n\), only \(\widehat d=1\), \(a=n-1\)
can contribute.  Hence that line vanishes, while
line~\eqref{eq:even-quadric-no-interior-closed-leading} is
\[
 \tfrac12\GW_1(\Delta_{n/2},\Delta_{n/2+1},\Delta_{n-1})\,
 \ogw_{d,k}(\Gamma_1)=4d\,\ogw_{d,k}.
\]
This uses \cref{prop:even-quadric-closed-initial-values} and the
divisor axiom.

In the first open product sum, \(d_1=d\) gives \(d_2=1\).
For \(\ogw_{1,k_2}(\Gamma_{n/2},\Gamma_{n/2+1})\), the degree axiom
becomes \((n-1)(k_2-1)=0\).  Hence \(k_2=1\) and \(k_1=k-2\).
The separated term in
line~\eqref{eq:even-quadric-no-interior-open-leading} is
\[
 \binom{k-1}{k-2}\ogw_{d,k}\,
 \ogw_{1,1}(\Gamma_{n/2},\Gamma_{n/2+1})
 =-2(k-1)\ogw_{d,k},
\]
by \cref{thm:even-quadric-basic-values}; it is absent when \(k=1\).
Combining the two separated terms gives
\[
 \bigl(4d-2(k-1)\bigr)\ogw_{d,k}
 =2(2d-k+1)\ogw_{d,k}.
\]
In line~\eqref{eq:even-quadric-no-interior-pair-remainder}, the
factor without interior insertions has degree \(d_1<d\).
All other factors of degree \(d\) in the remaining sums contain
\(\Gamma_{n/2}\) or \(\Gamma_{n/2+1}\), and \(n/2\ge2\).
Finally, as in the cited proof, the no-interior degree constraint
\[
 2nd=k(n-1)+3-n
 \quad\Longrightarrow\quad n(2d-k+1)=3-k
\]
shows that the coefficient vanishes precisely at \((d,k)=(1,3)\).
\end{proof}

\begin{proof}[Proof of \cref{thm:even-quadric-open-wdvv-relations}]
The first two assertions are
\cref{lm:even-quadric-single-insertion-recursion,lm:even-quadric-multiple-insertion-recursion}.
The third follows from \cref{lm:even-quadric-no-interior-recursion},
using the value computed in \cref{thm:even-quadric-basic-values}.
The unit and divisor cases follow from \cref{prop:enhanced-ogw-axioms}.
The exceptional invariant is computed in \cref{lm:even-quadric-seed} below.
\end{proof}

\begin{lm}
\label{lm:even-quadric-seed}
For the orientation, relative spin structure, and cyclic coefficient
normalization fixed above,
\[
 \ogw_{1,3}=2(-1)^{n/2+1}.
\]
\end{lm}
\begin{proof}
We follow \cite[Lemma~6.2 and Appendix~A]{hugtenburg2024examples}.
At zero bulk, grading and \cref{prop:even-quadric-standing-hypotheses}
give \(b|_{T^0}=\eta\otimes s\) with \(\int_L\eta=1\).
The coefficient of \(T^1s^j\) in \(b\) has form degree
\(1-2n+j(n-1)<0\) for \(j\le2\); since \(\q_{0,0}^0=0\),
degree-one corrections cannot contribute to \([T^1s^3]\overline\Omega\).
We extract the \(T^1s^3\)-coefficient using the \(s\)-degree projection
of \cref{lm:s-degree-decomposition}.  Since \(n\) is even, \(|s|\) is odd, and
\begin{align*}
 \ogw_{1,3}
 &\overset{\eqref{eq:intro-open-invariants}}{=}
 3!\,[T^1s^3]\,\overline\Omega\\
 &\overset{\ref{dfn:bg-deformation-qminusone}}{=}
 3!\,[T^1s^3]\sum_{j\ge0}\frac{1}{2j+1}
 \bigl\langle\q_{2j,0}\bigl((\eta s)^{\otimes2j}\bigr),\eta s\bigr\rangle_L\\
 &\overset{\textup{(E1)},\,\ref{dfn:scalar-extension-interval-pairing}}{=}
 3!\,[T^1s^3]\sum_{j\ge0}\frac{(-1)^{j+1}s^{2j+1}}{2j+1}
 \bigl\langle\q_{2j,0}(\eta^{\otimes2j}),\eta\bigr\rangle_L\\
 &=\frac{3!}{3}\bigl\langle\q_{2,0}^{1}(\eta,\eta),\eta\bigr\rangle_L\\
 &=2!\int_{\mathcal M_{3,0}(1)}e_0^*\eta\wedge e_1^*\eta\wedge e_2^*\eta.
\end{align*}
The last equality uses
\cite[Sections~2.2.1--2.2.2]{solomon2016differential} and the signed pairing
of \cref{dfn:pre-open-closed-pair}; here \(e_i\) are boundary evaluations.

Let \(\mathcal M^S_{3,0}(1)\) denote the moduli space with three labeled
boundary marked points, with no requirement that their labels respect
the cyclic order.  It is the disjoint union of \(2!\) copies of
\(\mathcal M_{3,0}(1)\), corresponding to the two cyclic orders.
For a permutation \(\sigma\) of \(\{0,1,2\}\) with \(\sigma(0)=0\), let
\(\varphi_\sigma:\mathcal M_{3,0}(1)\to\mathcal M^S_{3,0}(1)\)
be the embedding that relabels the boundary marked points by \(\sigma\).
Following the orientation convention of
\cite[Definition~3.2]{solomon2006intersection} for \(\dim L\) even,
the maps \(\varphi_\sigma\) preserve orientation.  Since \(|\eta|=n\) is even,
\[
 \ogw_{1,3}
 =\sum_{\sigma(0)=0}\int_{\mathcal M_{3,0}(1)}
 e_{\sigma(0)}^*\eta\wedge e_{\sigma(1)}^*\eta\wedge e_{\sigma(2)}^*\eta
 =\int_{\mathcal M^S_{3,0}(1)}e_0^*\eta\wedge e_1^*\eta\wedge e_2^*\eta.
\]
Since \(\int_L\eta=1\), this is the degree of the evaluation map
\(\mathcal M^S_{3,0}(1)\to L^3\).  Three distinct points of \(L\simeq S^n\)
lie on a unique real conic, whose two conjugate disks give one point
in each cyclic-order component of the evaluation fiber.
By \cite[Proposition~5.1]{solomon2006intersection}, conjugation has sign
\((-1)^\epsilon\), where
\[
 \epsilon=n(2n+1)+2n+3+(n+1)\equiv0\pmod2.
\]
Thus the two disks have the same sign.  The local framing calculation in
\cite[proof of Proposition~A.1]{hugtenburg2024examples} gives each disk sign
\((-1)^{1+n(n-1)/2}\).  Therefore
\[
 \ogw_{1,3}=2(-1)^{1+n(n-1)/2}=2(-1)^{n/2+1}.\qedhere
\]
\end{proof}

\begin{cor}[Determination of the open invariants]
\label{cor:even-quadric-open-reconstruction}
All enhanced invariants \(\ogw_{d,k}(\Gamma_I)\) are uniquely determined
by the closed invariants of \(X\), the axioms of
\cref{prop:enhanced-ogw-axioms}, the recursions of
\cref{thm:even-quadric-open-wdvv-relations}, and the single value
\(\ogw_{1,3}\).  By \cref{lm:even-quadric-seed},
\(\ogw_{1,3}=2(-1)^{n/2+1}\), so these data determine the numerical values of all the invariants.
\end{cor}
\begin{proof}
Induct on the Novikov degree \(d\); the degree-zero values are fixed
by the zero axiom.  At each positive degree use the following order:
first invariants with an insertion \(\Gamma_j\), \(j\ge2\), then
invariants without interior insertions, then invariants whose interior
insertions are all \(\Gamma_1\).
The closed invariants are already determined by
\cite[Theorem~8 and Lemma~7.5]{hugtenburg2024examples}.

For the first stage, the unit axiom gives zero for invariants with a unit
insertion, and the divisor axiom removes insertions of \(\Gamma_1\).
A single remaining insertion is determined by
\eqref{eq:even-quadric-wdvv-single}, as proved in
\cref{lm:even-quadric-single-insertion-recursion}.
For several insertions, sort their
indices and induct lexicographically on their number and their smallest
index.  By \cref{lm:even-quadric-multiple-insertion-recursion},
relation \eqref{eq:even-quadric-wdvv-multiple} lowers the smallest
index, while its other terms have strictly smaller Novikov degree.  When
that index reaches one, the divisor axiom reduces the number of
insertions.  This is the reduction of
\cite[Lemmas~6.4--6.5]{hugtenburg2024examples}.
The replacements of insertion indices and removal of divisors preserve
\[
 \sum_a(j_a-1)>0.
\]
Thus a nonzero term still has an insertion \(\Gamma_j\) with \(j\ge2\).
The reduction ends at a single \(\Gamma_j\), \(j\ge2\), or at an
out-of-range insertion, which is zero.  This stage never calls on the
still unknown degree-\(d\) invariant without interior insertions.

For the second stage, use
\eqref{eq:even-quadric-wdvv-no-interior}.  As verified in
\cref{lm:even-quadric-no-interior-recursion},
the coefficient of the unknown is \(2(2d-k+1)\), and every other term
has smaller Novikov degree or is an invariant of degree \(d\) with an
insertion \(\Gamma_j\), \(j\ge2\), already determined in the first stage.
By \eqref{lm:even-quadric-degree-constraints}, a nonzero invariant
\(\ogw_{d,k}\) satisfies
\[
 2nd=k(n-1)+3-n,\qquad n(2d-k+1)=3-k.
\]
Consequently the coefficient vanishes only for \((d,k)=(1,3)\).
To list the possible pairs explicitly, write \(k=3+2a\) with
\(a\in\Z\), since \(k\) is odd.  The degree condition becomes
\[
 n(d-1)=a(n-1).
\]
Since \(n\) and \(n-1\) are coprime and \(d>0\), there is an
integer \(m\ge0\) such that
\[
 d=1+m(n-1),\qquad a=mn,\qquad k=3+2mn.
\]
Substitution gives \(2(2d-k+1)=-4m\).
The case \(m=0\) is computed in \cref{lm:even-quadric-seed}; all other
cases are solved by the recursion.  Finally, the divisor axiom recovers
invariants whose interior insertions are all \(\Gamma_1\) from the
values without interior insertions.  This completes the induction.
\end{proof}

\begin{rem}
\label{rem:even-quadric-open-to-closed-reconstruction}
The closed invariants of positive degree required by
\cref{cor:even-quadric-open-reconstruction} can be replaced by the
additional open value \(\ogw_{1,1}(\Gamma_n)=-2\) of
\cref{thm:even-quadric-basic-values}, taken as open geometric input.
Indeed, setting \(d=k=1\) and \(j=n\) in
\eqref{eq:even-quadric-wdvv-single}, the zero axiom leaves only the
closed term paired with \(\ogw_{0,1}(\Gamma_0)=1\).  Hence the
divisor axiom gives
\[
 \GW_1(h^{n-1},h^n)
 =\GW_1(h^{n-1},h,h^n)
 =-2\ogw_{1,1}(\Gamma_n)=4.
\]
Together with the classical cohomology data, this value determines all
genus-zero closed invariants by closed WDVV and the closed axioms
\cite[Theorem~8 and Lemma~7.5]{hugtenburg2024examples}.
Then \cref{cor:even-quadric-open-reconstruction} determines all enhanced
open invariants from the additional initial value \(\ogw_{1,3}=2(-1)^{n/2+1}\).
Thus the two open values \(\ogw_{1,1}(\Gamma_n)\) and
\(\ogw_{1,3}\) suffice to determine all positive-degree open and
closed invariants, without independently supplied closed invariants of positive degree.
\end{rem}

\begin{rem}
There is also a folklore\footnote{This is a folklore result in a
community with one member, but one obtains it by asking an LLM whether results
similar to Sheridan's~\cite{sheridan2016fano} can be derived using the open
WDVV equations. I did not want to include LLM-generated results here.}
result that recovers the calsical result that there are $27$ lines on a cubic surface. The geometric input
consists of Woodward's computation of the number of Maslov-index-two disks
with boundary on a Lagrangian sphere~\cite{woodward2026holomorphic} and the
classical bijection between lines on the cubic surface and degree-two curves
through a general point on it~\cite{salmon1865geometry}.
\end{rem}

\subsection{The boundary character}
\label{subsec:even-quadric-boundary-character}

We apply \cref{subsubsec:boundary-character-extensions} to derive the
quadric characters of \cref{subsec:intro-quadric-characters}.
Observe that
\(\left.(\partial_sc)\right|_{s=0}=0\); together with
\(\kappa^*=\Id\), this verifies the hypotheses of
\cref{thm:open-character} for \(\mathcal L\).
We will extend it to
\(U\otimes Q_W=\mathcal A\oplus Q_Wp\) by computing its value on
\(p\star p\).

The class \(A_\omega=n\Gamma_1\) lies in
\(H^2_{\mathrm{adm}}(X,L;\R)\), lifts \(c_1(TX)=nh\),
and has periods \(nd=\mu_L(d)/2\) by \eqref{lm:even-quadric-classes}.  Hence
\cref{thm:euler-character} applies and gives
\[
 E=nh+\sum_{i=0}^n(1-i)t_i h^i\in\mathcal A,\qquad
 \mathcal L(E)=-c_0.
\]

At zero bulk, \eqref{eq:even-quadric-central-coefficient}
gives
\[
 \left.\mathcal L(h^i)\right|_{t=0}
 =\sum_{d\ge0}q^d\ogw_{d,1}(\Gamma_i).
\]
The degree constraint \eqref{lm:even-quadric-degree-constraints}
reduces to \(nd=i\).  The zero axiom and
\cref{thm:even-quadric-basic-values} therefore imply
\begin{equation}\label{eq:even-quadric-small-character}
 \left.\mathcal L(h^i)\right|_{t=0}
 =\begin{cases}
 1,&i=0,\\
 0,&1\le i<n,\\
 -2q,&i=n.
 \end{cases}
\end{equation}
For comparison, the small closed quantum ring is
\cite[Theorem~5 and Proposition~7.2]{hugtenburg2024examples}
\begin{equation}\label{prop:even-quadric-closed-quantum-ring}
 QH^*(X)=
 \frac{\R[[q]][x,y]}
 {(x^{n+1}-4qx,\ xy,\ y^2-(-1)^{n/2}(x^n-4q))},
 \qquad x=h,\quad y=p.
\end{equation}
In the ordinary cohomology basis,
\cite[proof of Proposition~7.2]{hugtenburg2024examples} gives
\[
 h^{\star n}=h^n+2q\,1_X.
\]
Thus \eqref{eq:even-quadric-small-character} identifies the character at zero bulk
with evaluation at \(x=0\) on
\(\R[[q]][x]/(x^{n+1}-4qx)\).

\begin{lm}\label[lm]{lm:even-quadric-primitive-products}
For \(u\in\mathcal A\), we have
\[
 p\star u=\mathcal L(u)p,\qquad
 p\star p=(-1)^{n/2}\sum_{i=0}^n\mathcal L(h^i)h^{n-i}.
\]
\end{lm}
\begin{proof}
By \cite[Lemma~7.3]{hugtenburg2024examples}, closed invariants with
three primitive insertions and all remaining insertions ambient vanish.
Since $t_L=0$, this gives $\Phi_{ppp}=0$.
The result follows from \cref{lm:boundary-primitive-products} and
\eqref{eq:even-quadric-metric}.  In particular, the coefficient $D$ of
\eqref{eq:boundary-character-discriminant-coefficient} is
\[
 D=\mathcal L(p\star p)
 =(-1)^{n/2}\sum_{i=0}^n(\partial_i c_0)(\partial_{n-i}c_0).
\]
\end{proof}

\begin{lm}\label[lm]{lm:even-quadric-three-boundary-character}
In \(Q_W\), one has
\[
\begin{aligned}
 \mathcal L(p\star p)
 &=2[s^3/3!]\overline\Omega\\
 &=2\sum_{d,l\ge0}\frac{q^d}{l!}
   \sum_{i_1,\ldots,i_l=0}^n
   \ogw_{d,3}(\Gamma_{i_1},\ldots,\Gamma_{i_l})
   t_{i_1}\cdots t_{i_l}.
\end{aligned}
\]
For \(l=0\), the inner summand is \(\ogw_{d,3}\).
\end{lm}
\begin{proof}
By \cref{lm:boundary-three-constraint-equation},
\[
 D\,\partial_i\bigl([s^3/3!]\overline\Omega\bigr)
 =\bigl([s^3/3!]\overline\Omega\bigr)\partial_iD,
 \qquad 0\le i\le n.
\]
The zero axiom gives $c_0|_{q=0}=-t_0$ and
$\left.([s^3/3!]\overline\Omega)\right|_{q=0}=0$.
Thus both $D$ and $[s^3/3!]\overline\Omega$ are divisible by $q$.
Moreover, \eqref{eq:even-quadric-small-character}, the degree axiom,
and \cref{lm:even-quadric-seed} give
\[
 D|_{t=0}=4(-1)^{n/2+1}q,\qquad
 \left.([s^3/3!]\overline\Omega)\right|_{t=0}=2(-1)^{n/2+1}q.
\]
Hence $D/q$ is a unit in $Q_W$.  After canceling the common factor
$q$, the quotient $([s^3/3!]\overline\Omega)/D$ belongs to $Q_W$.
The differential identity above implies that it is independent of
the bulk variables, and its value at $t=0$ is $1/2$.
Therefore $D=2[s^3/3!]\overline\Omega$.
The second equality follows from the definition of the open invariants
in \eqref{eq:intro-open-invariants}.
\end{proof}

The proof of \cref{lm:even-quadric-three-boundary-character}
establishes the factorization \eqref{eq:intro-primitive-square}.
The formal square root \(\vartheta\) introduced there has the binomial expansion
\[
 \vartheta=\sum_{m\ge0}\binom{1/2}{m}H^m
 =1+\frac12H-\frac18H^2+\cdots\in Q_W
\]
which converges in the coefficient filtration and satisfies
\(\vartheta^2=1+H\) and \(\vartheta|_{t=0}=1\).
In particular, \(\vartheta\) is a unit already in \(Q_W\), and
its expansion uses only nonnegative integer powers of the bulk variables.

\begin{prop}[Extensions of the boundary character]
\label[prop]{prop:even-quadric-character-extensions}
The maps \(\mathcal L_\pm\) of
\cref{thm:intro-quadric-characters} are distinct unital degree-zero
characters extending \(\mathcal L\).  The only root adjoined is
\(\zeta\), with \(\zeta^2=\varepsilon q\).
\end{prop}
\begin{proof}
By \cref{lm:even-quadric-primitive-products,lm:even-quadric-three-boundary-character},
\[
 \Phi_{ppp}=0,\qquad D=2[s^3/3!]\overline\Omega
 =4\varepsilon q\vartheta^2=(2\zeta\vartheta)^2.
\]
Thus \cref{prop:boundary-character-extensions} gives the characters
$\mathcal L_\pm(\alpha+fp)=\mathcal L(\alpha)\pm2f\zeta\vartheta$.
Since $|D|=|q|=2n$, we have $|\vartheta|=0$ and
$|2\zeta\vartheta|=n=|p|$, so the maps have degree zero.
Finally, $\vartheta$ is a unit and
$\zeta^2=\varepsilon q\ne0$.  The two maps therefore have opposite
nonzero values on $p$.
\end{proof}

\begin{rem}
\label[rem]{rem:even-quadric-primitive-bulk-extension}
Extending \(\mathcal L_\pm\) while retaining \(t_L\) in the
coefficient algebra involves convergence
issues that are resolvable but outside the scope of this manuscript.
Nevertheless, \cref{prop:even-quadric-character-extensions} gives
unital degree-zero characters
\[
 (U\otimes Q_U,\star_U)
 \xrightarrow{\Id\otimes\rho^*}(U\otimes Q_W,\star)
 \xrightarrow{\mathcal L_\pm}Q_W[\zeta].
\]
These characters factor through the algebra obtained by setting
\(t_L=0\).  We do not assert an extension retaining \(t_L\) here.
\end{rem}

\begin{rem}
\label[rem]{rem:even-quadric-complex-character-geometry}
There is interesting geometry behind the coefficient extension in
\cref{prop:even-quadric-character-extensions}.  In dimension \(n=4\),
the spin and orientation calculation in \cref{lm:even-quadric-seed}
forces the minus sign \(\ogw_{1,3}=-2\).  Consequently, at zero bulk,
\cref{lm:even-quadric-three-boundary-character} gives
\[
 \mathcal L(p\star p)=-4q,\qquad
 \mathcal L_\pm(p)=\pm2i\sqrt q.
\]
Thus, after adjoining and inverting \(\sqrt q\), quantum multiplication
by the normalized primitive class \(p/(2\sqrt q)\) has eigenvalues
\(\pm i\) on these two character factors.  In particular, splitting
this part of quantum cohomology into scalar character factors requires
adjoining complex coefficients.

An independent geometric manifestation occurs in the identification
\(Q^4\cong\operatorname{Gr}(2,4)\).  Nohara--Ueda computed the Floer
cohomology of a Lagrangian \(L'\cong U(2)\cong S^1\times S^3\):
the bounding cochains \(b_\pm=\pm(\pi i/2)e_1\), where
\(e_1\) generates \(H^1(L';\Z)\), have holonomy \(\pm i\),
zero potential, and nonzero Floer cohomology
\cite[Theorem~1.2]{nohara2016floer}.
If these cochains extend to compatible point-like bounding pairs
satisfying the hypotheses of \cref{thm:open-character,thm:euler-character}
on the full quantum cohomology, the resulting characters have value
zero on \(c_1(TQ^4)=4h\).  The ring relation
\eqref{prop:even-quadric-closed-quantum-ring} then forces their values
on \(p\) to be \(\pm2i\sqrt q\), realizing the same complex characters
geometrically.

The Nohara--Ueda Lagrangian also gives a more satisfying geometric
explanation of the square root itself.  Its two basic disk classes
\(\beta_1,\beta_2\) have Maslov index \(4\) and together represent
a line class \(\ell\), whose relative Maslov index is
\(2c_1(TQ^4)[\ell]=8\); see \cite[Section~4.3]{nohara2016floer}.
At the monotone fibre, the Maslov-index--area relation therefore gives
\(\omega(\beta_i)=\tfrac12\omega(\ell)\).
After passing to Novikov weights determined by symplectic area, with
the line assigned weight \(q\), each basic disk has weight \(\sqrt q\).
Together with the holonomies \(\pm i\), this explains geometrically
the factor \(i\sqrt q\) appearing in the character values.

The spin calculation for the Lagrangian sphere and the
\(S^1\times S^3\) calculation therefore give two independent routes,
the latter conditional on this bounding-pair extension, by which the
theory detects the need to adjoin complex values in order to split
quantum cohomology.
\end{rem}

\appendix
\raggedbottom

\section{Filtrations and scalar extension}
\label{app:scalar-extension}

Appendices \ref{app:scalar-extension}--\ref{app:standard-arguments}
extend \cite{solomon2019relative,solomon2016differential} to
non-commutative coefficients, with the conventions of
\cref{rem:conventions-st}.

We follow \cite[Sections~1.3 and~3]{kosloff2025pointlikenoncommutativef}
for filtrations and \cite[Sections~1.3.1 and~2.3]{solomon2019relative}
for scalar extension.  For the pairing, see
\cite[Proposition~2.15 and Lemmas~2.19--2.21]{kosloff2025pointlikenoncommutativef}.

\subsection{Filtrations and completions}\label{subsec:filtrations}

\begin{dfn}[Filtered completions]
\label{dfn:filtered-modules-completions}
For filtered rings and modules, we use
\cite[Definition~2.3]{kosloff2025pointlikenoncommutativef}; in particular,
a filtration \(\varsigma\) on \(M\) satisfies
\(\varsigma(x)=\infty\) if and only if \(x=0\).  Write
\[
 F_{\varsigma}^{a}M:=\{x:\varsigma(x)\geq a\},
 \qquad
 F_{\varsigma}^{>a}M:=\{x:\varsigma(x)>a\}.
\]
The module is \emph{complete} when the natural map
\[
 M\longrightarrow
 \varprojlim_{E\to+\infty}M/F_{\varsigma}^{>E}M
\]
is an isomorphism.  A linear map \(F\colon M\to M'\) is \emph{bounded} if,
for every \(E>0\), there is \(K>0\) such that
\[
 F\bigl(F_{\varsigma_M}^{>K}M\bigr)
 \subset F_{\varsigma_{M'}}^{>E}M'.
\]
Equivalently, \(F\) is continuous.  The tensor-product filtration is
\[
 F_{\varsigma}^{a}(M_1\otimes\cdots\otimes M_r)
 :=
 \sum_{a_1+\cdots+a_r\geq a}
 F_{\varsigma_1}^{a_1}M_1\otimes\cdots\otimes
 F_{\varsigma_r}^{a_r}M_r.
\]
Write \(\widehat\otimes\) for the completed tensor product; direct sums
are completed likewise.  We omit the hats in the main text.
\end{dfn}

\begin{dfn}[Novikov filtration]
\label{dfn:regular-novikov-filtration}
For a nonzero Novikov series, set
\[
 \nu\!\left(\sum_\beta a_\beta T^\beta\right)
 :=\inf_{a_\beta\neq0}\omega(\beta),
 \qquad \nu(0):=\infty.
\]
Use the same filtration on \(\Lambda_c\).  Both rings are complete by \cref{dfn:novikov-coefficient-rings}, and
\(\Lambda_c\hookrightarrow\Lambda\) preserves the filtration.
\end{dfn}

\begin{dfn}[Coefficient filtrations]
\label{dfn:coefficient-filtrations}
On \(\Lambda_c\otimes_\kk P_V\), write
\(f=\sum_{\beta,d}T^\beta f_{\beta,d}\), where \(f_{\beta,d}\) has total
bulk order \(d\), and set
\[
 \nu(f):=\inf_{f_{\beta,d}\neq0}\bigl(\omega(\beta)+d\bigr),
 \qquad \nu(0):=\infty.
\]
Use the same formula on \(K_L\otimes_\kk\Lambda\otimes_\kk P_W\),
with \(K_L\) trivially filtered.  Thus
\[
 \nu(T^\beta t^I)=\omega(\beta)+|I|,
 \qquad \nu(t_i)=1
\]
for every bulk coordinate, including that dual to \(1_X\).
Completion gives \(Q_V,Q_L,R_W^0\) of
\cref{dfn:filtered-coefficient-algebras}.

For \(S=\kk\langle s\rangle\), put on \(R_W^0\otimes_\kk S\)
\[
 \varsigma_1\!\left(\sum_{j=0}^{m}r_js^j\right)
 :=\inf_{r_j\neq0}\nu(r_j).
\]
Its completion is
\begin{equation}\label{eq:restricted-s-completion}
 R_W=
 \left\{
  \sum_{j\geq0}r_js^j
  \ \middle|
  r_j\in R_W^0,\quad
  \#\{j:\nu(r_j)\leq E\}<\infty
  \text{ for every }E>0
 \right\}.
\end{equation}

Give \(S\) the descending filtration \(\nu_S(s^j)=j(n-1)\), as in
\cite[Section~3.8]{kosloff2025pointlikenoncommutativef}, and put
\[
 \varsigma_2:=\nu_S\otimes\nu,
 \qquad
 \varsigma_2\!\left(\sum_{j\geq0}r_js^j\right)
 =\inf_{r_j\neq0}\bigl(\nu(r_j)+j(n-1)\bigr).
\]
On \(Q_V,Q_L,R_W^0\), set \(\varsigma_1=\varsigma_2=\nu\).  Then
\[
 \varsigma_1(t_i)=\varsigma_2(t_i)=1,
 \qquad
 \varsigma_1(s)=0,
 \qquad
 \varsigma_2(s)=n-1,
 \qquad
 \varsigma_2\geq\varsigma_1.
\]
The rings are \(\varsigma_1\)-complete, and
\[
 \mI_{Q_V}\subset F_{\varsigma_1}^{>0}Q_V,
 \qquad
 \mI_{R_W}\subset F_{\varsigma_1}^{>0}R_W,
 \qquad
 (s)+\mI_{R_W}\subset F_{\varsigma_2}^{>0}R_W,
\]
where the last inclusion uses \(n>1\).
\end{dfn}

\begin{dfn}[Pseudo-completeness]
\label{dfn:qkl-pseudo-completeness}
Let \((M_X,\nu)\) be complete, and let \(M_L\) be complete with respect to
\(\varsigma_1\) and carry \(\varsigma_2\geq\varsigma_1\).  For maps
\[
 F_{k,l}\colon M_X^{\otimes l}\otimes M_L^{\otimes k}\longrightarrow M_L,
 \qquad k,l\geq0,
\]
let \(\varsigma_2^\otimes\) be the tensor-product filtration given by
\(\nu\) on \(M_X\) and \(\varsigma_2\) on \(M_L\).  We call the
family \emph{pseudo-complete with respect to \(\varsigma_2\)} if, for every
\(E>0\) and \(D\in\Z\), there is \(K>0\), independent of \(k,l\), such that
every homogeneous input \(\xi\) satisfies
\[
 \varsigma_2^\otimes(\xi)>K,
 \qquad |F_{k,l}(\xi)|\geq D
 \quad\Longrightarrow\quad
 \varsigma_1\bigl(F_{k,l}(\xi)\bigr)>E.
\]
This is the open--closed form of
\cite[Definition~3.1]{kosloff2025pointlikenoncommutativef}.
The target may be replaced by any \(\varsigma_1\)-complete graded module.
\end{dfn}

\begin{rem}\label{rem:pseudo-completeness-continuity}
For each lower bound on the output degree, pseudo-completeness is
continuity from \(\varsigma_2^\otimes\) to \(\varsigma_1\), uniformly in
\(k,l\); see
\cite[Remark~3.2]{kosloff2025pointlikenoncommutativef}.
\end{rem}

Following \cite[Equation~(26) and Notation~4.14]
{kosloff2025pointlikenoncommutativef}, for homogeneous elements
\(\xi_1,\ldots,\xi_m\) of the completed coefficient or cochain modules,
write
\[
 \xi_i=\sum_{j\geq0}a_{ij}\lambda_{ij}s^{d_{ij}},
 \qquad
 \lambda_{ij}=T^{\beta_{ij}}t^{I_{ij}},
 \qquad \nu(\lambda_{ij})\longrightarrow\infty,
\]
where \(a_{ij}\) has no formal variables and \(d_{ij}=0\) for closed inputs.

\begin{notn}[Coefficient restrictions]
\label{dfn:adapted-coefficient-restriction}
Denote by \(A(\xi_1,\ldots,\xi_m)\) the unital \(\kk\)-subalgebra of
\(\Lambda\otimes_\kk P_W\) generated by
\[
 \{\lambda_{ij}\}_{i,j},
 \qquad \{t_1,\ldots,t_N\},
 \qquad \{T^\beta\}_{\beta\in\Pi_{\mathrm{op}}}.
\]
Here \(t_i\) are the bulk coordinates and \(\Pi_{\mathrm{op}}\) consists of
the classes of all nonzero ordinary, closed, and scalar
operations, including interval operations when used.
For closed classes \(B\), use \(\varpi(B)\).
\end{notn}

Put \(A=A(\xi_1,\ldots,\xi_m)\).  We write \(Q_W^\diamond\) for the
\(\nu\)-completion of \(A\cap(\Lambda_c\otimes_\kk P_W)\), and
\((R_W^0)^\diamond\) for the \(\nu\)-completion of
\(K_L\otimes_\kk A\).  The \(\varsigma_1\)-completion of
\((R_W^0)^\diamond\otimes_\kk\kk\langle s\rangle\) is denoted by
\(R_W^\diamond\).  These rings and their completed cochain modules,
with the induced gradings and filtrations, form the \emph{coefficient restriction}.

\begin{lm}\label{lm:adapted-restriction-exists}
The restricted modules contain \(\xi_1,\ldots,\xi_m\).  Any two
restrictions have a common restriction containing both, with
filtration-preserving inclusions.  For each filtration bound, the degrees
of nonzero coefficient monomials in \((R_W^0)^\diamond\) are bounded above.
\end{lm}
\begin{proof}
The expansions and energy axioms give finitely many nonconstant
generators below each filtration bound, with positive filtrations
bounded away from zero.  Products of bounded filtration have boundedly
many factors from a finite set.  Thus there are finitely many such
monomials, and their degrees are bounded.
Each \(T^{\beta_{ij}}t^{I_{ij}}\) is kept as one generator.
Completion adds no new monomials, and \(K_L\) has degree zero.
Thus \(A\) is sababa in the sense of
\cite[Definition~3.3]{kosloff2025pointlikenoncommutativef}.

The completed modules contain the chosen inputs and inherit the
filtrations.  The structure maps preserve them because their weights lie
in \(A\); closed maps preserve \(Q_W^\diamond\).
The union of two generating sets has the same finiteness property and
gives a common restriction.
\end{proof}

\begin{lm}\label{lm:local-uniform-estimate}
Consider a family \(F_{k,l}\) as in \cref{dfn:qkl-pseudo-completeness} whose target
is a completed scalar extension over a coefficient restriction.
Suppose the target degrees before scalar extension are bounded above and
\begin{equation}\label{eq:local-sigma-two-bound}
 \varsigma_2\bigl(F_{k,l}(\xi)\bigr)\ge\varsigma_2^\otimes(\xi).
\end{equation}
Then the family is pseudo-complete.
\end{lm}
\begin{proof}
Let \(E>0\) and \(D\in\Z\), and let \(h\) bound the target degrees
before scalar extension.  By \cref{lm:adapted-restriction-exists}, there
is \(M\) such that every nonzero coefficient monomial \(\lambda\) with
\(\nu(\lambda)\le E\) satisfies \(|\lambda|\le M\).  Choose \(K>0\) with
\(K\ge E+h+M-D\).  It depends only on \(E\) and \(D\).

Let \(\xi\) be a homogeneous input with \(\varsigma_2^\otimes(\xi)>K\) and
\(|F_{k,l}(\xi)|\ge D\), and write
\[
 F_{k,l}(\xi)=\sum_{j\ge0}a_j\lambda_j s^{d_j},
\]
where \(a_j\) has no formal variables and \(\lambda_j\) is a coefficient
monomial.  Each term has degree
\(|a_j|+|\lambda_j|-d_j(n-1)=|F_{k,l}(\xi)|\ge D\).  If
\(\nu(\lambda_j)\le E\), then \(|a_j|\le h\) and \(|\lambda_j|\le M\), so
\[
 \nu(\lambda_j)+d_j(n-1)\le E+h+M-D\le K.
\]
On the other hand, \eqref{eq:local-sigma-two-bound} gives
\[
 \nu(\lambda_j)+d_j(n-1)
 \ge\varsigma_2\bigl(F_{k,l}(\xi)\bigr)
 \ge\varsigma_2^\otimes(\xi)>K.
\]
Hence \(\nu(\lambda_j)>E\) for every \(j\), and so
\(\varsigma_1\bigl(F_{k,l}(\xi)\bigr)>E\).
\end{proof}

\begin{rem}\label{rem:de-rham-sababa-restriction}
After scalar extension, the coefficient of the output of each operation
or pairing is the product of the coefficients of the inputs and a weight
\(T^\beta\) with \(\nu(T^\beta)\ge0\).  Hence each such map \(F\)
satisfies
\[
 \varsigma_i\bigl(F(\xi)\bigr)\ge\varsigma_i^\otimes(\xi),
 \qquad i=1,2.
\]
Before scalar extension, the outputs of \(\q_{k,l}\) have degree at
most \(n\), and the pairing \(\langle\cdot,\cdot\rangle_L\) takes values
of degree \(0\).  So the families \(\q_{k,l}\) and
\(\langle\q_{k,l}(a_1,\ldots,a_k;\eta),a_{k+1}\rangle_L\) are
pseudo-complete by \cref{lm:local-uniform-estimate}.  For the pairing,
we apply the lemma before projecting to \(HH_0(R_W)\), and use that
\(\pi_{HH}\) preserves both filtrations by \cref{lm:hh0-filtered-complete}.
The inputs of \(\q_{\emptyset,l}\) and \(\q_{-1,l}\) lie in
\(C_X^{Q_W}\), which carries only the filtration \(\nu\).  So
\(\varsigma_1^\otimes(\xi)=\varsigma_2^\otimes(\xi)\), and the case
\(i=1\) of the inequality above already shows that these families are
pseudo-complete.  For the interval operations and pairing of the
de~Rham model in \cref{rem:interval-de-rham-model}, the forms live on
\([0,1]\times L\) and \([0,1]\), so the degree bounds are \(n+1\) and
\(1\), and the same argument applies.  This is the argument of
\cite[Proposition~3.8 and Lemma~4.13]{kosloff2025pointlikenoncommutativef}.
\end{rem}

\begin{rem}\label{rem:monotone-coefficient-bound}
If \(\mu_L=\lambda\omega\) with \(\lambda>0\), then
\[
 |T^\beta t^I|\le
 \max\{\lambda,|t_1|,\ldots,|t_N|\}\,\nu(T^\beta t^I).
\]
Thus bounded filtration bounds coefficient degrees on all of \(R_W^0\),
and \cref{lm:local-uniform-estimate} applies without restricting it.
\end{rem}

Give \(HH_0(R_W)\) the quotient filtrations
\[
 F_{\varsigma_i}^{a}HH_0(R_W)
 :=\pi_{HH}\bigl(F_{\varsigma_i}^{a}R_W\bigr),
 \qquad i=1,2,
\]
where \(\pi_{HH}\colon R_W\to HH_0(R_W)\) is the quotient map.

\begin{lm}
\label{lm:hh0-filtered-complete}
The subspace \([R_W,R_W]_{\mathrm{gr}}\) is closed for \(\varsigma_1\).
The quotient \(HH_0(R_W)\) is complete for \(\varsigma_1\), and
\[
 \bigcap_{E>0}F_{\varsigma_2}^{>E}HH_0(R_W)=\{0\}.
\]
The map \(\pi_{HH}\) preserves both filtrations.
Here \([R_W,R_W]_{\mathrm{gr}}\) denotes the algebraic span of commutators.
\end{lm}
\begin{proof}
By \cite[Corollary~3.29]{kosloff2025pointlikenoncommutativef}, the
commutator subspace is zero if \(|s|\) is even, while if \(|s|\) is odd,
\[
 [R_W,R_W]_{\mathrm{gr}}
 =\widehat{\bigoplus}_{j\geq1}^{\,\varsigma_1}R_W^0s^{2j}.
\]
Indeed, split coefficients by parity and remove one factor of \(s\).
Every series on the right is then a sum of at most two commutators in
\(R_W\), by \eqref{eq:restricted-s-completion}.
This subspace is closed.  Its quotient consists of series in
\(1,s,s^3,\ldots\), with the induced filtrations, proving the claims.
For \(|s|\) even, \(HH_0(R_W)=R_W\).
\end{proof}

\begin{lm}
\label{lm:s-degree-decomposition}
Every \(a\in R_W\) has a unique expansion
\[
    a=\sum_{k\ge0}a_k,\qquad a_k\in R_W^0s^k,
\]
converging for \(\varsigma_1\).  The projections descend to
\(HH_0(R_W)\) and \(\Coker\mathfrak i\); write \([\xi]_{s^k}\) for the
components of a class \(\xi\).  They determine \(\xi\) and sum to it.
The quotient map \(HH_0(R_W)\to\Coker\mathfrak i\) is an isomorphism
in every positive \(s\)-degree.
\end{lm}
\begin{proof}
Use \eqref{eq:restricted-s-completion} and
\cref{lm:hh0-filtered-complete}.  Since
\(\operatorname{Im}\mathfrak i\subset Q_L\) has \(s\)-degree zero,
passing to \(\Coker\mathfrak i\) changes only that component.
\end{proof}

\begin{lm}
\label{lm:s-degree-one-detector}
For \(r\in R_W^0\) and \(k\ge0\),
\begin{equation}\label{eq:s-euler-counts}
 \mathcal E_s(rs^k)=krs^k,
 \qquad
 \left.\mathcal E_s\right|_{HH_0(R_W)_{s^k}}=k\,\Id,
 \qquad
 \left.\mathcal E_s\right|_{(\Coker\mathfrak i)_{s^k}}=k\,\Id.
\end{equation}
In particular, the kernel of \(\mathcal E_s\) on either quotient is its
\(s\)-degree-zero component.  Moreover, for every $k\ge1$ such that
$|s|$ is even or $k$ is odd,
\[
 R_W^0\longrightarrow(\Coker\mathfrak i)_{s^k},
 \qquad q\longmapsto q[s^k],
\]
is an isomorphism.  If $|s|$ is odd and $k\ge2$ is even, then
$(\Coker\mathfrak i)_{s^k}=0$.
\end{lm}
\begin{proof}
The Leibniz rule gives \(\mathcal E_s(rs^k)=krs^k\).  This descends
to both quotients by \cref{lm:s-degree-decomposition}.
Since \(\operatorname{char}\kk=0\), its kernel consists of the
degree-zero classes.  By \cref{lm:hh0-filtered-complete}, every surviving
positive component is free over \(R_W^0\) with basis \([s^k]\).
\end{proof}

\subsection{Continuity of the coefficient operations}
\label{app:derivation-continuity}

Derivations of \(R_W\) are continuous by assumption.  In particular,
\(\partial_{w_i}\) lowers \(\nu\) by at most one,
\(\partial_s\) preserves \(\varsigma_1\) and lowers
\(\varsigma_2\) by at most \(n-1\), and \(\mathcal E_s\) preserves both
filtrations.  The divisor derivation \(E_A\) preserves \(\nu\).
Continuity is preserved by \(Q_W\)-multiples, brackets, and passage to
\(HH_0(R_W)\).

Give \(C_X,C_L\) the trivial filtration and complete their scalar
extensions.  The extension rules and nonnegative Novikov weights give
\begin{equation}\label{eq:ordinary-main-filtered}
\begin{aligned}
 \varsigma_1\bigl(\q_{k,l}(a_1,\ldots,a_k;\eta_1,\ldots,\eta_l)\bigr)
 &\ge \sum_{i=1}^k\varsigma_1(a_i)+\sum_{j=1}^l\nu(\eta_j),\\
 \varsigma_1\bigl(\langle x,y\rangle_L\bigr)
 &\ge \varsigma_1(x)+\varsigma_1(y).
\end{aligned}
\end{equation}
The same bound holds for closed and scalar operations, with only closed
inputs.  These bounds are uniform in the numbers of inputs.
Write \(\varsigma_1^\otimes\) for the induced filtration on tensors of inputs.
Energy finiteness gives convergence of the sums in
\cref{dfn:novikov-summed-structure-maps} in their completed targets.

\subsection{Scalar extension of the structure maps and pairing}
\label{app:subsec-scalar-extension}

\begin{dfn}[Scalar extension of the structure maps]
\label{dfn:scalar-extension-structure-maps}
The structure maps of \cref{dfn:structure-maps} extend to
bounded maps
\[
\q_{k,l}^\beta:(C_X^{Q_W})^{\otimes l}\otimes (C_L^{R_W})^{\otimes k}\to C_L^{R_W},
\]
\[
\q_{\emptyset,l}^A:(C_X^{Q_W})^{\otimes l}\to C_X^{Q_W},
\]
\[
\q_{-1,l}^\beta:(C_X^{Q_W})^{\otimes l}\to HH_0(R_W),
\]
with \(K_L\subset Q_L\subset HH_0(R_W)\).  For homogeneous
\(r_i\in R_W\), \(b_j\in Q_W\), and \(\eta_j\in C_X\), the extensions are
given by the following rules.  In \textup{(E1)}, \(a_i\in C_L\); in
\textup{(E2)}, \(a_i\in C_L^{R_W}\).
\begin{enumerate}
\item[\textup{(E1)}]
\[
\q_{k,l}^\beta(a_1r_1,\ldots,a_kr_k;\eta_1,\ldots,\eta_l)
=
(-1)^{\sum_i|r_i|\sum_{j>i}(|a_j|-1)}
\q_{k,l}^\beta(a_1,\ldots,a_k;\eta_1,\ldots,\eta_l)\,r_1\cdots r_k,
\]
\item[\textup{(E2)}]
\[
\q_{k,l}^\beta(a_1,\ldots,a_k;\eta_1b_1,\ldots,\eta_lb_l)
=
(-1)^{\sum_i|b_i|\left(\sum_{j>i}|\eta_j|+\sum_j(|a_j|-1)\right)}
\q_{k,l}^\beta(a_1,\ldots,a_k;\eta_1,\ldots,\eta_l)\,b_1\cdots b_l,
\]
\item[\textup{(E3)}]
\[
\q_{\emptyset,l}^A(\eta_1b_1,\ldots,\eta_lb_l)
=
(-1)^{\sum_i|b_i|\sum_{j>i}|\eta_j|}
\q_{\emptyset,l}^A(\eta_1,\ldots,\eta_l)\,b_1\cdots b_l,
\]
\item[\textup{(E4)}]
Multiply left coefficients in \(R_W\) before passing to \(HH_0(R_W)\):
\[
\q_{-1,l}^\beta(b_1\eta_1,\ldots,b_l\eta_l)
=
(-1)^{\sum_i|b_i|\sum_{j<i}|\eta_j|}
\bigl[b_1\cdots b_l\,\q_{-1,l}^\beta(\eta_1,\ldots,\eta_l)\bigr].
\]
\end{enumerate}
Rule \textup{(E1)} is well defined on \(C_L\otimes_{K_L}R_W\), since the
maps of \cref{dfn:structure-maps} are \(K_L\)-multilinear in the \(C_L\)
inputs and \(K_L\subset\mathcal Z(R_W)\).
Extend the period pairing \(Q_W\)-linearly to
\[
(\mathcal D_X^{\mathrm{adm}})^{Q_W}
:=\bigl(\mathcal D_X^{\mathrm{adm}}\widehat\otimes_\kk Q_W\bigr)^2
\subset C_X^{Q_W}.
\]
For \(\theta\in(\mathcal D_X^{\mathrm{adm}})^{Q_W}\), the divisor axiom gives
\[
\q_{k,l}^{\beta}(a_1,\ldots,a_k;\theta,\eta_2,\ldots,\eta_l)
=
\left(\int_\beta\theta\right)
\cdot
\q_{k,l-1}^{\beta}(a_1,\ldots,a_k;\eta_2,\ldots,\eta_l)
\]
for \(k\ge0\) and \(l\ge1\), while the scalar operation satisfies
\[
\q_{-1,l}^{\beta}(\theta,\eta_2,\ldots,\eta_l)
=
\left(\int_\beta\theta\right)
\cdot
\q_{-1,l-1}^{\beta}(\eta_2,\ldots,\eta_l).
\]
The closed operations satisfy
\[
\q_{\emptyset,l}^{A}(\theta,\eta_2,\ldots,\eta_l)
=
\begin{cases}
\theta\wedge\eta_2,
&(A,l)=(A_0,2),\\[2pt]
\left(\int_A\theta\right)
\cdot
\q_{\emptyset,l-1}^{A}(\eta_2,\ldots,\eta_l),
&\text{otherwise.}
\end{cases}
\]
\end{dfn}

\begin{rem}\label{rem:scalar-extension-shifts}
The signs in \textup{(E1)}--\textup{(E4)} are the Koszul signs for
\[
C_X^{Q_W}[2],
\qquad C_L^{R_W}[1],
\qquad HH_0(R_W)[3-n],
\]
and the map
\begin{equation}\label{eq:tensor-to-hh0}
v\otimes r\longmapsto(-1)^{(n+1)|r|}[vr]\in HH_0(R_W)
\end{equation}
for homogeneous \(v\in K_L\) and \(r\in R_W\).
The ordinary open, closed, and scalar component operators have shifted
degrees \(1-\mu_L(\beta)\), \(2-2c_1(A)\), and \(-\mu_L(\beta)\),
respectively.  Use degrees \(|\eta|-2\) and \(|a|-1\) when moving coefficients past
inputs, and multiply coefficients in input order.
For a homogeneous ordinary structure map \(F\) independent of formal
coefficients and a homogeneous coefficient derivation \(\delta\),
\[
\delta_0\bigl(F(z_1,\ldots,z_N)\bigr)
=
\sum_{i=1}^{N}
(-1)^{|\delta|\bigl(|F|+\sum_{h<i}|z_h|_{\mathrm{sh}}\bigr)}
F(z_1,\ldots,\delta_i z_i,\ldots,z_N).
\]
Here \(\delta_0\) acts on the target and \(\delta_i\) on the \(i\)-th
source.  Thus \([\delta,F]=0\) for the graded tensor-product action.

Since \(K_L\) has degree zero, the \(Q_W\)-action on
\(K_L[-m]\otimes_{K_L}R_W\) is
\[
 q\cdot(v\otimes r)=(-1)^{m|q|}v\otimes qr
 \qquad(q\in Q_W\text{ homogeneous})
\]
for \(m\in\mathbb Z\).  Combining this with
\eqref{eq:tensor-to-hh0} gives
\begin{equation}\label{eq:scalar-target-coefficient-actions}
\begin{aligned}
 q\cdot_{[-n]}\xi
 &=(-1)^{|q|}q\xi,\\
 q\cdot_{[-n-1]}\xi
 &=q\xi,
\end{aligned}
\end{equation}
where the products on the right are induced by multiplication in
\(R_W\).  Likewise for \(Q_L\).
\end{rem}

\begin{dfn}
\label{dfn:scalar-extension-interval-pairing}
The pairing extends with degree \(-n\) to
\[
\langle\cdot,\cdot\rangle_L\colon
C_L^{R_W}\otimes C_L^{R_W}\longrightarrow HH_0(R_W),
\]
defined with coefficients on the left by
\[
\langle r_1a,\,r_2b\rangle_L
:=
(-1)^{|r_2|(|a|-1)}
\bigl[r_1r_2\,\langle a,b\rangle_L\bigr],
\]
for homogeneous \(a,b\in C_L\) and \(r_1,r_2\in R_W\), with the signs
of \cref{rem:scalar-extension-shifts} and multiplication in \(R_W\).
\end{dfn}

\section{Deformed operations}
\label{subsec:bg-deformations}

We follow \cite[Section~3.5]{kosloff2025pointlikenoncommutativef},
with the insertion formulas of \cite[Section~3.4]{solomon2019relative}.
Use the filtrations of \cref{dfn:coefficient-filtrations} and fix
\[
 \gamma\in(C_X^{Q_W})^2,\qquad b\in(C_L^{R_W})^1,
 \qquad \nu(\gamma)>0,\qquad\varsigma_2(b)>0.
\]
\begin{lm}[Convergence of insertion sums]\label{lm:insertion-series}
Let \(d\in\Z\), and let \(F_{k,l}\) be a pseudo-complete family of
degree \(d-k-2l\) on a coefficient restriction containing \(\gamma,b\).
For homogeneous inputs \(a_1,\ldots,a_k,\eta_1,\ldots,\eta_l\)
in that restriction, put
\[
\begin{aligned}
 F_{\boldsymbol\ell,p}(a;\eta)
 :=F_{k+|\boldsymbol\ell|,l+p}
 \bigl(b^{\otimes\ell_0},a_1,b^{\otimes\ell_1},\ldots,
 a_k,b^{\otimes\ell_k};\eta_1,\ldots,\eta_l,\gamma^{\otimes p}\bigr),
\end{aligned}
\]
where \(\boldsymbol\ell=(\ell_0,\ldots,\ell_k)\in\Z_{\ge0}^{k+1}\),
\(p\ge0\), and \(|\boldsymbol\ell|=\sum_i\ell_i\).  Then
\begin{equation}\label{eq:insertion-tail}
 \lim_{r\to\infty}\ 
 \min_{|\boldsymbol\ell|+p=r}
 \varsigma_1\bigl(F_{\boldsymbol\ell,p}(a;\eta)\bigr)=\infty.
\end{equation}
\end{lm}
\begin{proof}
Argue as in
\cite[Lemma~3.9]{kosloff2025pointlikenoncommutativef}.
Every nonzero summand has degree
\[
 D=d-k-2l+\sum_{i=1}^k|a_i|+\sum_{j=1}^l|\eta_j|.
\]
Fix \(E>0\), and let \(K>0\) be given by pseudo-completeness for
\((E,D)\).  Since \(\varsigma_2(b)>0\) and \(\nu(\gamma)>0\), there is
\(r\) such that
\[
 |\boldsymbol\ell|\,\varsigma_2(b)+p\,\nu(\gamma)
 +\sum_{i=1}^k\varsigma_2(a_i)+\sum_{j=1}^l\nu(\eta_j)>K
\]
whenever \(|\boldsymbol\ell|+p\ge r\).  The left side is a lower bound
for \(\varsigma_2^\otimes\) of the input of \(F_{\boldsymbol\ell,p}(a;\eta)\),
so \(\varsigma_1\bigl(F_{\boldsymbol\ell,p}(a;\eta)\bigr)>E\) for
\(|\boldsymbol\ell|+p\ge r\).
\end{proof}

\begin{dfn}[Open deformation]\label{dfn:bg-deformation-qkl}
On homogeneous inputs, set
\[
\begin{aligned}
&\q_{k,l}^{\gamma,b}
(a_1,\ldots,a_k;\eta_1,\ldots,\eta_l)\\
&\quad:=
\sum_{\substack{\ell_0,\ldots,\ell_k\ge0\\m\ge l}}
\frac{1}{(m-l)!}\,
\q_{k+\ell_0+\cdots+\ell_k,m}\bigl(
b^{\otimes\ell_0},a_1,b^{\otimes\ell_1},\ldots,
a_k,b^{\otimes\ell_k};
\eta_1,\ldots,\eta_l,\gamma^{\otimes(m-l)}
\bigr).
\end{aligned}
\]
The \emph{open deformation} is the bounded map of degree \(2-k-2l\),
\[
 \q_{k,l}^{\gamma,b}\colon
 (C_X^{Q_W})^{\otimes l}\otimes(C_L^{R_W})^{\otimes k}
 \longrightarrow C_L^{R_W},
\]
whose restriction to homogeneous inputs is the displayed sum.
\end{dfn}

\begin{dfn}[Closed deformation]\label{dfn:bg-deformation-qempty}
On homogeneous inputs, set
\[
\q_{\emptyset,l}^{\gamma}(\eta_1,\ldots,\eta_l)
:=
\sum_{m\ge l}\frac{1}{(m-l)!}\,
\q_{\emptyset,m}
(\eta_1,\ldots,\eta_l,\gamma^{\otimes(m-l)}).
\]
The \emph{closed deformation} is the bounded map of degree \(4-2l\),
\[
 \q_{\emptyset,l}^{\gamma}\colon
 (C_X^{Q_W})^{\otimes l}\longrightarrow C_X^{Q_W},
\]
whose restriction to homogeneous inputs is the displayed sum.
\end{dfn}

\begin{dfn}[Scalar deformation]\label{dfn:bg-deformation-qminusone}
On homogeneous inputs, set
\[
\begin{aligned}
\q_{-1,l}^{\gamma,b}(\eta_1,\ldots,\eta_l)
:={}&
\sum_{\substack{k\ge0\\m\ge l}}
\frac{
\left\langle
\q_{k,m}(b^{\otimes k};
\eta_1,\ldots,\eta_l,\gamma^{\otimes(m-l)}),b
\right\rangle_L
}{(m-l)!(k+1)}
\\
&+
\sum_{m\ge l}\frac{1}{(m-l)!}\,
\q_{-1,m}(\eta_1,\ldots,\eta_l,\gamma^{\otimes(m-l)}).
\end{aligned}
\]
The \emph{scalar deformation} is the bounded map of degree \(3-n-2l\),
\[
 \q_{-1,l}^{\gamma,b}\colon
 (C_X^{Q_W})^{\otimes l}\longrightarrow HH_0(R_W),
\]
whose restriction to homogeneous inputs is the displayed sum.
\end{dfn}

\begin{lm}\label{lm:qkl-bg-insertions}
\label{lm:bg-deformation-qkl}
\label{lm:bg-deformation-qminusone}
The deformations above exist uniquely.  Their insertion formulas hold
for all completed inputs, independently of the coefficient restriction,
and satisfy
\[
 \varsigma_1\bigl(F^{\gamma,b}(\xi)\bigr)
 \ge\varsigma_1^\otimes(\xi).
\]
\end{lm}
\begin{proof}
For homogeneous inputs, choose the restriction of
\cref{dfn:adapted-coefficient-restriction} containing them and
\(\gamma,b\).  By \cref{rem:de-rham-sababa-restriction},
\cref{lm:insertion-series} applies to \(\q_{k,l}\) and to
\[
 F_{k+1,l}(a_1,\ldots,a_{k+1};\eta)
 :=\langle\q_{k,l}(a_1,\ldots,a_k;\eta),a_{k+1}\rangle_L,
\]
of degrees \(2-k-2l\) and \(3-n-(k+1)-2l\), respectively.
For the closed and direct scalar sums, each additional \(\gamma\)
increases the input filtration by \(\nu(\gamma)>0\).
Hence all displayed sums converge on homogeneous inputs.

Let \(S_N\) be the partial sum with at most \(N\) insertions.  Since
\(\varsigma_1(b),\nu(\gamma)\ge0\),
\[
 \varsigma_1\bigl(S_N(\xi)\bigr)\ge\varsigma_1^\otimes(\xi)
 \qquad(N\ge0).
\]
As in the proofs of Kosloff--Solomon
\cite[Lemmas~3.11--3.12]{kosloff2025pointlikenoncommutativef}, this
uniform bound, completeness, give a unique bounded
extension to completed modules, with the same insertion formula.
Finally, two restrictions lie in a common one by
\cref{lm:adapted-restriction-exists}; their partial sums, and hence
their limits, agree.
\end{proof}

\begin{rem}\label{app:bounding-pair-well-definedness}
In particular, \(\q_{0,0}^{\gamma,b}\in(C_L^{R_W})^2\), so the
Maurer--Cartan equation is defined in \(C_L^{R_W}\).
\end{rem}

\begin{lm}\label{lm:bg-deformed-axioms}
Assume \(d_X\gamma=0\).  The families
\[
\{\q_{k,l}^{\gamma,b}\},\qquad
\{\q_{\emptyset,l}^{\gamma}\},\qquad
\{\q_{-1,l}^{\gamma,b}\}
\]
satisfy the structure, symmetry, unit, fundamental-class, and cyclic
identities after inserting \(\gamma,b\) in the axioms.
\end{lm}
\begin{proof}
Fix homogeneous inputs, a coefficient restriction containing them and
\(\gamma,b\), and \(E>0\).  For the indices in
\eqref{eq:qkl-structure}, expand the compositions of deformed operations
using \cref{dfn:bg-deformation-qkl}.
By \cref{lm:insertion-series}, only finitely many summands of
\(\q_{k_2,|J|}^{\gamma,b}\) are nonzero modulo
\(F_{\varsigma_1}^{>E}\).  The bound \eqref{eq:ordinary-main-filtered}
allows us to omit the others from the composition, since all other inputs
have nonnegative first filtration.
For each remaining homogeneous summand, \cref{lm:insertion-series} makes the
sum defining \(\q_{k_1,|I|}^{\gamma,b}\) finite modulo the same filtration.
Pairings satisfy the same assertion by \eqref{eq:ordinary-main-filtered};
closed and scalar operations use \cref{lm:qkl-bg-insertions}.
By \cref{lm:local-uniform-estimate}, the estimates hold uniformly for
\(T^\beta\q_{k,l}^\beta\) and their pairings.  For closed and direct scalar sums, use \(\nu(\gamma)>0\).
Thus the same truncation applies before summing over classes, and
energy finiteness leaves only finitely many classes.

Consequently, only finitely many summands are nonzero modulo
\(F_{\varsigma_1}^{>E}\), for every \(E>0\).
Applying
\cite[Section~1.1.8, Proposition~2 and Corollary~3]{bosch1984nonarchimedean}
to the complete norm \(x\mapsto e^{-\varsigma_1(x)}\), we may regroup
and interchange the insertion and Novikov sums, then apply the
undeformed identities.
Since \(|b|-1=|\gamma|-2=0\) and \(d_X\gamma=0\), the calculation of
\cite[Lemma~3.17]{solomon2019relative} applies with the coefficient signs
of Appendix~\ref{app:scalar-extension}.
The identities extend to completed inputs by boundedness and density.
\end{proof}

\begin{rem}\label{rem:local-restriction-compatibility}
For a continuous coefficient derivation \(\delta\) and the partial sums
\(S_N\) above,
\[
 \delta\bigl(F^{\gamma,b}(\xi)\bigr)
 =\lim_{N\to\infty}\delta\bigl(S_N(\xi)\bigr).
\]
If \(\delta\in\mathfrak g_{\mathrm{for}}\), the Leibniz rule acts on
the inputs and on the \(b,\gamma\) insertions, with the signs of
\cref{rem:scalar-extension-shifts}.  Apply \cref{lm:insertion-series}
in a restriction containing these elements and their derivatives.
Repeated differentiation is treated in the same way.

The inclusion \(R_W^\diamond\subset R_W\) induces
\(HH_0(R_W^\diamond)\hookrightarrow HH_0(R_W)\) by
\cref{lm:hh0-filtered-complete}.  For every continuous linear map \(L\),
\[
 L\!\left(\sum_j u_j\right)=\sum_j L(u_j).
\]
This applies to traces, cone maps, bulk pullbacks, endpoint restriction,
and interval integration.  Pairings commute with the sums by their
filtration bound \eqref{eq:ordinary-main-filtered}.
\end{rem}

\subsection{Interval operations}
\label{subsec:interval-data}

The pseudoisotopies of \cref{app:pseudoisotopy-gauge-invariance} live on
interval DGAs carrying a pairing and structure maps.  We axiomatize these
data below; \cref{rem:interval-de-rham-model} describes the de~Rham model
of \cite[Section~2.4]{solomon2019relative} and
\cite[Section~4.2]{kosloff2025pointlikenoncommutativef}.

\begin{dfn}[Interval DGAs]
\label{dfn:interval-dgas}
Interval DGAs are DGAs \(\widetilde C_X,\widetilde C_L\) and a
commutative DGA \(\mathfrak K\), together with maps
\[
 p^*\colon C_X^{Q_W}\to\widetilde C_X,
 \qquad
 j_r^*,\,p_*\colon\widetilde C_X\to C_X^{Q_W}
 \quad(r=0,1),
\]
and likewise between \(C_L^{R_W}\) and \(\widetilde C_L\) and between
\(\kk\) and \(\mathfrak K\), called pullback, endpoint restriction, and
integration.  As for \(C_X^{Q_W}\) and \(C_L^{R_W}\), the DGAs
\(\widetilde C_X\) and \(\widetilde C_L\) are completed scalar extensions
by \(Q_W\) and \(R_W\), with the induced filtrations, and their
differentials and the maps above are scalar extensions.
Pullback and endpoint restriction are unital cochain maps,
integration has degree \(-1\), and
\begin{equation}\label{eq:interval-stokes-data}
j_r^*p^*=\Id,\qquad p_*p^*=0,\qquad
dp_*+p_*d=j_1^*-j_0^*.
\end{equation}
The restriction \(i^*\) extends to the interval DGAs and commutes with
\(j_r^*,p^*,p_*\).

Give \(HH_0(R_W)\) the zero differential, and extend the maps from
\(\mathfrak K\) to \(\mathfrak K\widehat\otimes HH_0(R_W)\): for
homogeneous \(\kappa\in\mathfrak K\) and \(\xi\in HH_0(R_W)\),
\[
\begin{aligned}
 p^*\xi&=1\otimes\xi,&
 j_r^*(\kappa\otimes\xi)&=(j_r^*\kappa)\xi,\\
 d(\kappa\otimes\xi)&=d\kappa\otimes\xi,&
 p_*(\kappa\otimes\xi)&=(p_*\kappa)\xi,
\end{aligned}
\]
Then \eqref{eq:interval-stokes-data} holds there as well.
\end{dfn}

\begin{dfn}[Interval pairing]
\label{dfn:interval-pairing-data}
An interval pairing is a map
\[
\langle\!\langle\cdot,\cdot\rangle\!\rangle\colon
\widetilde C_L\otimes\widetilde C_L
\longrightarrow\mathfrak K\widehat\otimes HH_0(R_W)
\]
of degree \(-n\), with coefficient signs as in
\cref{dfn:scalar-extension-interval-pairing}, such that for homogeneous
\(a,b\in\widetilde C_L\) and \(x,y\in C_L^{R_W}\),
\begin{gather}
\langle\!\langle a,b\rangle\!\rangle
=-(-1)^{(|a|+1)(|b|+1)}\langle\!\langle b,a\rangle\!\rangle,
\label{eq:interval-pairing-skew}\\
j_r^*\langle\!\langle a,b\rangle\!\rangle
=\langle j_r^*a,j_r^*b\rangle_L,
\qquad
\langle\!\langle p^*x,p^*y\rangle\!\rangle
=p^*\langle x,y\rangle_L,
\label{eq:interval-pairing-endpoints}\\
p_*\langle\!\langle p^*x,b\rangle\!\rangle
=(-1)^{|x|+1}\langle x,p_*b\rangle_L.
\label{eq:interval-pairing-mixed}
\end{gather}
Coefficients \(z\in\mathcal Z(R_W)\) act on the pairing as in
\cite[Definition~2.11\textup{(6)}]{kosloff2025pointlikenoncommutativef},
and on the scalar target by
\[
 z\cdot(\kappa\otimes\xi)
 =(-1)^{|z||\kappa|}\kappa\otimes z\xi.
\]
For homogeneous \(z\) and \(a\), \(p_*(za)=(-1)^{|z|}zp_*a\).
\end{dfn}

\begin{dfn}[Interval operations]
\label{dfn:interval-structure-maps}
Interval operations are maps
\[
\begin{aligned}
 \widetilde\q_{k,l}&\colon
 \widetilde C_X^{\otimes l}\otimes\widetilde C_L^{\otimes k}
 \longrightarrow\widetilde C_L,\\
 \widetilde\q_{\emptyset,l}&\colon
 \widetilde C_X^{\otimes l}\longrightarrow\widetilde C_X,\\
 \widetilde\q_{-1,l}&\colon
 \widetilde C_X^{\otimes l}
 \longrightarrow\mathfrak K\widehat\otimes HH_0(R_W),
\end{aligned}
\]
of degrees \(2-k-2l\), \(4-2l\), and \(3-n-2l\), respectively, which are
Novikov sums of component maps \(\widetilde\q_{k,l}^{\beta}\),
\(\widetilde\q_{\emptyset,l}^{A}\), \(\widetilde\q_{-1,l}^{\beta}\) as in
\cref{dfn:novikov-summed-structure-maps}, such that:
\begin{enumerate}[label=\textup{(\roman*)},ref=\textup{(\roman*)}]
\item The component maps satisfy the axioms of \cref{dfn:structure-maps},
with the differentials, products, \(i^*\), and pairing of
\cref{dfn:interval-dgas,dfn:interval-pairing-data}, except for three
changes:
\begin{enumerate}[label=\textup{(\alph*)}]
\item For \((\beta,k,l)=(\beta_0,1,0)\), the cyclic identity
\ref{ax:qkl-cyclic-symmetry} becomes
\[
\langle\!\langle da,b\rangle\!\rangle
=(-1)^{(|a|+1)(|b|+1)}\langle\!\langle db,a\rangle\!\rangle
+d\langle\!\langle a,b\rangle\!\rangle.
\]
\item The right-hand side of the scalar equation
\eqref{eq:qminus-structure} has the additional term
\(d\widetilde\q_{-1,l}^{\beta}(\eta_1,\ldots,\eta_l)\).
\item The exceptions in \ref{ax:qkl-top-degree} also include
\((\beta,k,l)=(\beta_0,1,0)\).
\end{enumerate}
\item Endpoint restriction intertwines the interval and ordinary
operations:
\[
 j_r^*\widetilde\q_{k,l}(a_1,\ldots,a_k;\eta_1,\ldots,\eta_l)
 =\q_{k,l}(j_r^*a_1,\ldots,j_r^*a_k;
 j_r^*\eta_1,\ldots,j_r^*\eta_l),
\]
and likewise for \(\widetilde\q_{\emptyset,l}\) and
\(\widetilde\q_{-1,l}\).
\item\label{item:interval-filtrations}
Every finite set of homogeneous interval inputs lies in a coefficient
restriction as in \cref{dfn:adapted-coefficient-restriction}, with the
induced filtrations.  On it, the families \(\widetilde\q_{k,l}\) and
\(\langle\!\langle\widetilde\q_{k,l}(a_1,\ldots,a_k;\eta),
a_{k+1}\rangle\!\rangle\) are pseudo-complete.  If \(F\) is one of these
families, an interval operation, or the pairing, then
\[
 \varsigma_i\bigl(F(\xi)\bigr)\ge\varsigma_i^\otimes(\xi),
 \qquad i=1,2.
\]
\end{enumerate}
\end{dfn}

\begin{rem}[de~Rham model]
\label{rem:interval-de-rham-model}
In \cite[Section~2.4]{solomon2019relative} and
\cite[Section~4.2]{kosloff2025pointlikenoncommutativef}, the interval
DGAs are built from differential forms on \([0,1]\times M\), where
\(M\) is \(X\), \(L\), or a point:
\(p^*\) is pullback along the projection \(\pi\colon[0,1]\times M\to M\),
and \(j_r^*\) is restriction to the endpoint \(r\).  We take
\[
 p_*(\alpha+dt\wedge\beta)=\int_0^1\beta\,dt,
\]
where \(\alpha\) and \(\beta\) contain no \(dt\).  This is not the
push-forward \(\pi_*\) of \cite[Section~2.1]{solomon2019relative}: by
\cite[Proposition~2.1]{solomon2019relative}, for homogeneous \(\xi\),
\[
 p_*\xi=(-1)^{\dim M+|\xi|+1}\pi_*\xi.
\]
In terms of \(\pi_*\), Stokes' theorem
\cite[Proposition~2.2]{solomon2019relative} reads
\[
 d\pi_*\xi=\pi_*d\xi+(-1)^{\dim M+|\xi|+1}(j_1^*\xi-j_0^*\xi);
\]
multiplying by \((-1)^{\dim M+|\xi|+1}\) gives
\eqref{eq:interval-stokes-data}.  For example, the sign
\((-1)^{|a|+|b|+n}\) in
\cite[Lemma~4.10]{kosloff2025pointlikenoncommutativef} is this sign for
\(M\) a point and \(\xi=d\langle\!\langle a,b\rangle\!\rangle\).
For the pairing identities, see
\cite[Definitions~2.11\textup{(8)} and~4.6, and
Lemma~4.10]{kosloff2025pointlikenoncommutativef} and
\cite[Proposition~2.1\textup{(c)} and Lemma~4.6]{solomon2019relative};
for the axioms of \cref{dfn:interval-structure-maps}, see
\cite[Definitions~2.11\textup{(9)} and~4.6, and
Proposition~4.8]{kosloff2025pointlikenoncommutativef} and
\cite[Propositions~2.19--2.20 and Lemma~2.22]{solomon2019relative}.
The filtration conditions hold by \cref{rem:de-rham-sababa-restriction}.
\end{rem}

\begin{lm}\label{lm:interval-deformed-scalar-identities}
Let \(\widetilde\gamma\in\widetilde C_X^2\) and
\(\widetilde b\in\widetilde C_L^1\) satisfy
\[
 \nu(\widetilde\gamma)>0,\qquad
 \varsigma_1(\widetilde b)\ge0,\qquad
 \varsigma_2(\widetilde b)>0,\qquad
 d\widetilde\gamma=0.
\]
The insertion formulas of
\cref{dfn:bg-deformation-qkl,dfn:bg-deformation-qempty,dfn:bg-deformation-qminusone},
with the interval pairing in the scalar formula, define deformed interval
operations \(\widetilde\q_{k,l}^{\widetilde\gamma,\widetilde b}\),
\(\widetilde\q_{\emptyset,l}^{\widetilde\gamma}\), and
\(\widetilde\q_{-1,l}^{\widetilde\gamma,\widetilde b}\).  They satisfy
\cref{lm:qkl-bg-insertions} and the axioms of
\cref{dfn:interval-structure-maps} with \(\widetilde\gamma,\widetilde b\)
inserted.  Omitting the superscripts,
\begin{align}
d\widetilde\q_{-1,0}
+\tfrac12\langle\!\langle\widetilde\q_{0,0},
\widetilde\q_{0,0}\rangle\!\rangle
-\langle\!\langle1,i^*\widetilde\q_{\emptyset,0}\rangle\!\rangle
&=0,
\label{eq:interval-deformed-scalar-zero}\\
d\widetilde\q_{-1,1}(\eta)-\widetilde\q_{-1,1}(d\eta)
+\langle\!\langle\widetilde\q_{0,0},
\widetilde\q_{0,1}(\eta)\rangle\!\rangle
-\langle\!\langle1,i^*\widetilde\q_{\emptyset,1}(\eta)\rangle\!\rangle
&=0.
\label{eq:interval-deformed-scalar-one}
\end{align}
\end{lm}
\begin{proof}
The proofs of \cref{lm:qkl-bg-insertions,lm:bg-deformed-axioms} apply,
using \ref{item:interval-filtrations} of
\cref{dfn:interval-structure-maps}.

For the two identities, follow
\cite[Section~7.1]{kosloff2025pointlikenoncommutativef}.  For a real
scalar \(\rho\in[0,1]\), put
\[
\widetilde\q_{k,l}^{\rho}
:=\widetilde\q_{k,l}^{\widetilde\gamma,\rho\widetilde b}.
\]
Since \(\int_0^1\rho^k\,d\rho=1/(k+1)\), the scalar insertion formula
gives, for \(l=0,1\),
\[
\widetilde\q_{-1,l}^{1}(\eta)
=\widetilde\q_{-1,l}^{0}(\eta)
+\int_0^1
\langle\!\langle\widetilde\q_{0,l}^{\rho}(\eta),\widetilde b
\rangle\!\rangle\,d\rho,
\]
where \(\eta\) is omitted for \(l=0\).  Differentiating the insertion
formulas termwise gives
\[
\partial_\rho\widetilde\q_{0,0}^{\rho}
=\widetilde\q_{1,0}^{\rho}(\widetilde b),
\qquad
\partial_\rho\widetilde\q_{0,1}^{\rho}(\eta)
=\widetilde\q_{1,1}^{\rho}(\widetilde b;\eta).
\]
By \cref{lm:insertion-series}, termwise integration and differentiation
are justified as in
\cite[Lemmas~3.9 and~3.12]{kosloff2025pointlikenoncommutativef}.

For \(l=0\), the open structure equation gives
\(\widetilde\q_{1,0}^{\rho}(\widetilde\q_{0,0}^{\rho})=0\).  The cyclic
identity for \((k,l)=(1,0)\), applied to
\(\widetilde\q_{0,0}^{\rho}\) and \(\widetilde b\), then gives
\[
\begin{aligned}
d\langle\!\langle\widetilde\q_{0,0}^{\rho},\widetilde b\rangle\!\rangle
&=-\langle\!\langle\widetilde\q_{1,0}^{\rho}(\widetilde b),
\widetilde\q_{0,0}^{\rho}\rangle\!\rangle\\
&=-\tfrac12\partial_\rho
\langle\!\langle\widetilde\q_{0,0}^{\rho},
\widetilde\q_{0,0}^{\rho}\rangle\!\rangle.
\end{aligned}
\]
The last equality uses \(|\widetilde\q_{0,0}^{\rho}|=2\) and
\eqref{eq:interval-pairing-skew}; see also
\cite[Proposition~2.14]{kosloff2025pointlikenoncommutativef}.

For \(l=1\), the open structure equation is
\[
\widetilde\q_{1,0}^{\rho}(\widetilde\q_{0,1}^{\rho}(\eta))
+(-1)^{|\eta|}\widetilde\q_{1,1}^{\rho}
(\widetilde\q_{0,0}^{\rho};\eta)
=\widetilde\q_{0,1}^{\rho}(d\eta).
\]
Pairing with \(\widetilde b\) and using the cyclic identities for
\((k,l)=(1,0),(1,1)\) and \(|\widetilde b|=1\) gives
\[
\begin{aligned}
&d\langle\!\langle\widetilde\q_{0,1}^{\rho}(\eta),
\widetilde b\rangle\!\rangle
-\langle\!\langle\widetilde\q_{0,1}^{\rho}(d\eta),
\widetilde b\rangle\!\rangle\\
&\quad=-(-1)^{|\eta|}
\langle\!\langle\widetilde\q_{1,1}^{\rho}(\widetilde b;\eta),
\widetilde\q_{0,0}^{\rho}\rangle\!\rangle
-\langle\!\langle\widetilde\q_{1,0}^{\rho}(\widetilde b),
\widetilde\q_{0,1}^{\rho}(\eta)\rangle\!\rangle\\
&\quad=-\partial_\rho
\langle\!\langle\widetilde\q_{0,0}^{\rho},
\widetilde\q_{0,1}^{\rho}(\eta)\rangle\!\rangle.
\end{aligned}
\]
Here \eqref{eq:interval-pairing-skew} moves
\(\widetilde\q_{0,0}^{\rho}\), of degree two, to the first argument
with sign \((-1)^{|\eta|}\).

At \(\rho=0\) only \(\widetilde\gamma\) is inserted, and both identities
are the scalar equation of \cref{dfn:interval-structure-maps} with
\(\widetilde\gamma\) inserted; for \(l=1\), its two quadratic terms
combine by \eqref{eq:interval-pairing-skew}.  Integrating the two
calculations above over \(\rho\in[0,1]\) and substituting into the
formula for \(\widetilde\q_{-1,l}^{1}\) gives the identities at
\(\rho=1\).
\end{proof}

\section{Degree and gauge invariance of the tensor potential}
\label{app:tensor-potential-lemmas}

\begin{lm}\label{lm:tensor-potential-degree}
The map
\[
\nun^{\gamma,b}\colon\Cone{\mathfrak{i}}\to\Cone{\mathfrak{i}}
\]
is homogeneous of degree \(2\).
\end{lm}
\begin{proof}
Let \(\eta\) be homogeneous.  Since \(|\gamma|=2\) and \(|b|=1\),
the closed, paired, and scalar summands in
\cref{dfn:bg-deformation-qempty,dfn:bg-deformation-qminusone}
have degrees, respectively,
\[
\begin{aligned}
&(4-2m)+|\eta|+2(m-1)=|\eta|+2,\\
&(2-k-2m)+k+|\eta|+2(m-1)+1-n=|\eta|+1-n,\\
&(3-n-2m)+|\eta|+2(m-1)=|\eta|+1-n.
\end{aligned}
\]
The shift \(HH_0(R_W)[-n-1]\) therefore gives
\(|\nun^{\gamma,b}(\eta,0)|=|\eta|+2\).
On the second summand, the map \((0,\xi)\mapsto(0,-c\xi)\)
also has degree \(2\), since \(c\) is central and \(|c|=2\).
\end{proof}

\subsection{Pseudoisotopy and gauge invariance}
\label{app:pseudoisotopy-gauge-invariance}

\begin{dfn}[Gauge equivalence of bounding pairs]
\label{dfn:gauge-equivalence-bounding-pairs}
Two bounding pairs \((\gamma,b)\), \((\gamma',b')\) are
\emph{gauge equivalent} if there are interval data as in
\cref{dfn:interval-dgas,dfn:interval-pairing-data,dfn:interval-structure-maps} and
\[
\widetilde\gamma\in(\mI_{Q_W}\widetilde C_X)^2,
\qquad
\widetilde b\in(((s)+\mI_{R_W})\widetilde C_L)^1,
\qquad d\widetilde\gamma=0,
\]
with convergent insertion series, endpoints
\[
 (j_0^*\widetilde\gamma,j_0^*\widetilde b)=(\gamma,b),
 \qquad
 (j_1^*\widetilde\gamma,j_1^*\widetilde b)=(\gamma',b'),
\]
and
\begin{equation}\label{eq:gauge-equivalence-interval-mc}
 \widetilde\q_{0,0}^{\widetilde\gamma,\widetilde b}
 =c\cdot1_L,
 \qquad
 c\in\bigl(\mI_{R_W}\cap\mathcal Z(R_W)\bigr)^2,
\end{equation}
where \(c\) is constant.  We call \((\widetilde\gamma,\widetilde b)\)
a \emph{pseudoisotopy} between the two pairs.
For the de~Rham complex of the interval, the right-hand side of
\eqref{eq:gauge-equivalence-interval-mc} has no \(dt\)-component, as in
\cite[Definition~2.23]{solomon2019relative}; compare
\cite[Definition~4.20]{kosloff2025pointlikenoncommutativef}.
\end{dfn}

\begin{rem}
\label{rem:gauge-equivalence-no-dtau}
In the de~Rham model, \cite[Definition~4.20]{kosloff2025pointlikenoncommutativef}
allows a closed coefficient \(c+dt\,h_t\) in
\eqref{eq:gauge-equivalence-interval-mc}.  Our bulk variables have even
degree, so \(h_t\in\mathcal Z(R_W)^{\mathrm{odd}}=0\).
The forms \(\widetilde\gamma\) and \(\widetilde b\) may still have
\(dt\)-components.
\end{rem}

\begin{thm}\label{thm:relative-potential-gauge-invariance}
If $(\gamma, b)$ is gauge equivalent to
\(\left(\gamma^{\prime}, b^{\prime}\right)\), then
\(\Psi(\gamma, b)=\Psi\left(\gamma^{\prime}, b^{\prime}\right)\).
\end{thm}
\begin{proof}
We follow \cite[proof of Theorem~1]{solomon2019relative}.
Choose a pseudoisotopy \((\widetilde\gamma,\widetilde b)\) and omit its
superscripts from the deformed operations.  By
\cref{lm:interval-deformed-scalar-identities},
\(d\widetilde\q_{\emptyset,0}=0\).
Stokes' identity~\eqref{eq:interval-stokes-data} gives
\[
\q_{\emptyset,0}^{\gamma'}-\q_{\emptyset,0}^{\gamma}
=(j_1^*-j_0^*)\widetilde\q_{\emptyset,0}
=d_Xp_*\widetilde\q_{\emptyset,0}.
\]
For the scalar component,
\eqref{eq:interval-deformed-scalar-zero} gives
\[
d\widetilde\q_{-1,0}
=\langle\!\langle1,i^*\widetilde\q_{\emptyset,0}\rangle\!\rangle
-\tfrac12\langle\!\langle c1,c1\rangle\!\rangle.
\]
Since \(n>1\),
\[
\langle\!\langle c1,c1\rangle\!\rangle
=c^2\langle\!\langle1,1\rangle\!\rangle
=c^2p^*\langle1,1\rangle_L=0.
\]
By \eqref{eq:interval-stokes-data}, \eqref{eq:interval-pairing-mixed},
and \(i^*p_*=p_*i^*\),
\[
\begin{aligned}
-\q_{-1,0}^{\gamma',b'}+\q_{-1,0}^{\gamma,b}
&=-p_*d\widetilde\q_{-1,0}\\
&=-p_*\langle\!\langle1,i^*\widetilde\q_{\emptyset,0}\rangle\!\rangle\\
&=\langle1,i^*p_*\widetilde\q_{\emptyset,0}\rangle_L
=\mathfrak i(p_*\widetilde\q_{\emptyset,0}).
\end{aligned}
\]
Thus
\[
\psi(\gamma',b')-\psi(\gamma,b)
=d_{\Cone}\bigl(p_*\widetilde\q_{\emptyset,0},0\bigr).
\]
Hence \(\Psi(\gamma',b')=\Psi(\gamma,b)\).
\end{proof}

\begin{thm}[Gauge invariance of the tensor potential]
\label{thm:tensor-potential-gauge-invariance}
If \((\gamma,b)\) and \((\gamma',b')\) are gauge equivalent, then
\(\nun^{\gamma,b}\) and \(\nun^{\gamma',b'}\) are chain homotopic:
there is a degree \(1\) map
\(H\colon\Cone{\mathfrak i}\to\Cone{\mathfrak i}\) such that
\[
\nun^{\gamma',b'}-\nun^{\gamma,b}
=d_{\Cone}H+Hd_{\Cone}.
\]
Consequently, the tensor potential induced on
\(H^*(\Cone{\mathfrak i})\) is invariant under gauge equivalence.
\end{thm}
\begin{proof}
Following \cite[Lemma~4.7]{solomon2019relative}, choose a pseudoisotopy
and omit its superscripts.  Put
\[
H(\eta,\xi)
:=\left(
p_*\widetilde\q_{\emptyset,1}(p^*\eta),
-p_*\widetilde\q_{-1,1}(p^*\eta)
\right).
\]
By \cref{lm:tensor-potential-degree} and \(|p_*|=-1\), we have \(|H|=1\).

The deformed closed chain-map identity of
\cref{lm:interval-deformed-scalar-identities} gives
\(d\widetilde\q_{\emptyset,1}(p^*\eta)
=\widetilde\q_{\emptyset,1}(p^*d_X\eta)\).
Thus \eqref{eq:interval-stokes-data} gives
\[
\begin{aligned}
&d_Xp_*\widetilde\q_{\emptyset,1}(p^*\eta)
+p_*\widetilde\q_{\emptyset,1}(p^*d_X\eta)\\
&\qquad=(j_1^*-j_0^*)\widetilde\q_{\emptyset,1}(p^*\eta)
=\q_{\emptyset,1}^{\gamma'}(\eta)-\q_{\emptyset,1}^{\gamma}(\eta).
\end{aligned}
\]
The top-degree and zero-energy axioms give
\[
\langle\!\langle1,\widetilde\q_{0,1}(p^*\eta)\rangle\!\rangle
=-\langle\!\langle1,i^*p^*\eta\rangle\!\rangle,
\]
since every summand has an interior input and only
\((\beta,k,m)=(\beta_0,0,1)\) survives; see
\cite[Lemma~2.21]{solomon2019relative}.
With \(\widetilde\q_{0,0}=c1\),
\eqref{eq:interval-deformed-scalar-one} becomes
\[
\begin{aligned}
d\widetilde\q_{-1,1}(p^*\eta)
-\widetilde\q_{-1,1}(p^*d_X\eta)
={}&c\langle\!\langle1,i^*p^*\eta\rangle\!\rangle
+\langle\!\langle1,i^*\widetilde\q_{\emptyset,1}(p^*\eta)
\rangle\!\rangle.
\end{aligned}
\]
Since \(c\) is constant and \(|c|=2\),
\[
p_*\bigl(c\langle\!\langle1,i^*p^*\eta\rangle\!\rangle\bigr)
=cp_*p^*\langle1,i^*\eta\rangle_L=0.
\]
Applying \(p_*\) and using \eqref{eq:interval-stokes-data}
and \eqref{eq:interval-pairing-mixed}, we get
\[
\begin{aligned}
&-\q_{-1,1}^{\gamma',b'}(\eta)+\q_{-1,1}^{\gamma,b}(\eta)\\
&\qquad=
-p_*\widetilde\q_{-1,1}(p^*d_X\eta)
-p_*\langle\!\langle1,i^*\widetilde\q_{\emptyset,1}(p^*\eta)
\rangle\!\rangle\\
&\qquad=
\mathfrak i\bigl(p_*\widetilde\q_{\emptyset,1}(p^*\eta)\bigr)
-p_*\widetilde\q_{-1,1}(p^*d_X\eta).
\end{aligned}
\]
These are the two components of \((d_{\Cone}H+Hd_{\Cone})(\eta,\xi)\)
for the cone in \cref{rem:conventions-st}.
The terms \(-c\xi\) cancel because both pairs have the same constant
\(c\).  This proves the homotopy identity.
\end{proof}

\section{Closedness and flatness}
\label{app:standard-arguments}

\subsection{Closedness of the relative potential}

By \cref{lm:bg-deformed-axioms}, our operations satisfy the identities
used in \cite[Lemma~4.8]{solomon2019relative}: the deformed \(l=0\)
scalar equation, the Maurer--Cartan equation, and the unit and top-degree
properties.  Its proof applies in \(HH_0(R_W)\), since the only scalar
multiplication is by the central coefficient \(c\).
The chain isomorphism \(F_n\) of \cref{rem:conventions-st} then gives
\cref{lm:relative-potential-closed}.

\subsection{Flatness}

The proof of \cite[Theorem~5]{solomon2019relative} builds the homotopy
from deformed geodesic operators and uses their structure equations
\cite[Lemma~3.17]{solomon2019relative}.  The geodesic operators extend to
our coefficients by the rules \textup{(E1)}--\textup{(E4)} of
\cref{dfn:scalar-extension-structure-maps}, with one change: their degree
is one higher than that of the ordinary operators, so in
\textup{(E2)}--\textup{(E3)} moving the coefficient \(b_i\) of a closed
input out of the operator gives an extra sign \((-1)^{|b_i|}\).
Once their classes are included in \(\Pi_{\mathrm{op}}\), the proofs of
\cref{lm:qkl-bg-insertions,lm:bg-deformed-axioms} apply to them unchanged.
This gives the deformed geodesic operators and their structure equations.

Each \(\delta\in\mathfrak g_{\mathrm{for}}\) is \(\Lambda\)-linear,
since it is a derivation vanishing on \(\Lambda\).  Hence, when \(\delta\)
is applied to a value of a deformed operator, it passes through the
Novikov weights and acts either on an input or on one copy of \(\gamma\)
or \(b\).  The convention \eqref{eq:operator-derivative-convention}
removes the terms in which \(\delta\) acts on an input.  Thus the
derivative of a deformed operator is the sum of the terms in which one
copy of \(\gamma\) or \(b\) is replaced by \(\delta\gamma\) or
\(\delta b\).  The proof in \cite{solomon2019relative} uses only that
\(\partial_u\) is a \(\Lambda\)-linear derivation, so we may replace
\(\partial_u\gamma,\partial_u b\) by \(\delta\gamma,\delta b\), and
similarly for \(\epsilon\).

The scalar geodesic operations and the pairing \(\langle\cdot,\cdot\rangle_L\)
take values in \(HH_0(R_W)\), where \([ab]=(-1)^{|a||b|}[ba]\).
Hence the symmetry of the pairing and the cyclic symmetry of the
operations used in the proof still hold with coefficients in \(R_W\):
a cyclic rotation of the inputs rotates the product of their
coefficients, which does not change its class.  The proof multiplies
these values only by \(c,\delta c,\epsilon c\), which lie in
\(\mathcal Z(R_W)\) by \cref{lm:mc-coefficient-actions-coker}, so these
products are well defined on \(HH_0(R_W)\).

After scalar extension, conjugation by \(F_n\) gives the homotopy on
our cone; see \cref{rem:conventions-st}.  In the notation of
\cite{solomon2019relative}, the homotopy that replaces \(H_{uv}\) in the
proof of \cite[Theorem~5]{solomon2019relative} is
\[
\begin{aligned}
H_{\delta,\epsilon}(\eta,\xi)
=\Bigg(&
(-1)^{1+|\eta|+|\delta|+|\epsilon|+|\eta||\epsilon|}
\q_{\emptyset,3;\chi_0}^{\gamma}
(\delta\gamma,\eta,\epsilon\gamma),
\\
&
(-1)^{n+|\eta|+|\delta|+|\epsilon|+|\eta||\epsilon|}
\q_{-1,3;0,3}^{\gamma,b}
(\delta\gamma,\eta,\epsilon\gamma)
\\
&\quad+
(-1)^{n+(|\eta|+1)(|\delta|+|\epsilon|+1)}
\left\langle
\q_{0,2;1_0,2}^{\gamma,b}(\eta,\delta\gamma),
\epsilon b
\right\rangle_L
\\
&\quad+
(-1)^{n+|\eta|+(|\eta|+1)(|\delta|+|\epsilon|)
               +|\delta||\epsilon|}
\left\langle
\q_{0,2;1_0,2}^{\gamma,b}(\eta,\epsilon\gamma),
\delta b
\right\rangle_L
\\
&\quad+
(-1)^{n+|\eta|+(|\eta|+1)(|\delta|+|\epsilon|)
               +|\delta||\epsilon|}
\left\langle
\q_{1,1;2_{0,1},1}^{\gamma,b}(\epsilon b;\eta),
\delta b
\right\rangle_L
\Bigg).
\end{aligned}
\]
The factors \((-1)^n\) in the second component come from \(F_n\).
Finally,
\cref{app:derivation-continuity,lm:qkl-bg-insertions,lm:bg-deformed-axioms}
give continuity and convergence, so the coefficientwise calculation
extends to the completions.  This proves \cref{thm:flatness-relation}.

\bibliographystyle{amsabbrvcnobysame}
\bibliography{refs}
\end{document}